\documentclass[12pt]{amsart}
\usepackage{amscd,amssymb,graphicx,color,a4wide,hyperref,cite,amsthm,amsmath}
\usepackage{listings}
\usepackage{color}

\newtheoremstyle{cited}%
  {3pt}
  {3pt}
  {\itshape}
  {}
  {\bfseries}
  {.}
  {.5em}
  {\thmname{#1} \thmnumber{#2} \thmnote{\normalfont#3}}

\theoremstyle{cited}

\definecolor{dkgreen}{rgb}{0,0.6,0}
\definecolor{gray}{rgb}{0.5,0.5,0.5}
\definecolor{mauve}{rgb}{0.58,0,0.82}

\usepackage{hyperref}
\usepackage{url}

\usepackage{mathrsfs}

\usepackage{epstopdf}

\usepackage[all]{xy}
\newtheorem{theorem}{Theorem}[section]
\newtheorem{lemma}[theorem]{Lemma}
\newtheorem{proposition}[theorem]{Proposition}

\theoremstyle{definition}
\newtheorem{definition}[theorem]{Definition}

\theoremstyle{remark}
\newtheorem{remark}[theorem]{Remark}

\numberwithin{equation}{section}
\numberwithin{figure}{section}

\newcommand{\ZZ} {\mathbb{Z}}

\newcommand{\RR} {\mathbb{R}}

\newcommand\restr[2]{{
  \left.\kern-\nulldelimiterspace 
  #1 
  \vphantom{\big|} 
  \right|_{#2} 
  }}

\usepackage{tikz}
\usepackage{MnSymbol}
\usetikzlibrary{matrix,arrows,decorations.pathmorphing}
\usepackage{hyperref}
\usepackage{enumerate}
\usepackage{float}
\usepackage{tikz-cd}
\usepackage{relsize}
\usepackage{mathtools}

\newcommand\C{\mathbb C}

\newcommand\NN{\mathbb N}

\newcommand\R{\mathbb R}
\newcommand\Q{\mathbb Q}

\newcommand\Z{\mathbb Z}

\newcommand\cJ{\mathcal J}

\newcommand\cM{\mathcal M}
\newcommand\cN{\mathcal N}

\newcommand\cT{\mathcal T}

\newcommand\fp{\mathfrak p}
\newcommand\fd{\mathfrak d}
\newcommand\fe{\mathfrak e}
\newcommand\fD{\mathfrak D}
\newcommand\fg{\mathfrak g}

\newcommand\fh{\mathfrak h}
\newcommand\fj{\mathfrak j}

\newcommand\fq{\mathfrak q}

\newcommand\fS{\mathfrak S}

\newcommand\Aut{\operatorname{Aut}}

\newcommand\Hom{\operatorname{Hom}}

\newcommand\Tr{\operatorname{Tr}}

\newcommand\sgn{\operatorname{sgn}}
\newcommand\Supp{\operatorname{Supp}}
\newcommand\Wall{\operatorname{Wall}}
\newcommand\mult{\operatorname{mult}}

\begin{document}

\title[The flow tree formula]{The flow tree formula for Donaldson-Thomas 
invariants of quivers with potentials}
\author[H.\,Arg\"uz]{H\"ulya Arg\"uz}
\address{University of Georgia, Department of Mathematics, Athens, GA 30605}
\email{Hulya.Arguz@uga.edu}

\author[P.\,Bousseau]{Pierrick Bousseau}
\address{University of Georgia, Department of Mathematics, Athens, GA 30605}
\email{Pierrick.Bousseau@uga.edu}

\date{\today}

\begin{abstract}
We prove the flow tree formula conjectured by Alexandrov and Pioline which computes Donaldson-Thomas invariants of quivers with potentials in terms of a smaller set of  attractor invariants.
This result is obtained as a particular case of a more general flow tree formula reconstructing a consistent scattering diagram from its initial walls.
\end{abstract}
\maketitle

\tableofcontents

\section{Introduction}
\label{section_introduction}

Donaldson-Thomas (DT) theory is a topic at the intersection of algebraic geometry, symplectic geometry, representation theory and theoretical physics. 
Given a triangulated category $C$ which is Calabi-Yau of dimension $3$
(CY3) together with a choice of Bridgeland stability condition $\theta$ \cite{MR2373143}, DT invariants are defined by virtually counting $\theta$-semistable objects 
in $C$ \cite{donaldson1998gauge, MR1818182, JoyceSong, kontsevich2008stability}.
In quantum field theory and string theory, they play an important role as counts of BPS states and D-branes \cite{MR2567952}.

Quivers with potentials \cite{MR2480710} provide a natural source of examples of CY3 categories coming from representation theory \cite{ginzburg2006calabi-yau, MR2484733}.
Due to its more algebraic nature, DT theory of quivers with potentials is an ideal setting to study and explore many questions which are also of interest in the geometric incarnations of DT theory given by counts of semistable objects in the derived category of coherent sheaves on Calabi-Yau 3-folds \cite{MR1818182} and 
by counts of special Lagrangian submanifolds in Calabi-Yau 3-folds \cite{MR1957663, MR1941627}.

A key phenomenon in DT theory is wall-crossing in the space of stability conditions: DT invariants are constant in the complement of
countably many real codimension one loci in the space of stability conditions called walls, but they
jump discontinuously in general when the stability condition crosses a wall. The precise description
of this jumping behaviour of DT invariants across walls in the space of stability conditions is given by the wall-crossing formula of Joyce-Song and 
Kontsevich-Soibelman \cite{JoyceSong,kontsevich2008stability}, which is a universal
algebraic expression that contains some amount of combinatorial complexity.

By successive applications of the wall-crossing formula, one can show that the DT invariants of a quiver with potential are determined by a much smaller subset of \emph{attractor DT invariants} defined by picking particular stability conditions \cite{MR3330788, AlexandrovPioline}. In \cite{AlexandrovPioline}, Alexandrov-Pioline
conjectured, based on string-theoretic predictions, a new formula that expresses DT invariants in terms of the attractor DT invariants as a sum over trees, called the \emph{flow tree formula}. Their conjecture reduces the general wall-crossing
formula to an iterative application of the much simpler primitive wall-crossing formula. The main result of the present paper is a proof of the flow tree formula. 
In fact, we prove a version of the flow tree formula in the more general context of consistent scattering diagrams.

The flow tree formula is a new tool to unravel some of the deep and hidden structures in DT theory.  For example, versions of the flow tree formula are a major tool in the recent formulation of the conjectural proposal of
\cite{MR4072224} (see also 
\cite{MR4170291}) for the construction of modular completions for generating series of DT invariants counting coherent sheaves supported on surfaces inside Calabi-Yau 3-folds.

\subsection{Background}
\label{section_backgound_intro}
A quiver with potential $(Q,W)$ is given by a finite oriented graph $Q$, and a finite formal linear combination $W$ of oriented cycles in $Q$. 
We assume that $Q$ does not contain oriented $2$-cycles, and we denote by 
$Q_0$ the set of vertices of $Q$.
For every dimension vector $\gamma \in N \coloneqq \Z^{Q_0}$
and stability parameter
\begin{equation} \theta \in \gamma^{\perp}  \subset M_\R \coloneqq \Hom(N,\R)\,,\end{equation} 
where $\gamma^{\perp} \coloneqq \{\theta\in M_\R|\,\theta(\gamma)=0\}$, 
the theory of King's stability for quiver representations
\cite{MR1315461} defines a quasiprojective variety
$M_\gamma^{\theta}$, parametrizing S-equivalence classes of 
$\theta$-semistable representations of $Q$ of dimension $\gamma$, and a regular function 
\begin{equation} \Tr (W)_\gamma^\theta \colon M_\gamma^\theta \longrightarrow \C\,.
\end{equation}
Assuming that $\theta$ is $\gamma$-generic in the sense that $\theta(\gamma')=0$ implies $\gamma'$ collinear with $\gamma$, the Donaldson-Thomas (DT) invariant $\Omega_\gamma^{\theta}$ is an integer which is a virtual count of the critical points of $\Tr (W)_\gamma^\theta$. 
Applying Hodge theory to the sheaf of vanishing cycles of  $\Tr (W)_\gamma^\theta$, the integer 
$\Omega_\gamma^\theta$ can be refined into a Laurent polynomial $\Omega_\gamma^\theta(y,t)$ in two variables $y$ and $t$ and with integer coefficients, referred to as \emph{refined DT invariants} \cite{JoyceSong, kontsevich2008stability, MR2801406, MR2650811, MR4000572, davison2015donaldson, MR4132957}.
It is often convenient to use the rational functions $\overline{\Omega}_\gamma^\theta(y,t) \in \Q(y,t)$
defined as in \cite{JoyceSong, kontsevich2008stability, MR2875965} by 
\begin{equation}
    \overline{\Omega}_{\gamma}^\theta(y,t) 
    \coloneqq \sum_{\substack{\gamma' \in N\\ 
    \gamma=k \gamma',\, k\in \Z_{\geq 1}}} \frac{1}{k} \frac{y-y^{-1}}{y^k - y^{-k}} \Omega_{\gamma'}^{\theta}(y^k,t^k) \,,
\end{equation}
and referred to as \emph{rational DT invariants}. 

The DT invariants $\Omega_\gamma^\theta(y,t)$ are locally constant functions of the $\gamma$-generic stability parameter $\theta \in \gamma^{\perp}$ and their jumps across the loci of non-$\gamma$-generic stability parameters are given by the wall-crossing formula of Joyce-Song and 
Kontsevich-Soibelman \cite{JoyceSong,kontsevich2008stability}. 
Using the wall-crossing formula, the DT invariants can be computed in terms of the simpler \emph{attractor DT invariants}, which are 
DT invariants at specific values of the stability parameter.

Let $\langle -,-\rangle \colon N \times N \rightarrow \Z$ be the skew-symmetric form given by
\begin{equation}
\label{Eq: Euler form}
 \langle \gamma, \gamma' 
\rangle  = \sum_{i,j\in Q_0}(a_{ij}-a_{ji})\gamma_i\gamma_j'\,,   
\end{equation}
where $a_{ij}$ is the number of arrows in $Q$ from the vertex $i$ to the vertex $j$. The specific point $\langle \gamma,- \rangle \in \gamma^{\perp} \subset M_{\RR}$ is called the \emph{attractor point} for $\gamma$ \cite{AlexandrovPioline, mozgovoy2020attractor}. In general, the attractor point 
$\langle \gamma,-\rangle$ is not $\gamma$-generic and we define
the attractor DT invariants $\Omega_\gamma^{*}(y,t)$ by
\begin{equation} \Omega_\gamma^{*}(y,t)
\coloneqq \Omega_\gamma^{\theta_\gamma}(y,t)\,,\end{equation} 
where $\theta_\gamma$
is a small $\gamma$-generic perturbation of 
$\langle \gamma,- \rangle$ in $\gamma^{\perp}$ \cite{AlexandrovPioline, mozgovoy2020attractor}. One can check that $\Omega_\gamma^{*}(y,t)$ is independent of the choice of the small perturbation
\cite{AlexandrovPioline, mozgovoy2020attractor}.

For an acyclic quiver $Q$ (and so $W=0$), or more generally for a quiver $Q$ with a non-degenerate potential $W$
admitting a green-to-red sequence \cite{mou2019scattering}, the attractor DT invariants are as simple as possible: 

\begin{equation} \Omega_\gamma^{*}(y,t)= \begin{cases} 
     1 & \mathrm{if} \gamma=(\delta_{ij})_{i\in Q_0} \mathrm{~for~ some~} j \in Q_0 \\
     0 & \mathrm{elsewise,}
   \end{cases}
\end{equation}
where $\delta_{ij}$ is the Kronecker delta.
Similarly, for a quiver with potential $(Q,W)$ describing the derived category of coherent sheaves on a local del Pezzo surface, it is recently conjectured \cite{beaujard2020vafa, mozgovoy2020attractor}
that $\Omega_\gamma^{*}(y,t)=0$ unless
$\gamma=(\delta_{ij})_{i\in Q_0}$
for some $j \in Q_0$ or unless $\gamma$ is the class of the skyscraper sheaf of a point. However, for quivers with potential $(Q,W)$ describing interesting parts of the derived category of coherent sheaves on a compact Calabi-Yau 3-fold, the attractor DT invariants are expected to be non-vanishing and to typically exhibit an exponential growth. We refer to \cite{MR2913216, MR3036499, MR2967676, MR3033854, MR3036440} for some explicit examples
involving $n$-gon quivers.

The rational DT invariants $\overline{\Omega}_\gamma^{\theta}(y,t)$
for general $\gamma$-generic stability parameters $\theta \in \gamma^{\perp}$
are expressed in terms of the rational attractor DT invariants $\overline{\Omega}_\gamma^{*}(y,t)$ by a formula of the form
\begin{equation} \label{eq_reconstruction_intro}
    \overline{\Omega}_\gamma^\theta(y,t) = \sum_{r\geq 1} \sum_{\substack{\{\gamma_i\}_{1\leq i\leq r}\\ \sum_{i=1}^r \gamma_i = \gamma}} \frac{1}{|\Aut(\{\gamma_i\}_i)|} 
    F_r^{\theta}(\gamma_1,\dots,\gamma_r) \prod_{i=1}^r \overline{\Omega}_{\gamma_i}^{*}(y,t)\,,
\end{equation}
where the second sum is over the multisets $\{\gamma_i\}_{1\leq i\leq r}$ with $\gamma_i \in N$ and $\sum_{i=1}^r \gamma_i=\gamma$. Here, the denominator $|\Aut(\{\gamma_i\}_i)|$ is the order of the symmetry group of $\{\gamma_i\}$: if $m_{\gamma'}$ is the number of times that $\gamma' \in N$ appears in $\{\gamma_i\}_i$, then 
$|\Aut(\{\gamma_i\}_i)|=\prod_{\gamma'\in N}m_{\gamma'}!$.
The coefficients $F_r^{\theta}(\gamma_1,\dots,\gamma_r)$ 
are element of $\Q(y)$ and are universal in the sense that they depend 
on $(Q,W)$ only through the skew-symmetric form 
$\langle -,-\rangle$ on $N$. Our main result is the proof of an explicit formula, called the 
\emph{flow tree formula} and conjectured by 
Alexandrov-Pioline
 \cite{AlexandrovPioline}, which computes the coefficients
$F_r^{\theta}(\gamma_1,\dots,\gamma_r)$
in \eqref{eq_reconstruction_intro} combinatorially in terms of a sum over
binary rooted 
trees, and where the contribution of each tree is computed following the flow
on the tree starting at the root and ending at the leaves.

\subsection{Main result: the flow tree formula}
\label{section_main_result_intro}
We introduce some notations which are necessary to state precisely the 
flow tree formula in Theorem \ref{main_thm_intro} below.
We fix $\gamma \in N$, a 
$\gamma$-generic stability parameter 
$\theta \in \gamma^{\perp}$, and $\gamma_1,\dots,\gamma_r \in N$
such that $\sum_{i=1}^r \gamma_i=\gamma$.

An essential ingredient in the formulation of the flow tree
formula for $F_r^\theta(\gamma_1,\dots,\gamma_r)$ is 
the choice of a generic skew-symmetric perturbation 
$(\omega_{ij})_{1\leq i,j\leq r}$ of the skew-symmetric
matrix 
$(\langle \gamma_i,\gamma_j\rangle)_{1\leq i,j \leq r}$.
The matrix $(\omega_{ij})_{1\leq i,j\leq r}$ cannot be viewed in general as a skew-symmetric 
bilinear form on the sublattice of $N$ generated by $\gamma_1,\dots,\gamma_r$ 
because $\gamma_1,\dots,\gamma_r$ are not necessarily linearly independent in $N$. Nevertheless, the matrix 
$(\omega_{ij})_{1\leq i,j \leq r}$
can always be interpreted as a skew-symmetric bilinear form $\omega$ on a rank $r$ free abelian group
$\cN \coloneqq \bigoplus_{i=1}^r \Z e_i$ with a basis $\{e_i\}_{1\leq i\leq r}$
and such that $\omega_{ij}=\omega(e_i,e_j)$. 
From this point of view, there is a natural additive map 
\begin{align}
    p \colon \mathcal{N} \longrightarrow N \\
    e_i \longmapsto \gamma_i \,, \nonumber
\end{align}
which enables us to define a skew-symmetric bilinear form $\eta$ on $\cN$ as being the pullback of $\langle-,-\rangle$ on $N$, that is, 
$\eta(e_i,e_j) \coloneqq \langle \gamma_i,\gamma_j\rangle$,
and we consider a real-valued skew-symmetric form $\omega$ on 
$\cN$ obtained as a small enough generic perturbation of $\eta$.
Let $\cM_\R \coloneqq \Hom(\cN,\R)$ and 
$q \colon M_\R \rightarrow \cM_\R$ be the map induced from 
$p \colon \cN \rightarrow N$ by duality. We denote by 
\begin{equation} \label{Eq: alpha_intro}
\alpha \coloneqq q(\theta)
\end{equation}
the image in $\cM_\R$ of the stability parameter $\theta \in M_\R$ by the map $q$.

The flow tree formula in Theorem \ref{main_thm_intro} below takes the form of a sum over 
trees. More precisely, we consider rooted trees which apart from the root vertex have $r$ univalent vertices, or leaves, decorated by the basis elements $e_1,\dots,e_r$ of $\cN$.
For such a tree $T$, we denote by
$V_T^{\circ}$ the set of interior, that is, non-univalent, vertices.
We endow each such tree with the flow from the root to the leaves.
Given a vertex $v$ in a tree, the vertex adjacent to $v$ coming before $v$ along the flow is referred to as the \emph{parent} of $v$ and denoted by $p(v)$, and the vertices adjacent to $v$ and coming after $v$ along the flow are referred to as the \emph{children} of $v$, as illustrated in Figure \ref{Fig:Atree}. 
Any vertex that comes after $v$ along the flow is a \emph{descendent} of $v$.
Let $\cT_r$ be the set of such trees which are \emph{binary}, that is such that
each interior vertex $v$ of a tree $T \in \cT_r$ has exactly two children.
For every tree $T \in \cT_r$ and $v$ a vertex of $T$, we define $e_v \in \mathcal{N}$ as the sum of all elements that appear as decorations on the leaves which are descendent of a vertex $v$. We denote by $\cT_r^{\eta}$ the set of trees $T \in \cT_r$
such that $\eta(e_{v'},e_{v''}) \neq 0$ 
where $v$ is the child of the root and 
$v',v''$ are the children of $v$.

\begin{figure}
\center{\scalebox{.6}{\input{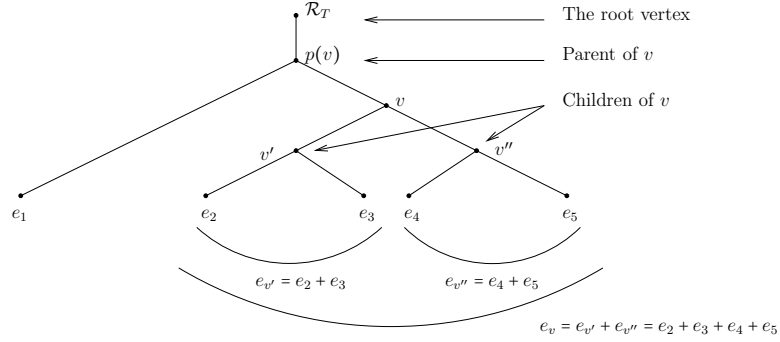}}}
\caption{A binary tree $T$ with five leaves for $\mathcal{N}=\ZZ e_1 \oplus \ZZ e_2 \oplus \ZZ e_3 \oplus \ZZ e_4 \oplus \ZZ e_5$.}
\label{Fig:Atree}
\end{figure}

For every tree $T \in \cT_r^\eta$ and $v$ a vertex of $T$ distinct from the leaves, we define $\theta_{T,v}^{\alpha,\omega} \in \cM_\R$
recursively as follows: 
If $v$ is the root vertex, then set $\theta_{T,v}^{\alpha,\omega}\coloneqq \alpha$. If $v$ is not the root, let $p(v)$ be the parent of $v$, and for any of the children, say $v'$ of $v$,
and $\iota_{e_v} \omega \coloneqq \omega(e_v,-) \in \cM_{\R}$, define
    \begin{equation} \label{eq_discrete_flow_intro}
        \theta_{T,v}^{\alpha,\omega} \coloneqq  
        \theta_{T,p(v)}^{\alpha,\omega} 
- \frac{\theta_{T,p(v)}^{\alpha,\omega}(e_{v'})}{
\omega(e_v,e_{v'})} \iota_{e_v} \omega \,.
\end{equation}
We show in Lemma \ref{lem_attractor} that this definition is independent of the choice of the child $v'$ of $v$.
Following  \cite{AlexandrovPioline}, we call 
$v \mapsto \theta_{T,v}^{\alpha,\omega}$ the 
\emph{discrete attractor flow}.

For every tree $T \in \cT_r^\eta$ and interior vertex 
$v \in V_T^\circ$, we
fix a labeling $v'$ and $v''$ of the two children of $v$, and we define 
\begin{equation} \label{eq_epsilon_intro}
    \epsilon_{T,v}^{\alpha, \omega} \coloneqq - \frac{\sgn 
   (\theta_{T,p(v)}^{\alpha,\omega}(e_{v'}))
   + \sgn ( \omega(e_{v'}, e_{v''}))}{2}  \in \{0,1,-1\} \,,
\end{equation}
where $\sgn(x) \in \{\pm 1\}$ is the sign of $x \in \R-\{0\}$.
We show in \S \ref{section_trees_flows} that for generic $\omega \in \bigwedge^2\mathcal{M}_{\RR}$, we have $\theta_{T,p(v)}^{\alpha,\omega}(e_{v'}) \neq 0$
and $\omega(e_{v'}, e_{v''}) \neq 0$ and so the definition of $\epsilon_{T,v}^{\alpha,\omega}$ indeed makes sense. Our main result is the following \emph{flow tree formula}, conjectured in \cite{AlexandrovPioline}, which enables us to determine the coefficients $F_r^{\theta}(\gamma_1,\dots,\gamma_r)$ in \eqref{eq_reconstruction_intro} expressing the DT invariants $\overline{\Omega}_\gamma^{\theta}(y,t)$ in terms of the attractor DT invariants 
$\overline{\Omega}_{\gamma_i}^{*}(y,t)$.
\begin{theorem} \label{main_thm_intro}
For every choice a small enough generic perturbation $\omega \in \bigwedge^2 \cM_\R$ of the skew-symmetric bilinear form $\eta$, the universal coefficient $F_r^{\theta}(\gamma_1,\dots,\gamma_r)$ in \eqref{eq_reconstruction_intro} is given by the flow tree formula:
\begin{equation}
\label{eq:flow_tree_formula_1_intro}
    F_r^\theta(\gamma_1,\dots,\gamma_r) 
    = \sum_{T \in \cT_r^{\eta}} \prod_{v \in V_T^\circ} \epsilon_{T,v}^{\alpha,\omega} \,\, \kappa(\eta(e_{v'},e_{v''}))\,,
\end{equation}
where $\epsilon_{T,v}^{\alpha,\omega}$ is as in \eqref{eq_epsilon_intro}
and \begin{equation}
    \label{Eq: kappa}
    \kappa(x)\coloneqq {(-1)^x} \cdot \frac{y^x-y^{-x}}{y-y^{-1}}
\end{equation}
for every $x \in \Z$.
\end{theorem}

Theorem \ref{main_thm} presents a version of 
Theorem \ref{main_thm_intro} in which we phrase more explicitly the 
condition that $\omega$ should be a small enough generic perturbation of $\eta$.

We also prove a variant of the flow tree formula
recently
conjectured by Mozgovoy \cite{mozgovoy2021operadic},
which relies on a perturbation of points in $\mathcal{M}_{\RR}$ rather than the skew-symmetric form. We first remark that  
$\theta \in \gamma^{\perp}$ implies that $\alpha \in \cM_\R$ defined in \eqref{Eq: alpha_intro}
satisfies $\alpha \in (\sum_{i=1}^r e_i)^{\perp}$.
For $\beta$ a small perturbation of $\alpha$ in the hyperplane 
$(\sum_{i=1}^r e_i)^{\perp}$, we define $\theta_{T,v}^{\beta,\eta} \in \mathcal{M}_{\RR}$ and $\epsilon_{T,v}^{\beta,\eta} \in \{0,1,-1\}$ by replacing 
$\alpha$ by 
$\beta$ and $\omega$ by $\eta$
in \eqref{eq_discrete_flow_intro}
and \eqref{eq_epsilon_intro}.

\begin{theorem} \label{main_thm_intro_moz_gps}
For every choice $\beta \in (\sum_{i=1}^r e_i)^{\perp}$ 
of small enough generic perturbation of $\alpha \coloneqq q(\theta)$
in the hyperplane $(\sum_{i=1}^r e_i)^{\perp}$, the universal coefficient $F_r^{\theta}(\gamma_1,\dots,\gamma_r)$ 
is given by:
\begin{equation}
\label{eq:flow_tree_formula_2_intro}
    F_r^\theta(\gamma_1,\dots,\gamma_r) 
    = \sum_{T \in \cT_r^{\eta}} \prod_{v \in V_T^\circ} \epsilon_{T,v}^{\beta,\eta} \,\, \kappa(\eta(e_{v'},e_{v''}))\,,
\end{equation}
where $\epsilon_{T,v}^{\alpha,\omega}$ is as in \eqref{eq_epsilon_intro}
and $\kappa$ is as in \eqref{Eq: kappa}.
\end{theorem}

In Theorem \ref{main_thm_moz_gps}, we present a version of 
Theorem \ref{main_thm_intro_moz_gps} in which we state more precisely the 
condition that $\beta$ should be a small enough generic perturbation of $\alpha$.

\subsection{Structure of the proof}
\label{section_proof_intro}
The proof of Theorems \ref{main_thm_intro}
and \ref{main_thm_intro_moz_gps} relies on the 
notion of a \emph{scattering diagram}, introduced in \cite{MR2846484}, based on the insights of \cite{KS}, to provide an algebro-geometric understanding of the mirror symmetry phenomenon in physics. To give the rough idea of a scattering diagram, which we elaborate further in \S \ref{section_N_g_scattering}, fix a
nilpotent
$N^+$-graded Lie algebra  $\fg=\bigoplus_{n \in N^+} \fg_n$. There is an associated unipotent
algebraic group $G$ with a bijective exponential map
$\exp:\fg \to G$
defined using the Baker–Campbell–Hausdorff formula. Given this data, a 
$(N^+,\fg)$-scattering diagram is defined as the collection of real codimension $1$ cones in $M_{\RR}$, called walls, which are decorated by wall-crossing automorphisms, that are elements of $G$. 
We focus attention on scattering diagrams relevant to DT and cluster theory, which have wall-crossing automorphism preserving a holomorphic symplectic form as in
\cite{MR2667135, MR2662867, MR3330788, MR3710055,
MR3758151,mou2019scattering,mandel2015scattering,MR4131036, davison2019strong}, and not on the more general scattering diagrams
that have wall-crossing automorphisms preserving a holomorphic volume form, and which appear frequently in the context of mirror symmetry
\cite{MR2846484, gross2016theta, arguz2020higher, keel2019frobenius}.

A codimension $2$ locus in $M_{\RR}$ along which distinct walls intersect is called a \emph{joint}.  A scattering diagram is said to be \emph{consistent} if for any joint, the path-ordered product of all wall-crossing automorphisms of walls that are adjacent to the joint is identity.
It is shown in \cite{KS,MR2846484} that there is an algorithmic prescription for constructing a consistent scattering diagram from the data of an initial set of walls. This prescription is based on inserting new walls, along with wall-crossing automorphisms, which order-by-order decrease the divergence of the path-ordered products of wall-crossing automorphisms around joints from being identity. 

Given a quiver with potential $(Q,W)$, Bridgeland
\cite{MR3710055} constructed from the DT invariants of $(Q,W)$ a consistent scattering diagram in $M_\R$, called the 
\emph{stability scattering diagram}, 
whose initial walls are determined by the attractor DT invariants. 
The stability scattering diagram is a very useful tool to study 
DT invariants of quivers. For example, the transformation properties of DT invariants under mutations of a quiver with potential, conjectured in
\cite{manschot2014generalized}\cite[Conjecture 3.14]{mozgovoy2020attractor},
are proved in \cite[Theorem 4.22]{mou2019scattering} by a study of the corresponding transformation of the stability scattering diagram.

The main technical goal of the paper is to prove Theorems 
\ref{thm_flow_tree_formula_scattering} and \ref{thm_flow_tree_formula_scattering_2}: they are flow tree formulas for consistent scattering diagrams which express as a sum over binary trees the wall-crossing automorphism attached to a general wall in terms of the wall-crossing automorphisms attached to the initial walls. In \S \ref{section_DT_invariants}, we then derive
Theorems \ref{main_thm_intro}
and \ref{main_thm_intro_moz_gps} from the flow tree formulas for scattering diagram applied to the stability scattering diagram.

The proof of Theorems 
\ref{thm_flow_tree_formula_scattering} and \ref{thm_flow_tree_formula_scattering_2} is given in \S \ref{section_proof_scattering} and consists of two parts.
In the first part of the proof, described in 
\S \ref{section_reduction_to_Gamma}, 
we relate the $(N^+,\fg)$-scattering diagrams, which live in 
$M_\R$, to auxiliary $(\cN^+,\fh)$-scattering diagrams which live in 
$\cM_\R$, where $\fh$ is a $\cN^+$-graded Lie algebra constructed from 
$\fg$. In the second part of the proof in \S \ref{section_proof_thm},  
we show that the discrete attractor flow naturally defines a embedding of the binary rooted trees
inside the walls of the auxiliary scattering diagrams in $\cM_\R$. The images of the trees in $\cM_\R$ are embedded graphs in $\cM_\R$ with a balancing condition satisfied at each vertex distinct from the root, that is, essentially 
\emph{tropical disks} in $\cM_\R$
\cite{MR2259922, MR2600995, cps}.
The generic perturbation of either the skew-symmetric form or the position in $\cM_\R$ of the root of the embedded trees guarantees that the vertices of the embedded trees are always contained in double intersections of walls, but never in triple intersections. 
The iteration of the local consistency condition around double intersection of walls determines the contribution of each tree. In the language of DT invariants, this reduces the general wall-crossing formula to an iteration of the much simpler primitive wall-crossing formula.

We note in Remark \ref{rem_gps} that the perturbation of the position in $\cM_\R$ of the root of the trees used in the formulation of 
Theorems \ref{main_thm_intro_moz_gps} and \ref{thm_flow_tree_formula_scattering_2} is related to a way of perturbing scattering diagrams going back to the work of Gross-Pandharipande-Siebert \cite{MR2667135}. However the perturbation of the skew-symmetric form used in the formulation of Theorems \ref{main_thm_intro} and 
\ref{thm_flow_tree_formula_scattering} seems to be a completely new way to 
study scattering diagrams. Thus, most of the paper is focused on the 
study of this perturbation of the skew-symmetric form and on the proof of 
Theorems \ref{main_thm_intro} and 
\ref{thm_flow_tree_formula_scattering}.

\subsection{Related work}

\subsubsection{Operads and wall-crossing}

Very recently, while this paper was being completed, Mozgovoy  \cite{mozgovoy2021operadic} proved using an operadic approach to the wall-crossing formula, a different formula for the coefficients $F_r^{\theta}(\gamma_1,\dots,\gamma_r)$, called the \emph{attractor tree formula} and originally conjectured in \cite{mozgovoy2020attractor},
following \cite{MR4072224,MR4170291}.
The key differences between the 
flow tree formula that we prove in this paper and the attractor tree formula proved in \cite{mozgovoy2021operadic}
are the following:
the flow tree formula involves binary trees, requires a choice of generic perturbation, and is naturally phrased in terms of Lie algebras, whereas the attractor tree formula involves 
general (not necessarily binary) trees, does not require the choice of generic perturbation, and is naturally phrased in terms of associative algebras.
It is currently not known if one of these two formulas implies the other in a simple way.

Both the flow tree formula and the attractor tree formula, formulated precisely and proved for DT invariants of quivers with potentials, are expected to have versions holding more generally in DT theory as long as a global understanding of the space of stability conditions is available.
For example, the flow tree formula and the attractor tree formula play an important role in the conjectural proposal of Alexandrov and Pioline \cite{MR4072224} (see also 
\cite{MR4170291}) for the construction of modular completions for generating series of DT invariants counting coherent sheaves supported on surfaces inside Calabi-Yau 3-folds.

\subsubsection{BPS states}

From a physics perspective, a quiver with potential $(Q,W)$ defines a supersymmetric quantum mechanical system with $4$ supercharges
\cite{MR1952307}
and the (refined) DT invariants are counts of supersymmetric ground states, which often can be identified with supersymmetric indices counting 
BPS particles in $4$-dimensional $\cN=2$ supersymmetric
quantum field theories \cite{MR2219440, MR3087917, MR3268234, MR3106506, cecotti2010r}
and BPS configurations of black holes 
in $4$-dimensional $\cN=2$ string compactifications
\cite{MR1952307, MR3036499, MR2875965,MR3080495, MR2967676, MR3033854, MR2913216, MR2511440, MR3036440}. In particular, the definition of the attractor point, as well as the attractor invariants, is motivated by the attractor mechanism for BPS black holes in 
$\cN=2$ supergravity \cite{MR1360416, MR1402863}. 
The attractor invariants are closely related but not equal in general to the 
\emph{single-centered invariants} \cite{MR3036499}, which are expected 
to count micro-states of a single, spherically symmetric black hole, but whose conceptual definition is still 
mysterious mathematically.
The flow tree formula conjectured by Alexandrov--Pioline \cite{AlexandrovPioline}, that we prove in this paper, is motivated by the split attractor flow picture in $\cN=2$ supergravity
\cite{MR1792870, MR1845419, MR2913216}. The idea that the supergravity attractor flow could be replaced by a discrete attactor flow using sign functions was first suggested by Manschot 
\cite{MR2888006}.

\subsubsection{Tropical curves and mirror symmetry}
In \cite{MR2667135, cps, MR3383167, mandel2015scattering},
the perturbation of scattering diagrams originally introduced 
by Gross-Pandharipande-Siebert
\cite{MR2667135} is used to express general walls of a consistent scattering diagram in terms of the initial walls using sums over tropical curves.
The connection between scattering diagrams and tropical geometry is particularly interesting from the point of view of mirror symmetry and connection with Gromov--Witten theory, as shown in dimension $2$ by Gross-Pandharipande-Siebert
\cite{MR2667135} in genus $0$ and the second author \cite{MR4157555}
in higher genus, and generalized to higher dimensions in the work of the first author with Gross \cite{arguz2020higher}.

However, the point of view adopted in the present paper is different:
the main interest of the flow tree formula  is that it is not written as a sum over tropical curves but as a sum over abstract trees.
The resulting formula is therefore entirely combinatorial, and more amenable to formal manipulations, as exemplified in  
\cite{AlexandrovPioline,MR4072224, MR4170291}. In particular, the flow tree formula can be easily implemented efficiently on a computer, 
as done in \cite{piolinem}.

\subsection{Plan of the paper}
In \S
\ref{section_trees_flows}, we introduce our notation for trees and the discrete attractor flow, and we prove the existence of suitably generic perturbations of the skew-symmetric form. 
In \S \ref{section_scattering}, we first review the reconstruction of consistent scattering diagrams from initial data, and then we state the flow tree formula for scattering diagrams.
The technical heart of the paper is
\S \ref{section_proof_scattering}
in which we prove the flow tree formula for scattering diagrams. Finally, we prove in \S \ref{section_DT_invariants}
the flow tree formula for DT
invariants of quivers with potentials by applying the flow tree formula for scattering diagrams to the stability scattering diagram.

\subsection{Acknowledgments} 
We thank Boris Pioline and Sergey Mozgovoy for exchanges on their works \cite{mozgovoy2020attractor} and 
\cite{mozgovoy2021operadic}. We also thank the anonymous referee for their careful reading and helpful comments which improved the manuscript.

\section{Trees and flows}
\label{section_trees_flows}

In \S \ref{section_trees} and \S\ref{section_skew}, 
we introduce elementary notions on trees and skew-symmetric forms that are used throughout the paper. In \S \ref{section_discrete_flow}, we review the discrete attractor flow following \cite{AlexandrovPioline}. In \S \ref{section_generic}, we prove the existence of sufficiently generic skew-symmetric bilinear forms 
to allow the definition of the flow tree map in \S \ref{section_flow_tree_map}.

Throughout this section we fix a free abelian group $\cN$ of finite rank $r$, and let $\cM \coloneqq \Hom_{\Z}(\mathcal{N},\Z)$ and 
$\cM_\R \coloneqq \cM \otimes_{\Z} \R$. 
We introduce the notation 
$I \coloneqq \{1,\dots,r\}$, we fix a basis $\{e_i\}_{i \in I}$ of $\cN$,
and we denote
\begin{equation} \label{eq:cNplus}
    \cN^+ \coloneqq \big\{ \sum_{i\in I} a_i e_i \,|\, a_i \geq 0,\, \sum_{i\in I} a_i >0 \big\} \,.
\end{equation}
We also fix
a skew-symmetric bilinear form 
$\eta \in \bigwedge^2 \cM$ on $\cN$,
a subset $J \subset I$ of cardinality $|J|$, and let 
\begin{equation}
e_J \coloneqq \sum_{i\in J}e_i\,.
\end{equation}
Finally, for every non-zero $n \in \cN$, we denote by
$n^{\perp}\coloneqq \{ \theta \in \cM_{\R}\,|\, 
\theta(n)=0\}$ the corresponding hyperplane in $\cM_{\R}$.

\subsection{Trees}
\label{section_trees}

\begin{definition}
\label{Def: tree parent child etc}
A \emph{rooted tree} $T$ is a connected tree with a finite number of vertices and edges,
with no divalent vertices, 
together with the additional data of 
a distinguished univalent vertex referred to as the \emph{root}. 
We denote by $V_T$ the set of vertices of $T$, by
$\mathcal{R}_T$ the set with the root for unique element, 
$V_T^\circ$ the set of \emph{interior vertices}, which are vertices of valency bigger than one, and by $V_T^L$ the set of univalent vertices that are not the root, that is the set of \emph{leaves} of $T$. 
An \emph{isomorphism} between 
two rooted trees $T$ and $T'$ is a bijection $\varphi \colon V_T \rightarrow V_{T'}$, which maps adjacent vertices of $T$ to adjacent vertices of $T'$ and the root of $T$ to the root of $T'$.
\end{definition}

\begin{definition}
\label{Def: decoration}
A \emph{J-decorated} rooted tree is a rooted tree $T$ endowed with a decoration of the leaves of $T$
by $\{e_i\}_{i\in J}$, that is, \ a bijection $\psi \colon 
V_T^L \rightarrow \{e_i\}_{i\in J}$.
An \emph{isomorphism} between two $J$-decorated rooted trees $(T,\psi)$ and $(T',\psi')$ is an isomorphism of tree $\varphi \colon V_T \rightarrow V_{T'}$ compatible with the decorations, in the sense that 
$\psi= \psi'  \circ \varphi$.
\end{definition}

\begin{figure}
\center{\scalebox{.6}{\input{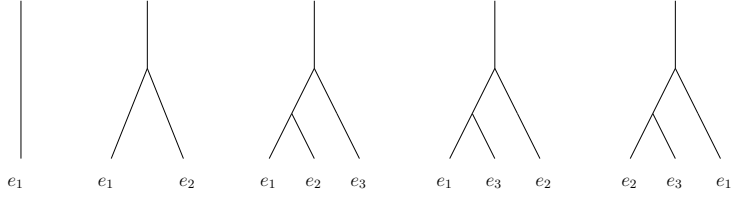}}}
\caption{Decorated binary rooted trees with $\leq 3$ leaves.}
\label{Fig:trees_1}
\end{figure}

\begin{definition}
\label{Def:parent child}
Let $T$ be a rooted tree.
The \emph{parent} of a vertex $v \in V_T \setminus \mathcal{R}_T$ is the unique vertex denoted by $p(v)$ which is adjacent to $v$ and lies on the shortest path between $v$ and the root. A \emph{child} of a vertex $v \in V_T$ is a vertex for which $v$ is a parent, and a \emph{descendant} of $v$ is any vertex which is either the child of $v$ or is (recursively) the descendant of any of the children of $v$. 
\end{definition}

\begin{definition}
\label{def_tree_T}
A rooted tree $T$ is \emph{binary} if the root has exactly one child and each interior vertex has two children. 
\end{definition}

\begin{remark}
\label{Rem: v and v'}
We illustrate in Figure \ref{Fig:trees_1} some decorated binary rooted trees.
Our binary rooted trees are unordered in the sense that we do not fix an order on the set of children of a vertex.
In a binary rooted tree $T$, for every vertex 
$v \in V_T^\circ$, we denote by $\{v',v''\}$ the set of the children of $v$, without specifying an ordering. Nonetheless, for some constructions in what follows it will be sometimes useful to choose an ordering for the children. At any occasion where such a choice is made we will show that the result
of the construction is in fact independent of this choice.
\end{remark}

\begin{lemma} \label{lem_vertices_edges}
Let $T$ be a $J$-decorated binary rooted tree. Then, 
$T$ has $2|J|$ vertices and $2|J|-1$ edges. 
\end{lemma}

\begin{proof}
The proof is by induction on the cardinality $|J|$ of $J$. The result is immediate for $|J|=1$. 
For $|J|>1$, write $J=\{i_0\} \sqcup \{i\}_{i\in |J'|}$
with $|J'|=|J|-1$. Removing from $T$ the leg decorated by $e_{i_0}$, and erasing the resulting 
divalent vertex, we obtain a $J'$-decorated binary rooted tree $T'$. The result follows since $T'$ has two less edges and two less vertices than $T$.
\end{proof}

\begin{lemma} \label{lem_cardinal_T}
The set $\cT_J$ of isomorphism classes of $J$-decorated binary rooted trees is of cardinality $(2|J|-3)!!
=\prod_{k=1}^{|J|-1}(2k-1)$.
\end{lemma}

\begin{proof}
The proof is by induction on the cardinality $|J|$ of $J$.
The result is immediate for $|J|=1$.
For $|J|>1$,
write $J=\{i_0\} \sqcup \{i\}_{i\in |J'|}$
with $|J'|=|J|-1$. Removing from $T$ the leg decorated by $e_{i_0}$, we obtain a $J'$-decorated binary rooted tree $T'$ with an added divalent vertex on one of its edges $E$. Conversely, given a $J'$-decorated binary rooted tree $T'$ and an edge $E$ of $T'$, then adding a divalent vertex $v$ in the middle of $E$ and  gluing a leg decorated by 
$e_{i_0}$ to $v$, we obtain a $J$-decorated binary rooted tree. Therefore, we have a bijection between $\cT_J$ and the set of pairs $(T',E)$, where 
$T' \in \cT_{J'}$ and $E$ is an edge of $T'$.
By Lemma \ref{lem_vertices_edges},
a $J'$-decorated binary rooted tree has 
$2|J'|-1$ edges, and so 
$|\cT_J|=(2|J'|-1)|\cT_{J'}|=(2|J|-3)|\cT_{J'}|$.
\end{proof}

\subsection{Skew-symmetric bilinear forms}
\label{section_skew}
We view elements $\omega \in \bigwedge^2 \cM_\R$ as $\R$-valued skew-symmetric bilinear forms on $\cN$, given by
\begin{align}
    \omega \colon \cN \times \cN & \longrightarrow \R \\
    (v_1, v_2) & \longmapsto \omega(v_1,v_2) \nonumber    \,.
\end{align}

\begin{definition}
\label{Def: charge}
For every tree $T \in \cT_J$, and a vertex 
$v \in V_T$, we define an associated element $e_v \in \cN^+$, referred to as the \emph{charge} of $v$ as follows: Let $J_{T,v} \subset J$ be the subset of indices with which the leaves that are descendant to $v$ are labeled, that is, $j \in J_{T,v}$ if and only if the leaf decorated by 
$e_j$ is a descendant of $v$. Then, we set
\begin{equation}
    e_v \coloneqq e_{J_{T,v}}=\sum_{i \in J_{T,v}} e_i \,.
\end{equation}
\end{definition}
Note that if $v$ is the leaf decorated by 
$e_i$, then the associated charge 
$e_v= e_i$.
For $v \in V_T^\circ$,
the sets $J_{T,v'}$ and $J_{T,v''}$ are disjoint, and
we have 
$e_v=e_{v'}+e_{v''}$.
If $v$ is the root of $T$ or the child of the root of $T$, then 
$J_{T,v}=J$ and 
$e_v=e_J$.

\begin{lemma} \label{lem_delta}
For every tree $T \in \cT_J$ and interior vertex $v \in V_T^\circ$, the linear form
\begin{align}
    \wedge^2 \cM_\R &\longrightarrow \R \\
    \omega &\longmapsto \omega(e_{v'},e_{v''}) \nonumber
\end{align}
is not identically zero.
\end{lemma}

\begin{proof}
As $\{e_i\}_{i\in I}$ is a basis of $\cN$, the linear forms $\omega \mapsto \omega(e_i,e_j)$ for
$i,j \in I$ and $i<j$ form a basis of the space of linear forms on 
$\bigwedge^2 \cM_\R$. We have 
\begin{equation} \label{eq_delta}
    \omega(e_{v'},e_{v''}) = \sum_{j' \in J_{T,v'}}
    \sum_{j'' \in J_{T,v''}} \omega(e_{j'},e_{j''}) \,.
\end{equation}
As the sets $J_{T,v'}$ and $J_{T,v''}$ are disjoint, each basis element $\omega \mapsto \omega(e_{j'},e_{j''})$ with $j'<j''$
appears up to sign at most once in the sum \eqref{eq_delta}.
In particular, there are no cancellations and 
$\omega \mapsto \omega(e_{v'},e_{v''})$ is not the zero linear form.
\end{proof}

\begin{proposition} \label{prop_U_open_dense}
Let $U_J \subset \bigwedge^2 \cM_\R$ be the subset of $\omega \in \bigwedge^2 \cM_\R$ such that for every tree 
$T \in \cT_J$ and interior vertex $v \in V_T^\circ$, 
we have $\omega(e_{v'},e_{v''}) \neq 0$. Then, the following holds:
\begin{itemize}
    \item[(i)] $U_J$ is open and dense in $\bigwedge^2 \cM_\R$.
     \item[(ii)] For every $\omega \in U_J$, $T \in \cT_J$ and $v \in V_T^\circ$, we have
$\omega(e_v,e_{v'}) \neq 0$ and $\omega(e_v,e_{v''}) \neq 0$.
    \item[(iii)] For every $J_2 \subset J_1 \subset I$, we have $U_{J_1} \subset U_{J_2}$. 
\end{itemize}
\end{proposition}

\begin{proof}
By Lemma \ref{lem_delta}, $U_J$ is the complement of finitely many hyperplanes in $\bigwedge^2 \cM_\R$. Thus, the statement in (i) follows. To show (ii), observe that as $e_v=e_{v'}+e_{v''}$, we have 
$\omega(e_v,e_{v'})=\omega(e_{v''}, e_{v'})$ and 
$\omega(e_v,e_{v''})=\omega(e_{v'},e_{v''})$. 
Finally, (iii) follows from the fact that every $J_2$-decorated binary rooted tree can be realized as a subtree of a $J_1$-decorated binary rooted tree.
\end{proof}

\subsection{The discrete attractor flow}
\label{section_discrete_flow}
We review the description of the discrete attractor flow introduced in \cite[\S~2.6]{AlexandrovPioline}.

\begin{definition}
\label{Def: the discrete flow}
Fix a tree $T \in \cT_J$, a skew-symmetric bilinear form 
$\omega \in U_J \subset \bigwedge^2 \cM_\R$ and a point
$\alpha \in e_J^{\perp} 
\subset \cM_\R$. We also fix a labeling 
$v', v''$ of the children of the vertices 
$v \in V_T^\circ$.
\emph{The discrete attractor flow} for $(T,\omega,\alpha)$ is the map 
\begin{align}
    \theta_T^{\alpha,\omega}  \colon \mathcal{R}_T \cup V_T^\circ &\longrightarrow \cM_\R \\
     v &\longmapsto  \theta_{T,v}^{\alpha, \omega} \nonumber
\end{align}
defined inductively, following the flow on $T$ starting at the root and ending at the leaves, as follows:
\begin{enumerate}
    \item For the root vertex $v \in \mathcal{R}_T$, we set 
    \begin{equation}  \label{eq_discrete_flow_1}
    \theta_{T,v}^{\alpha,\omega} \coloneqq \alpha
    \end{equation}
    \item For $v \in V_T^\circ$, and a child $v'$ of $v$, we set
    \begin{equation} \label{eq_discrete_flow_2}
        \theta_{T,v}^{\alpha,\omega} = 
        \theta_{T,p(v)}^{\alpha,\omega} 
- \frac{\theta_{T,p(v)}^{\alpha,\omega}(e_{v'})}{
\omega(e_v,e_{v'})} \iota_{e_v} \omega \,.
    \end{equation}
where $p(v)$ is the parent of $v$, and for every $n \in \cN$, $\iota_n \omega = \omega(n,-) \in \cM_{\RR}$.
\end{enumerate}
\end{definition}

Note that since $\omega \in U_{J}$, we have 
$\omega(e_v,e_{v'}) \neq 0$ for every 
$v \in V_T^\circ$ by Proposition \ref{prop_U_open_dense},
and so \eqref{eq_discrete_flow_2} makes sense.

\begin{lemma} \label{lem_attractor}
Using the notations of Defn.\ \ref{Def: the discrete flow}, we have for every 
$v \in V_T^\circ$:
\begin{equation} \label{eq_wall_intersection}
    \theta_{T,v}^{\alpha, \omega} \in e_{v'}^\perp \cap e_{v''}^{\perp} \subset e_v^{\perp} \,,
\end{equation}
and
    \begin{equation} \label{eq_discrete_flow_4}
        \theta_{T,v}^{\alpha,\omega} = 
        \theta_{T,p(v)}^{\alpha,\omega} 
- \frac{\theta_{T,p(v)}^{\alpha,\omega}(e_{v''})}{
\omega(e_v,e_{v''})} \iota_{e_v} \omega \,.
    \end{equation}
In particular, the discrete flow $\theta_T^{\alpha, \omega}$ defined as in \eqref{Def: the discrete flow} is independent of the choice of labeling
$v'$ and $v''$ of children of vertices $v \in V_T^\circ$.
\end{lemma}

\begin{proof}
We prove the result inductively following the flow on $T$
starting at the root and ending at the leaves.
If $v \in V_T^\circ$ is the child of the root of $T$, then $\theta_{T,v}^{\alpha,\omega}$
is given by \eqref{eq_discrete_flow_2}.
By \eqref{eq_discrete_flow_1}, we have $\theta_{T,p(v)}^{\alpha, \omega}=\alpha$ and so
\begin{equation}
    \theta_{T,v}^{\alpha,\omega}(e_{v'})=\alpha(e_{v'})
    -\frac{\alpha(e_{v'})}{\omega(e_v,e_{v'})}\omega(e_v,e_{v'})=0 \,.
\end{equation}
On the other hand, as $\alpha \in e_J^{\perp}$,
we have 
$\alpha(e_v)=\alpha(e_J)=0$, and so, using $e_v=e_{v'}+e_{v''}$, we have 
$\alpha(e_{v''})=-\alpha(e_{v'})$. As we also have $\omega(e_v,e_{v'})=-\omega(e_v,e_{v''})$, we finally obtain
\begin{equation}
    \theta_{T,v}^{\alpha,\omega}(e_{v''})=\alpha(e_{v''})
    +\frac{\alpha(e_{v''})}{\omega(e_v,e_{v'})}
    \omega(e_v,e_{v''})=0 \,.
\end{equation}
Similarly, if $v \in V_T^\circ$ is not the root of $T$, then 
$\theta_{T,v}^{\alpha, \omega}$ is given by 
\eqref{eq_discrete_flow_2} and so 
\begin{equation}
    \theta_{T,v}^{\alpha,\omega}(e_{v'})=\theta_{T,p(v)}^{\alpha, \omega}(e_{v'})
    -\frac{\theta_{T,p(v)}^{\alpha, \omega}(e_{v'})}{\omega(e_v,e_{v'})}\omega(e_v,e_{v'})=0 \,.
\end{equation}
By the induction hypothesis, we have 
$\theta_{T,p(v)}^{\alpha, \omega}(e_v)=0$ and so, using $e_v=e_{v'}+e_{v''}$, we have 
$\theta_{T,p(v)}^{\alpha, \omega}(e_{v''})=-\theta_{T,p(v)}^{\alpha, \omega}(e_{v'})$. As we also have $\omega(e_v,e_{v'})=-\omega(e_v,e_{v''})$, we 
finally obtain \eqref{eq_discrete_flow_4} and 
\begin{equation}
    \theta_{T,v}^{\alpha,\omega}(e_{v''})=\theta_{T,p(v)}^{\alpha, \omega}(e_{v''})
    +\frac{\theta_{T,p(v)}^{\alpha, \omega}(e_{v''})}{\omega(e_v,e_{v'})}\omega(e_v,e_{v''})=0 \,.
\end{equation}
\end{proof}

\subsection{Generic skew-symmetric bilinear forms}
\label{section_generic}

Recall that we are fixing a skew-symmetric bilinear form
$\eta \in \bigwedge^2 \cM$ on $\cN$. 

\begin{definition}
\label{def_tree_eta}
We denote
by $\cT_J^\eta$ the set of trees $T \in \cT_J$
such that $\eta(e_{v'},e_{v''}) \neq 0$ where 
$v$ is the child of the root of $T$.
\end{definition}

\begin{definition}
\label{def_J_generic}
A point $\alpha \in \cM_\R$
is \emph{$(J,\eta)$-generic} if
$\alpha \in e_J^{\perp}$ and 
for every tree 
$T \in \cT_J^\eta$, we have $\alpha(e_{v'}) \neq 0$, 
where $v$ is the child of the root of $T$.
\end{definition}

Note that for $T \in \cT_J^\eta$ and $v$ the child of the root of $T$, we have $e_{v'}+e_{v''}=e_v=e_J$, and so, if $\alpha \in e_J^{\perp}$, then $\alpha(e_{v'}) \neq 0$
is equivalent to $\alpha(e_{v''}) \neq 0$.
Equivalently, a point 
$\alpha \in e_J^{\perp}$ is $(J,\eta)$-generic if $\alpha \notin e_{J'}^{\perp}$
for every strict subset $J'$ of $J$ such that $\eta(e_J,e_{J'}) \neq 0$.

\begin{definition} \label{Def: alpha generic delta}
Let $\alpha \in e_J^{\perp}$ be a $(J,\eta)$-generic point. 
A skew-symmetric bilinear form $\omega \in U_{J} \subset \bigwedge^2 \cM_\R$ is called \emph{$(J,\alpha)$-generic}
if for every $T \in \cT_J^\eta$ and $v \in V_T^\circ$, we have 
\begin{equation} \label{eq: good}
    \theta_{T, p(v)}^{\alpha,\omega}(e_{v'})\neq 0 \,\,\,
    \text{and} \,\,\, \theta_{T, p(v)}^{\alpha,\omega}(e_{v''})\neq 0\,.
\end{equation}
We denote by
$U_{J,\alpha} \subset U_J$ the set of $(J,\alpha)$-generic skew-symmetric bilinear forms.
\end{definition}

\begin{lemma} \label{lem_theta_zero}
Using the notations of Defn.\ \ref{Def: alpha generic delta},
for every $T \in \cT_J^\eta$ and $v \in V_T^\circ$, we have 
$\theta_{T,p(v)}^{\alpha,\omega}(e_v)=0$
and $
    \theta_{T,p(v)}^{\alpha,\omega}(e_{v'})=
    -\theta_{T,p(v)}^{\alpha,\omega}(e_{v''})$.
\end{lemma}

\begin{proof}
As $e_v=e_{v'}+e_{v''}$, it is enough to show that 
$\theta_{T,p(v)}^{\alpha,\omega}(e_v)=0$. If 
$v$ is the child of the root, then
$\theta_{T,p(v)}^{\alpha,\omega}=\alpha$ by
\eqref{eq_discrete_flow_1}, and so, as $\alpha \in e_J^{\perp}$, we have $\alpha(e_v)=\alpha(e_J)=0$. 
If $v$ is not the child of the root, the result follows by \eqref{eq_wall_intersection} of Lemma \ref{lem_attractor} applied to the parent $p(v)$ of $v$.
\end{proof}

\begin{lemma} \label{lem_rational_function}
Let $\alpha \in e_J^{\perp}$ be 
a $(J,\eta)$-generic point, $T \in \cT_J^\eta$ and $v \in V_T$. 
Denote by $v_0, \dots, v_m$ the unique sequence of vertices 
of $T$ such that $v_0$ is the root of $T$, $v_m=v$, and for every $0 \leq a \leq m-1$, $v_{a+1}$ is a child of 
$v_a$. Then, the following holds.
\begin{itemize}
   \item[(i)] The elements 
$e_{v_0}, \dots, e_{v_m}$ are linearly independent in 
$\cN$.
    \item[(ii)] For every $0 \leq a \leq m-1$ and $0\leq  b \leq m$, the map
\begin{align}
\nonumber
U_{J} &\longrightarrow \R \\
\omega &\longmapsto \theta_{T,v_a}^{\alpha, \omega}(e_{v_b}) \nonumber
\end{align}
is a rational function with $\R$-coefficients, in the variables given by the linear maps 
\begin{align}
\nonumber
U_{J} &\longrightarrow \R \\
\omega &\longmapsto \omega(e_{v_{a'}},e_{v_{b'}}) \nonumber
\end{align}
for $0 \leq a', b' \leq m$ and $\min(a',b') \leq a$.
  \item[(iii)] For every $0 \leq a \leq m-2$, the map
\begin{align}
\nonumber
U_{J} &\longrightarrow \R \\
\omega &\longmapsto \theta_{T,v_a}^{\alpha, \omega}(e_{v_{a+2}}) \nonumber
\end{align}
is not identically zero.   
\end{itemize}
\end{lemma}

\begin{proof}
(i) Assume that $\sum_{i=0}^m a_i e_{v_i}=0$ with some $a_{i} \neq 0$. Let 
$i_{\min}$ be the smallest index $i$ such that $a_i \neq 0$. There exists $j \in J_{T,v_{i_{\min}}}$ such that 
$j \notin J_{T,v_i}$ for every $i >i'$, and so we obtain a contradiction.

(ii) We prove this by induction on $a$ from $0$ to $m-1$. 
For $a=0$, $v_0$ is the root of $T$, and so
 by  \eqref{eq_discrete_flow_1}
we have
$\theta_{T,v_0}^{\alpha,\omega}(e_{v_b})= \alpha(e_b)$
which is constant as a function of $\omega$. Now, assume that the result holds for $a \geq 0$ and that $a+1 \leq m-1$.
Then, we have $v_{a+1} \in V_T^\circ$, and so by  \eqref{eq_discrete_flow_2}, 
\begin{equation} 
        \theta_{T,v_{a+1}}^{\alpha,\omega}(e_{v_b}) 
        = 
        \theta_{T,v_a}^{\alpha,\omega}(e_{v_b}) 
- \frac{\theta_{T,v_a}^{\alpha,\omega}(e_{v_{a+2}})}{
\omega(e_{v_{a+1}},e_{v_{a+2}})} \omega(e_{v_{a+1}}, e_{v_b}) \,.
    \end{equation}
By the induction hypothesis, $ \theta_{T,v_a}^{\alpha,\omega}(e_{v_b})$ and 
$\theta_{T,v_a}^{\alpha,\omega}(e_{v_{a+2}})$ are rational functions in the variables $\omega(e_{v_{a'}}, e_{v_{b'}})$ with $\min(a',b') \leq a$ and so in particular with $\min(a',b') \leq a+1$. The only extra variables appearing in $\theta_{T,v_{a+1}}^{\alpha,\omega}(e_{v_b})$ are 
$\omega(e_{v_{a+1}},e_{v_{a+2}})$ and $\omega(e_{v_{a+1}}, e_{v_b})$,
which are both of the form $\omega(e_{v_{a'}}, e_{v_{b'}})$ with $\min(a',b') \leq a+1$. This shows the result for $a+1$.

(iii) First note that by part (i), the elements 
$e_{v_0},\dots, e_{v_m}$ are linearly independent in $\cN$, and so the 
linear forms $\omega \mapsto \omega(e_{v_a}, e_{v_b})$ with 
$a<b$ are linearly independent. 

We prove the result by induction on $a$ from $0$ to $m-2$. 
For $a=0$, $v_0$ is the root of $T$ and we have by  \eqref{eq_discrete_flow_1} that $\theta_{T,v_0}^{\alpha,\omega}(e_{v_2})= \alpha(e_{v_2})$, which is nonzero 
because $T \in \cT_J^\eta$ and $\alpha$ is $(J,\eta)$-generic (see Defn.\ \ref{def_J_generic}).

Assume that the result holds for $a$  and that $a+1 \leq m-2$. We have to show that the result holds for 
$a+1$.
As $a+1 \leq m-2$, we have in particular $v_{a+1} \in V_T^\circ$ and so, by  \eqref{eq_discrete_flow_2},
\begin{equation} 
        \theta_{T,v_{a+1}}^{\alpha,\omega}(e_{v_{a+3}}) 
        = 
        \theta_{T,v_a}^{\alpha,\omega}(e_{v_{a+3}}) 
- \frac{\theta_{T,v_a}^{\alpha,\omega}(e_{v_{a+2}})}{
\omega(e_{v_{a+1}},e_{v_{a+2}})} \omega(e_{v_{a+1}}, e_{v_{a+3}}) \,.
    \end{equation}
By Lemma \ref{lem_rational_function} (ii), $\omega \mapsto \theta_{T,v_a}^{\alpha,\omega}(e_{v_{a+2}})$ and 
$\omega \mapsto \theta_{T,v_a}^{\alpha,\omega}(e_{v_{a+3}})$ are rational functions in the linear forms $\omega \mapsto \omega(e_{v_{a'}}, e_{v_{b'}})$ with $\min(a',b') \leq a$. 
In particular, they are algebraically independent of 
$\omega \mapsto \omega(e_{v_{a+1}},e_{v_{a+2}})$ and 
$\omega \mapsto \omega(e_{v_{a+1}}, e_{v_{a+3}})$. 
On the other hand, by the induction hypothesis, $\omega \mapsto \theta_{T,v_a}^{\alpha,\omega}(e_{v_{a+2}})$ is not identically zero. We conclude that $\omega \mapsto \theta_{T,v_{a+1}}^{\alpha,\omega}(e_{v_{a+3}})$ is not identically zero.
\end{proof}

\begin{proposition}
\label{Prop: def_V}
Let $\alpha \in e_J^{\perp}$ 
be a $(J,\eta)$-generic point.
Then the set $U_{J,\alpha} \subset U_{J} \subset 
\bigwedge^2 \cM_\R$ defined in 
Defn.\ \ref{Def: alpha generic delta} 
is the complement of finitely many algebraic hypersurfaces in 
$U_{J}$. In particular, 
$U_{J,\alpha}$ is open and dense in $U_{J}$, and so in $\bigwedge^2 \cM_\R$.
\end{proposition}

\begin{proof}
By Lemma \ref{lem_rational_function} (ii) and (iii), for every $T \in \cT_J^\eta$, $v \in V_T^\circ$ and $v'$ child of $v$, the map $\omega \mapsto \theta_{T,p(v)}^{\alpha, \omega}(e_{v'})$
is a  not identically zero rational function. 
Therefore, the set \[\{ \omega \in U_{J}\,|\, \theta_{T,p(v)}^{\alpha, \omega}(e_{v'}) \neq 0 \}\]
is the complement of an algebraic hypersurface in 
$U_{J}$. By definition, 
$U_{J,\alpha}$ is the intersection of the finitely many sets of this form obtained by varying $T$, $v$, and $v'$.
Hence, $U_{J,\alpha}$ is the complement of finitely
many algebraic hypersurfaces in $U_{J}$ and 
is open and dense in $U_{J}$. By Proposition 
\ref{prop_U_open_dense}, $U_{J}$ is open and dense in 
$\bigwedge^2 \cM_\R$, and so it is also the case for 
$U_{J,\alpha}$.
\end{proof}

We end this section in a different direction: instead of fixing 
$\alpha \in e_J^{\perp}$ and looking for $(J,\alpha)$-generic $\omega
\in \bigwedge^2 \cM_\R$, we look for all 
$\alpha \in e_J^{\perp}$ such that the fixed 
$\eta \in \bigwedge^2 \cM_\R$ is $(J,\alpha)$-generic.

\begin{lemma} \label{lem_linear}
Let $T \in \cT_J^\eta$ and $v \in V_T$. 
Denote by $v_0,\dots,v_m$ the unique sequence of vertices of $T$
such that $v_0$ is the root of $T$, $v_m=v$, and for every 
$0 \leq a \leq m-1$, $v_{a+1}$ is a child of $v_a$. 
Then for every $0 \leq a \leq m-2$, the map
\begin{align}
    e_J^{\perp} &\longrightarrow \cM_\R \\
    \alpha &\longmapsto \theta^{\alpha,\eta}_{T,v_a} \nonumber
\end{align}
is linear, and the linear form
\begin{align}
    e_J^{\perp} &\longrightarrow \R \\
    \alpha & \longmapsto \theta^{\alpha,\eta}_{T,v_a}(e_{v_{a+2}}) \nonumber
\end{align}
is not identically zero.
\end{lemma}

\begin{proof}
The result is easily proved by induction on $a$, using 
Lemma \ref{lem_rational_function}(i) and the fact that
the linear form
$\alpha \mapsto \theta_{T,v_a}^{\alpha,\eta}(e_{v_{a+2}})$ is equal to
the sum of the linear form $\alpha \mapsto \alpha(e_{v_{a+2}})$ and of a linear 
combination of the linear forms $\alpha \mapsto \alpha(e_{v_b})$ with $b<a+2$.
\end{proof}

\begin{proposition} \label{prop_generic_point}
Let $V_{J,\eta}$ be the set of $\alpha \in e_J^{\perp} \subset \cM_\R$
such that $\alpha$ is $(J,\eta)$-generic and $\eta$ is $(J,\alpha)$-generic. Then $V_{J,\eta}$ is open and dense in $e_J^{\perp}$.
\end{proposition}

\begin{proof}
It follows from Defn.\ \ref{def_J_generic}
and Lemma \ref{lem_linear} that $V_{J,\eta}$ is the complement of
finitely many hyperplanes in $e_J^{\perp}$.
\end{proof}

\subsection{The flow tree map}
\label{section_flow_tree_map}

Let $\fh=\bigoplus_{n \in \cN^+} \fh_n$ be a Lie algebra over 
$\Q$ which is $\cN^+$-graded, that is, such that
$[\fh_{n_1},\fh_{n_2}]\subset \fh_{n_1+n_2}$
for every $n_1,n_2 \in \cN^+$.
We say that $\fh$ is \emph{finitely 
$\cN^+$-graded} if its 
\emph{support} 
$\Supp(\fh) \coloneqq \{ n \in  \cN^+ \,|\, \fh_n \neq 0 \}$ is finite.
Note that a finitely $\cN^+$-graded Lie algebra is nilpotent.
In what follows, we fix $\fh=\bigoplus_{n \in \cN^+} \fh_n$
a finitely $\cN^+$-graded Lie algebra.
For every $x \in \R-\{0\}$, we denote by 
$\sgn(x)$ the sign of $x$ defined as follows: 
\begin{equation} \sgn(x) =  \begin{cases} 
      1 & \mathrm{if} ~ x>0, ~ \mathrm{and}  \\
      -1 & \mathrm{if} ~ x<0.
   \end{cases}
\end{equation}
\begin{definition} \label{def_flow_tree_map}
Fix a $(J,\eta)$-generic point
$\alpha \in e_J^{\perp} \subset \cM_\R$,
a skew-symmetric bilinear form
$\omega \in U_{J,\alpha}\subset \bigwedge^2 \cM_\R$
as in Defn.\ \ref{Def: alpha generic delta}, 
a tree $T \in \cT_J^\eta$, and for every 
interior vertex $v \in V_T^\circ$ a labeling $v'$ and 
$v''$ of the children of $v$. 
We define a multilinear map 
\begin{equation}
    A_{J,T,v}^{\alpha, \omega} 
    \colon \prod_{i\in J_{T,v}} \fh_{e_i} \longrightarrow 
    \fh_{e_v} 
\end{equation}
for every 
$v \in V_T^L \cup V_T^\circ$
inductively, following the flow on $T$ starting at the leaves and ending at the root, as follows:
\begin{enumerate}
    \item If $v \in V_T^L$, that is, if 
    $v$ is a leaf of $T$ decorated by some $e_i$, we define 
$A_{J,T,v}^{\alpha, \omega} \colon \fh_{e_i} \rightarrow \fh_{e_i}$ as the identity map.
    \item If $v \in V_T^\circ$, we set 
\begin{equation} \label{eq:epsilon}
\epsilon^{\alpha, \omega}_{T,v} \coloneqq 
- \frac{\sgn 
   (\theta_{T,p(v)}^{\alpha,\omega}(e_{v'}))
   + \sgn ( \omega(e_{v'}, e_{v''}))}{2}
   \in \{0,1,-1\}\,,
\end{equation}    
and
\begin{equation} \label{eq_flow_tree_map}
A_{J,T,v}^{\alpha, \omega} 
    \coloneqq \epsilon^{\alpha, \omega}_{T,v} \, [A_{J,T,v'}^{\alpha, \omega},
   A_{J,T,v''}^{\alpha, \omega}] \,,
\end{equation}
where $[A_{J,T,v'}^{\alpha, \omega},
A_{J,T,v''}^{\alpha, \omega}]$ is the composition of the maps 
$
    A_{J,T,v'}^{\alpha, \omega} 
    \colon \prod_{j\in J_{v'}} \fh_{e_j} \longrightarrow 
    \fh_{e_{v'}}$ and 
$    A_{J,T,v''}^{\alpha, \omega} 
    \colon \prod_{j\in J_{v''}} \fh_{e_j} \longrightarrow 
    \fh_{e_{v''}} $
with the Lie bracket
$    [-,-]
    \colon  \fh_{e_{v'}} \times \fh_{e_{v''}} 
    \longrightarrow 
    \fh_{e_{v'}+e_{v''}}=\fh_{e_v} \,. $
\end{enumerate}
\end{definition}
Note that by the definition of $U_J$,
we have $\omega(e_{v'},e_{v''}) \neq 0$
for every $v \in V_T^\circ$. Moreover, by Defn.\ \ref{Def: alpha generic delta} 
of 
$U_{J,\alpha}$, we have 
$\theta_{T,p(v)}^{\alpha, \omega}(e_{v'}) \neq 0$. Hence, both of the signs $\sgn(\omega(e_{v'}, e_{v''}))$ and $\sgn (\theta_{T,p(v)}^{\alpha,\omega}(e_{v'}))$
in  \ref{eq:epsilon} make sense.

\begin{lemma}
Using the notations of Defn.\ \ref{def_flow_tree_map}, for every $v \in V_T^\circ$, we have 
\begin{equation} 
A_{J,T,v}^{\alpha, \omega} 
    = -\frac{\sgn 
   (\theta_{T,p(v)}^{\alpha,\omega}(e_{v''}))
   + \sgn ( \omega(e_{v''}, e_{v'}))}{2} \, [A_{J,T,v''}^{\alpha, \omega},
   A_{J,T,v'}^{\alpha, \omega}]\,.
\end{equation}
In particular, the map $A_{J,T,v}^{\alpha, \omega}$ is independent of the choice of the labeling of the children $v'$
and $v''$ of $v \in V_T^\circ$.
\end{lemma}
\begin{proof}
Since the Lie bracket is skew-symmetric, we have
$[A_{J,T,v''}^{\alpha, \omega},
   A_{J,T,v'}^{\alpha, \omega}]=
   -[A_{J,T,v'}^{\alpha, \omega},
   A_{J,T,v''}^{\alpha, \omega}]$. 
Moreover, since $\omega$ is skew-symmetric, we have 
$\sgn (\omega(e_{v''},e_{v'}))=-\sgn (\omega(e_{v'},e_{v''}))$.
Finally, by Lemma \ref{lem_theta_zero}, we have 
$\sgn (\theta_{T,p(v)}^{\alpha,\omega}(e_{v''}))
= - \sgn (\theta_{T,p(v)}^{\alpha,\omega}(e_{v'}))$.
\end{proof}

\begin{definition} \label{def_final_A}
For every $(J,\eta)$-generic $\alpha \in e_J^{\perp}$,
$\omega \in U_{J,\alpha}$ and 
$T \in \cT_J^\eta$, let
\begin{equation}
    A_{J,T}^{\alpha, \omega} \colon \prod_{i \in J} \fh_{e_i}
    \longrightarrow \fh_{e_J} 
\end{equation}
be the linear map associated to $T$, defined by $A_{J,T}^{\alpha, \omega} \coloneqq 
A_{J,T, v}^{\alpha, \omega}$, where 
$v$ is the child of the root of $T$. 
For every $(J,\eta)$-generic $\alpha \in e_J^{\perp}$
and $\omega \in U_{J,\alpha}$, we define the \emph{flow tree map} $A_{J}^{\alpha, \omega}$ \emph{with initial point $\alpha$},
by summing over all the trees in $\cT_J^\eta$:
\begin{equation} \label{eq:sum_trees}
A_J^{\alpha, \omega} \coloneqq 
   \sum_{T \in \cT_J^\eta} A_{J,T}^{\alpha, \omega}.
\end{equation}
\end{definition}

\section{Scattering diagrams}
\label{section_scattering}
In \S \ref{section_N_g_scattering}, we review the
concept of consistent scattering diagram, mainly following 
\cite{MR3710055,MR3330788,MR3758151}.
In \S
\ref{section_initial_scattering}, we recall the notion of initial data for scattering diagrams.
Finally, in \S \ref{section_universality}, we make explicit the universal nature of the reconstruction of consistent scattering diagrams from their initial data.

\subsection{Consistent scattering diagrams}
\label{section_N_g_scattering}

Throughout this section, we fix a free abelian group $N$ of finite rank $\ell$, and let 
$M \coloneqq \Hom(N,\Z)$ and $M_\R \coloneqq M \otimes_\Z \R$.
We fix a basis $\{s_i\}_{1 \leq i \leq \ell}$ of $N$, and we denote
\begin{equation}
    N^+ \coloneqq \big\{ \sum_{i=1}^\ell a_i s_i \,|\, a_i \in \Z_{ \geq 0},\, \sum_{i=1}^\ell a_i >0 \big\} \,.
\end{equation}
For every $n \in N-\{0\}$,
we denote 
    $n^{\perp}\coloneqq \{ \theta \in M_{\R}\,|\, 
\theta(n)=0\}$, 
and for every subset 
$\fd \subset M_{\R}$, we denote 
$\fd^{\perp} \coloneqq \{ n \in N^+ \,|\, \theta(n)=0 \,\,\, \text{for every} \,\, \theta \in \fd \}$.
Finally, we fix a
finitely
$N^+$-graded Lie algebra  $\fg=\bigoplus_{n \in N^+} \fg_n$ 
over $\Q$, that is, a $N^+$-graded Lie algebra whose support
\begin{equation}\label{eq_support}
\Supp(\fg)\coloneqq \{ n \in  N^+ \,|\, \fg_n \neq 0 \}
\end{equation}
is a finite set. In particular, $\fg$ is a nilpotent Lie algebra.

For us, a  \emph{cone} in $M_\R$ is a closed, convex, rational, polyhedral cone in $M_\R$, that is, a subset of 
$M_\R$ of the form 
\begin{equation}
\sigma = \big\{ \sum_{i=1}^q \lambda_i m_i \,|\, \lambda_i \in \R_{\geq 0} \big\}\,,\,\,\, m_1, \dots, m_q \in M.
\end{equation}
By definition, the \emph{codimension} of a cone is the codimension of the subspace of $M_\R$ it spans. 
A \emph{wall} is a cone of codimension $1$ and a 
\emph{joint} is a cone of codimension $2$. 
If $\fd$ is a wall in $M_\R$, we denote by $n_\fd$ the unique primitive element in $N^+ \cap \fd^{\perp}$, referred to as the \emph{normal vector to the wall}.
A \emph{face} of a cone $\sigma$ is a subset of the form
$\sigma \cap n^{\perp}$
where $n \in N$ satisfies $\theta(n) \geq 0$ for all $\theta \in \sigma$. 
Note that every face of a cone is itself a cone, and every
intersection
of faces of a given cone is also a face.
Finally, a \emph{cone complex} in $M_\R$ is a finite collection 
$\fS$ of cones in $M_{\R}$, such that
 any face of a cone in $\fS$ is also a cone in 
 $\fS$,
and the intersection of any two cones in $\fS$ is a face of each.

\begin{definition} \label{def_cone}
For every finite subset $P \subset N^+$, we denote by $\fS_P$ the cone complex in $M_\R$ whose cones are indexed by partitions
$P=P_+ \sqcup P_0 \sqcup P_-$ with $P_0$ non-empty
and given by 
\[\sigma(P_+,P_0,P_-)
    \coloneqq \{ \theta \in M_\R \,|\, \theta(n)=0
    \,\, \text{for} \,\, n\in P_0\,, \pm \theta(n) \geq 0
    \,\, \text{for} \,\, n \in P_{\pm} \}\,.\]
We denote by $\Wall_P$ the set of walls in $\fS_P$.
\end{definition}

In what follows, we take for the finite set $P \subset N^+$ in Defn.\ \ref{def_cone}
the support $\Supp(\fg) \subset N^+$ of the Lie algebra $\fg$ defined by 
\eqref{eq_support}.

\begin{definition}
\label{def_scattering_diagram}
A \emph{$(N^+, \fg)$-scattering diagram} 
is a map 
\[ \phi \colon \Wall_{\Supp(\fg)} \longrightarrow 
\fg\] 
with the property that \[
\phi(\fd) \in \bigoplus_{n \in \Z_{\geq 1}n_\fd} \fg_{n}
\subset \fg\] for every $\fd \in \Wall_{\Supp(\fg)}$.
For every $n \in \Z_{\geq 1}n_\fd$, the projection of $\phi(\fd)$
on $\fg_{n}$ is denoted by $\phi(\fd)_{n}$.
\end{definition}

\begin{definition}
A smooth path $\fp \colon [0,1] \rightarrow M_\R$ 
is \emph{$\fg$-generic} if
\begin{enumerate}
\item the endpoints $\fp(0)$ and $\fp(1)$ do not lie in any wall $\fd  \in \Wall_{\Supp(\fg)}$,
\item $\fp$ does not meet any cone of $\fS_{\Supp(\fg)}$ of codimension $>1$,
\item  all intersections of $\gamma$ with walls
$\fd \in \Wall_{\Supp(\fg)}$ are transversal.
\end{enumerate}
\end{definition}
Note that, given a $\fg$-generic path $\fp \colon [0,1] \rightarrow M_\R$  there is a finite set of points 
\begin{equation} 0<t_1<\ldots<t_k<1
\end{equation}
for which $\fp(t_i)$ lies in 
$\bigcup_{\fd \in \Wall_{\Supp(\fg)}} \fd$, and 
for each of these points $t_i$ there is a unique wall 
$\fd_i \in \Wall_{\Supp(\fg)}$ such that 
$\fp(t_i) \in \fd_i$.
Given a $(N^+,\fg)$-scattering diagram 
$\phi$ and a $\fg$-generic path 
$\fp \colon [0,1] \rightarrow M_\R$, 
we define the \emph{path-ordered product along $\fp$ of $\phi$} by 
\begin{equation}\label{eq:path_automorphism}
\Psi_{\fp, \phi} 
\coloneqq 
\exp(\epsilon_k \phi(\fd_k))\cdot 
\exp(\epsilon_{k-1} \phi(\fd_{k-1}))
\dots 
\exp(\epsilon_2 \phi(\fd_2)) 
\cdot \exp(\epsilon_1 \phi(\fd_1)) \in G \,,
\end{equation}
where $\epsilon_i \in \{ \pm 1\}$ is the sign of the derivative of 
$t \mapsto -\fp(t)(n_{\fd_i})$ at $t=t_i$, $G$ is the unipotent group associated to the nilpotent Lie algebra $\fg$, and $\exp \colon \fg \rightarrow G$ is the exponential map.

\begin{definition} \label{def_consistency}
A $(N^+,\fg)$-scattering diagram $\phi$ is \emph{consistent}
if $\Psi_{\fp_1,\phi}=\Psi_{\fp_2,\phi}$ for every two $\fg$-generic paths $\fp_1$ and $\fp_2$ with the same endpoints.
\end{definition}
Note that Defn.\ \ref{def_consistency} is equivalent to the definition of the consistency mentioned in the introduction, which requires the composition of all wall-crossing automorphisms on walls adjacent to a given joint to be identity. We set 
$M^+_\R \coloneqq \{ \theta \in M_{\R}\,|\, \theta(n)>0 \,\,
\forall n\in N^+ \}$
and 
$M^-_\R \coloneqq \{ \theta \in M_{\R}\,|\, \theta(n)<0 \,\,
\forall n\in N^+ \}\,.$
The cone complex 
$\fS_{\Supp(\fg)}$
is disjoint from 
$M^+_\R$ and $M^-_\R$. Therefore, 
if $\phi$ is a consistent $(N^+,\fg)$-scattering diagram,
we can consider the element
$\Psi_{\fp,\phi}$ of $G$, where $\fp$ is
a $\fg$-generic path with initial point in
$M^+_\R$
and final point in
$M^-_\R$. By consistency of $\phi$, 
$\Psi_{\fp,\phi}$
is independent of the particular choice of $\fp$, and we set
$\Psi_\phi \coloneqq \Psi_{\fp,\phi} \in G$. 

\begin{proposition} 
\label{prop_scattering}
The map $\phi \mapsto \Psi_\phi$ is a bijection 
between consistent 
$(N^+,\fg)$-scattering diagrams and elements of the group 
$G$.
\end{proposition}

\begin{proof}
In the setting of scattering diagrams as cone complexes, this is exactly Proposition 3.3 of \cite{MR3710055}. In the setting of scattering diagrams as set of walls, this result is originally Theorem 2.1.6 of \cite{MR3330788} 
(see also Theorem 1.17 of \cite{MR3758151}). Note that 
Proposition 3.3 of \cite{MR3710055} in fact shows that these two possible points of view on scattering diagrams are in fact equivalent.
\end{proof}

\subsection{Initial data for scattering diagrams}
\label{section_initial_scattering}
From now on, we assume given a real-valued skew-symmetric bilinear form
$\langle -,-\rangle$ on $N$
such that the finitely $N^+$-graded Lie algebra $\fg=\bigoplus_{n \in N^+} \fg_n$
satisfies
\begin{equation}
\label{eq:skew_lie}
[\fg_{n_1}, \fg_{n_2}]=0 \,\,\,\text{as soon as} \,\,\, 
\langle n_1, n_2 \rangle=0\,.
\end{equation}
In this section we review the notion of initial data for a 
$(N^+,\fg)$-scattering diagram.

For every primitive $\overline{n} \in N^+$, we have a direct sum decomposition 
$\fg=\fg_{\overline{n},+} \oplus \fg_{\overline{n},0} \oplus \fg_{\overline{n},-}$
of $\fg$ into Lie subalgebras 
\begin{equation}
    \fg_{\overline{n},+} \coloneq \bigoplus_{\substack{n \in N^+ \\ 
    \langle \overline{n},n \rangle>0}}\fg_{n} \,,\,\,\,
     \fg_{\overline{n},0} \coloneq \bigoplus_{\substack{n \in N^+ \\ 
    \langle \overline{n},n\rangle=0}}\fg_{n} \,,\,\,\,
     \fg_{\overline{n},-} \coloneq \bigoplus_{\substack{n \in N^+ \\ 
    \langle \overline{n},n \rangle <0}}\fg_{n} \,.
\end{equation}
It follows that, denoting by $G_{\overline{n},+} \coloneq \exp(\fg_{\overline{n},+})$,
$G_{\overline{n},0} \coloneq \exp(\fg_{\overline{n},0})$,
$G_{\overline{n},-} \coloneq \exp(\fg_{\overline{n},-})$ the corresponding subgroups 
of $G$, every element $g \in G$ can be written uniquely as 
a product $g=g_{\overline{n},+} g_{\overline{n},0} g_{\overline{n},-}$ with 
$g_{\overline{n},+} \in G_{\overline{n},+}$, $g_{\overline{n},0} \in G_{\overline{n},0}$,
$g_{\overline{n},-} \in G_{\overline{n},-}$.
We have a further decomposition
$\fg_{\overline{n},0} = \fg_{\overline{n},0}^{\parallel} \oplus \fg_{\overline{n},0}^{\perp}$, 
where 
\begin{equation}
    \fg_{\overline{n},0}^{\parallel} \coloneq \bigoplus_{n \in \Z_{\geq 1} \overline{n}} \fg_{n} \,,\,\,\, 
    \fg_{\overline{n},0}^{\perp} \coloneq 
    \bigoplus_{\substack{n \in N^+ \\ 
    \langle \overline{n},n \rangle=0 \\ 
    n\notin \Z_{\geq 1}\overline{n}}} \fg_{n}\,.
\end{equation}
If $n_1+n_2=k\overline{n}$ with $\langle \overline{n}, n_1 \rangle=0$ and
$\langle \overline{n},n_2\rangle=0$, then $\langle n_1, n_2\rangle=0$
and so $[\fg_{n_1},\fg_{n_2}] = 0$ by \eqref{eq:skew_lie}. 
In particular, we have
$[\fg_{\overline{n},0}, \fg_{\overline{n},0}^{\perp}] \subset \fg_{\overline{n},0}^{\perp}$. Hence, $\fg_{\overline{n},0}^{\perp}$ is a
Lie ideal in $\fg_{\overline{n},0}$ and so the 
subgroup $G_{\overline{n},0}^{\perp}
\coloneq \exp (\fg_{\overline{n},0}^{\perp})$
is normal. We denote by 
\begin{equation}
    \pi_{\overline{n},0} \colon G_{\overline{n},0} \longrightarrow 
    G_{\overline{n},0}/G_{\overline{n},0}^{\perp}=G_{\overline{n},0}^{\parallel}
\end{equation}
the quotient group morphism, where 
$G_{\overline{n},0}^{\parallel}
    \coloneqq \exp(\fg_{\overline{n},0}^{\parallel})$. 
Given $g=g_{\overline{n},+}g_{\overline{n},0}g_{\overline{n},-}$, set 
$g_{\overline{n},0}^{\parallel} \coloneq \pi_{\overline{n},0}(g_{\overline{n},0})$.
This defines a map 
\begin{align}
    \pi_{\overline{n}} \colon G &\longrightarrow G_{\overline{n},0}^{\parallel} \\
    g &\longmapsto g_{\overline{n},0}^{\parallel} \,. \nonumber
\end{align}

\begin{proposition} \label{prop_initial}
The map 
\begin{align}
    \pi \colon G &\longrightarrow \prod_{\substack{\overline{n} \in N^+ \\ n \,\, \text{primitive}}} G_{\overline{n},0}^{\parallel} \\
 g &\longmapsto (\pi_{\overline{n}}(g))_{\overline{n}} \nonumber
\end{align}
is a bijection.
\end{proposition}

\begin{proof}
This is Proposition 3.3.2 of
\cite{MR3330788}. 
See also Proposition 1.20 of \cite{MR3758151}.
\end{proof}

\begin{definition}
\label{Def: initial data}
For every $n \in N^+$, the \emph{initial data} 
$I_{\phi,n}$ of a consistent $(N^+,\fg)$-scattering diagram 
$\phi$ is the projection on $\fg_n$ of 
\begin{equation}
    \log(\pi_{\overline{n}}(\Psi_\phi)) \in \fg_{\overline{n},0}^{\parallel}
    = \bigoplus_{n' \in \Z_{\geq 1} \overline{n}} \fg_{n'}\,,
\end{equation}
where $\overline{n}$ is the unique primitive element of 
$N^+$ such that $n \in \Z_{\geq 1} \overline{n}$, and 
$\Psi_\phi$ is the element of $G$ attached to $\phi$ as in 
Proposition \ref{prop_scattering}.
\end{definition}

\begin{proposition} \label{prop_scattering_from_initial}
The map 
$\phi \mapsto (I_{\phi,n})_{n \in N^+}$ is a bijection between 
equivalence classes of consistent 
$(N^+,\fg)$-scattering diagrams 
and elements of $\fg=\bigoplus_{n \in N^+} \fg_n$. In other words,
for every 
$(\mathcal{I}_n)_{n \in N^{+}} \in \fg=\bigoplus_{n \in N^+}\fg_n$, 
there exists a unique consistent 
$(N^+,\fg)$-scattering diagram $\phi$ with initial data 
$(I_{\phi,n})_{n \in N^+}=(\mathcal{I}_n)_{n \in N^+}$.
\end{proposition}

\begin{proof}
This is an immediate consequence of Propositions \ref{prop_scattering} and 
\ref{prop_initial}.
\end{proof}

The next Proposition \ref{prop_initial_scattering} describes how to read the initial data $I_{\phi,n}$ of a consistent $(N^+,\fg)$-scattering diagram 
$\phi$ from the walls.

\begin{proposition} \label{prop_initial_scattering}
Let $\phi$ be a consistent $(N^+,\fg)$-scattering diagram,
$n \in N^+$ and $\overline{n}$ the unique primitive 
element of $N^+$ such that $n \in \Z_{\geq 1} \overline{n}$. 
For every wall $\fd
\in \Wall_{\Supp(\fg)}$
with $n_\fd=\overline{n}$ and
containing the 
\emph{attractor point} $\langle n,-\rangle \in M_\R$ for $n$, we have 
\begin{equation}
    \phi(\fd)_n=I_{\phi,n}\,.
\end{equation}
\end{proposition}

\begin{proof}
This follows from Theorem 1.21-(1) of \cite{MR3758151}.
\end{proof}

Note that in the context of Proposition \ref{prop_initial_scattering} there are in general several walls $\fd$ with $n_\fd=\overline{n}$ and containing the attractor point $\langle n, -\rangle$. Proposition \ref{prop_initial_scattering}
implies in particular that $\phi(\fd)_n$ does not depend on the choice of $\fd$.

\subsection{Universality of the reconstruction of scattering diagrams from initial data} \label{section_universality}

The next proposition shows that the elements $\phi(\fd) \in \fg$ assigned to walls $\fd \in \Wall_{\Supp(\fg)}$ by a consistent 
$(N^+,\fg)$-scattering diagram $\phi$
are determined by the initial data 
$(I_{\phi,n})_{n \in N^+}$
via universal formulas.

\begin{definition}
A finite multiset $\Gamma=\{\gamma_i\}_{1 \leq i\leq r}$ of elements of $N^+$
is a finite unordered collection $\gamma_1,\dots,\gamma_r$ of elements of 
$N^+$ where multiple occurrences of elements are allowed.
For every $n \in N^+$, the \emph{multiplicity} 
$m_\Gamma(n) \in \Z_{\geq 0}$ of $n$ in $\Gamma$
is the number of occurrences of $n$ in $\Gamma$.
Given a multiset $\Gamma$, we denote by 
$\overline{\Gamma}$ the set of $n \in N^+$
such that $m_\Gamma(n) \neq 0$.
The set of finite multisets of elements of 
$N^+$ is denoted by $\mult(N^+)$.
\end{definition}

\begin{proposition} \label{prop_scattering_from_initial_explicit}
For every $\Gamma \in \mult(N^+)$ and $\fd \in \Wall_{\Supp(\fg)}$, there exists a unique map 
\begin{equation}
\label{Eq: FgDeltaGamma}
    F_\Gamma^{\fg,\fd} \colon \prod_{n\in \overline{\Gamma}} \fg_n\longrightarrow \fg_{\sum_{n\in \Gamma} n} 
\end{equation}
which is a homogeneous polynomial map of degree $m_\Gamma(n)$ in restriction to the factor $\fg_n$, and such that for every 
consistent $(N^+,\fg)$-scattering diagram $\phi$ and 
$\gamma \in \Z_{\geq 1}n_\fd \in N^+$, the component $\phi(\fd)_\gamma$
of $\phi(\fd)$ in $\fg_\gamma$
is given by 
\begin{equation} \label{eq_H}    \phi(\fd)_{\gamma} = \sum_{\substack{\Gamma \in \mult(N^+) \\  
\sum_{n\in \Gamma} n=\gamma}}
F_\Gamma^{\fg,\fd}\left((I_{\phi,n})_{n \in\overline{\Gamma}}\right)\,,
\end{equation}
where the sum is over all finite multisets $\Gamma$ of $N^+$ whose elements sum up to $\gamma$.
\end{proposition}

\begin{proof}
We first prove the uniqueness part. Assume that we have two collections $(F_\Gamma^{\fg,\fd})_1$ and $(F_\Gamma^{\fg,\fd})_2$ of maps satisfying the conditions of Proposition \ref{prop_scattering_from_initial_explicit}. 
By Proposition \ref{prop_scattering_from_initial}, there exists a consistent 
$(N^+,\fg)$-scattering diagram for every initial data.
Therefore  \eqref{eq_H} implies the equality of maps 
\begin{equation} \label{eq_uniqueness}  \sum_{\substack{\Gamma \in \mult(N^+) \\ 
\sum_{n \in \Gamma}=\gamma}}
(F_\Gamma^{\fg,\fd})_1
= \sum_{\substack{\Gamma \in \mult(N^+) \\  
\sum_{n \in \Gamma} n=\gamma}}
(F_\Gamma^{\fg,\fd})_2\,.
\end{equation}
For every $\Gamma \in \mult(N^+)$ with 
$\sum_{n \in\Gamma}n=\gamma$, isolating on both sides of  \eqref{eq_uniqueness} the part homogeneous of degree 
$m_\Gamma(n)$ in restriction to each factor 
$\fg_n$, we obtain 
$(F_\Gamma^{\fg,\fd})_1 = (F_\Gamma^{\fg,\fd})_2$.

We now prove the existence claim. 
Let $\delta \colon N \rightarrow \Z$ be an additive map such that $\delta(N^+) \subset \Z_{\geq 1}$. 
For every $k \in \Z_{\geq 0}$, we define the Lie subalgebra 
$\fg^{>k}
\coloneqq 
\bigoplus_{\substack{n \in N^+ \\ \delta(n)>k}} 
\fg_n  \subset \fg$.
We prove by induction on 
$k$ that for every $k \in \Z_{\geq 0}$, 
$\Gamma \in \mult(N^+)$
and $\fd \in \Wall_{\Supp(\fg)}$,
there exists a map
\begin{equation}
    F_{k, \Gamma}^{\fg,\fd} \colon \prod_{n \in
    \overline{\Gamma}} \fg_n 
    \longrightarrow \fg_{\sum_{n \in\Gamma} n} \,,
\end{equation}
such that for every consistent $(N^+,\fg)$-scattering diagram $\phi$ and $\gamma \in \Z_{\geq 1}n_\fd$,
we have 
\begin{equation} 
    \phi(\fd)_\gamma =
    \sum_{\substack{\Gamma \in \mult(N^+) \\ 
    \sum_{n \in\Gamma} n =\gamma}}
    F_{k,\Gamma}^{\fg,\fd} \left(
    (I_{\phi,n})_{n \in\overline{\Gamma}} \right) \mod \fg^{>k}\,.
\end{equation}
As $\fg$ is nilpotent, we have $\fg^{>k}=0$
for $k$ large enough, and so it will be enough to take $F_{\Gamma}^{\fg,\fd} \coloneqq F_{k,\Gamma}^{\fg,\fd}$ for $k$ large enough.

For the base step of the induction, we have $\fg^{>0}=\fg$, so 
$\phi(\fd)_\gamma=0 \mod \fg^{>0}$ for every $\phi$, $\fd$, $\gamma$, and so we can take $F_{0,\Gamma}^{\fg,\fd}=0$ for every $\Gamma$ and $\fd$. For the induction step, fix $k \geq 0$, and assume that the existence of the maps 
$F_{k, \Gamma}^{\fg,\fd}$ is known. We have to show the existence of the maps $F_{k+1,\Gamma}^{\fg,\fd}$.
For every wall $\fd \in \Wall_{\Supp(\fg)}$ and for every consistent $(N^+,\fg)$-scattering diagram 
$\phi$, define 
\begin{equation} \label{eq_bar_phi}
    \overline{\phi(\fd)} \coloneqq 
    \sum_{\substack{\Gamma \in \mult(N^+)\\
    \sum_{n \in\Gamma} n \in\Z_{\geq 1}n_{\fd}}} F_{k, \Gamma}^{\fg,\fd}
    \left( (I_{\phi,n})_{n \in\overline{\Gamma}} \right) \,.
\end{equation}
By the induction hypothesis, we have 
\begin{equation} \label{eq:modulo}
\phi(\fd)=\overline{\phi(\fd)} \mod \fg^{>k}\,.
\end{equation}

By \cite[Definition-Lemma C.2]{MR3758151}, a joint
 $\fj \in \mathfrak{S}_{\Supp(\fg)}$, that is a codimension $2$ cone, is \emph{perpendicular} if for every wall 
$\fd \in \Wall_{\Supp(\fg)}$ containing $\fj$, the contraction 
$\iota_{n_\fd}\langle-,-\rangle = \langle n_\fd, -\rangle$ of 
$\langle -,-\rangle$ with the normal vector $n_\fd$ to $\fd$ is not contained in the $\R$-linear span of $\fj$.
For every perpendicular joint 
$\fj \in \fS_{\Supp(\fg)}$, 
let $\Wall(\fj)$ be the set of walls 
$\fd \in \Wall_{\Supp(\fg)}$ containing $\fj$, and let
$\fp_{\fj} \colon [0,1] \rightarrow M_\R$ be a 
$\fg$-generic loop around $\fj$, intersecting
only once
each wall $\fd \in \Wall(\fj)$ and no other wall.
For every wall 
$\fd \in \Wall(\fj)$, denote by $t_\fd^\fj \in [0,1]$
the point
such that $\fp_\fj(t_\fd^\fj) \in \fd$, and denote by 
$\epsilon_\fd^\fj \in \{ \pm 1\}$ the sign of the derivative of 
$t \mapsto -\fp_\fj(t)(n_{\fd})$ at $t=t_\fd^\fj$.
We label $\fd_1, \dots, \fd_m$ the elements of 
$\Wall(\fj)$
so that 
$0<t_{\fd_1}^{\fj}<\dots<t_{\fd_m}^{\fj}<1$.
By Defn.\ \ref{def_consistency} the relation
\begin{equation} 
\exp(\epsilon_{\fd_m}^\fj \phi(\fd_m))\cdot \exp(\epsilon_{\fd_{m-1}}^\fj \phi(\fd_{m-1}))
\dots 
\exp(\epsilon_{\fd_2}^\fj \phi(\fd_2)) \cdot \exp(\epsilon_{\fd_1}^\fj \phi(\fd_1)) =1 
\end{equation}
holds for every consistent $(N^+,\fg)$-scattering diagram
$\phi$.
Therefore, it follows from  \eqref{eq:modulo}
that 
\begin{equation} \label{eq_log_exp}
    \log \left( \exp(\epsilon_{\fd_m}^\fj \overline{\phi(\fd_m)})
    \dots \exp(\epsilon_{\fd_1}^\fj \overline{\phi(\fd_1)}) \right) 
    =\sum_{\substack{\gamma \in N^+ \\ \delta(\gamma) \geq k+1}} g_{\phi,\gamma}^\fj 
\end{equation}
for some $g_{\phi,\gamma}^\fj \in \fg_n$. 
Using the Baker-Campbell-Hausdorff formula to compute the left-hand side of \eqref{eq_log_exp}, together with \eqref{eq_bar_phi}, we deduce that for every $\Gamma \in \mult(N^+)$, 
there exists a map $G_{\Gamma}^\fj \colon \prod_{n \in\overline{\Gamma}} \fg_n
\rightarrow \fg_{\sum_{n \in\Gamma}n}$,
which is a homogeneous polynomial map of degree $m_\Gamma(n)$
in restriction to the factor $\fg_n$, such that for every 
consistent $(N^+,\fg)$-scattering diagram $\phi$ and 
$\gamma \in N^+$ with $\delta(\gamma) \geq k+1$, we have 
\begin{equation}
    g_{\phi,\gamma}^\fj = 
    \sum_{\substack{\Gamma \in \mult(N^+)\\
    \sum_{i=1}^r \gamma_i=\gamma}} G_{\Gamma}^\fj 
    \left( (I_{\phi,n})_{n \in\overline{\Gamma}} \right) \,.
\end{equation}
where the sum is over multisets $\Gamma=\{ \gamma_i\}_{i\in I}$ for some index set $I$, whose elements sum up to $\gamma$.
According to Appendix C.1 of \cite{MR3758151}
(see the Equations defining 
$\tilde{\fD}_{k+1}$ and $\fD[\mathfrak{j}]$
before Lemma C.6),
for every wall $\fd \in \Wall_{\Supp(\fg)}$ we have 
\begin{equation} \label{eq:scattering}
    \phi(\fd)=\overline{\phi(\fd)}
    +\sum_{\substack{\gamma \in\Z_{\geq 1}n_\fd \\ \delta(\gamma)=k+1}}
    I_{\phi,\gamma}
    - 
    \sum_{\gamma \in \Z_{\geq 1}n_\fd}
    \sum_{\fj} \epsilon_{\fd_\fj}^\fj
    g_{\phi,\gamma}^{\fj} \mod \fg^{>k+1}\,,
\end{equation}
where the sum over $\fj$ 
is over the perpendicular joints 
$\fj$ such that 
$\fd \subset \fj - \R_{\geq 0} \langle n_\fd, -\rangle$, and where $\fd_\fj \in \Wall(\fj)$ is the 
wall containing $\fj$ and contained in
$\fj - \R_{\geq 0} \langle n_\fd, -\rangle$.
Therefore, for every $\Gamma \in \mult(N^+)$ with
$\sum_{n\in \Gamma} n \in\Z_{\geq 1} n_\fd$, 
we can take 
\begin{equation}
    F_{k+1,\Gamma}^{\fg,\fd} 
    =F_{k,\Gamma}^{\fg,\fd} +I_{k+1,\Gamma}^\fd- \sum_{\fj}
    \epsilon_{\fd_\fj}^{\fj}
    G_\Gamma^{\fj} \,.
\end{equation}
where $I_{k+1,\Gamma}^\fd$ is the identity map 
$\fg_{\gamma} \rightarrow \fg_\gamma$ if $\Gamma=\{\gamma\}$
with $\gamma \in \Z_{\geq 1}n_\fd$
such that
$\delta(\gamma)=k+1$, and $I_{k+1,\Gamma}^\fd=0$ else.
\end{proof}

\section{The flow tree formula for scattering diagrams}
\label{section_proof_scattering}
In this section we prove our main result, Theorem \ref{thm_flow_tree_formula_scattering}, which provides an explicit description of the maps $F_\Gamma^{\fg,\fd}$ in \eqref{Eq: FgDeltaGamma} in terms of the (specialization of the) flow tree maps. 

\subsection{$(\cN^+,\fh)$-scattering diagrams}
\label{section_N_h_scattering}
As in \S \ref{section_scattering}, we work with $(N^+,\fg)$-scattering diagrams.
We fix a  wall $\fd \in \Wall_{\Supp(\fg)}$, an element
$\gamma \in
\Z_{\geq 1}n_\fd \subset N^+$ proportional to the normal vector 
$n_\fd$ to $\fd$, and a multiset $\Gamma = \{ \gamma_i\}_{i\in I} \in \mult(N^+)$
of elements of $N^+$
such that 
$\sum_{i\in I} \gamma_i=\gamma$, where $I=\{ 1,\ldots r\}$ is some index set. 
Applying Proposition \ref{prop_scattering_from_initial_explicit}
to the multiset $\Gamma= \{\gamma_i\}_{i \in I}$ and to the wall $\fd$, we obtain a map 
\begin{equation} \label{eq: F_gamma_}
    F_{\Gamma}^{\fg,\fd} \colon \prod_{n \in \overline{\Gamma}} \fg_{n} \longrightarrow \fg_{\gamma} \,.
\end{equation}
Our goal is to state a formula for the map $F_\Gamma^{\fg,\fd}$.
As a first step to achieve this goal, we define in this section another class of 
scattering diagrams, referred to as $(\cN^+,\fh)$-scattering diagrams.

We introduce a rank $r$ free abelian group
$\cN \coloneqq \bigoplus_{i \in I} \Z e_i$
with a basis $\{e_i\}_{i \in I}$,
and the additive map 
\begin{align}
\label{Eq: p}
    p \colon \cN &\longrightarrow N \\
    e_i &\longmapsto \gamma_i \,. \nonumber
\end{align}
For every $J \subset I$, let
\begin{align}
    \label{Eq: e_J}
    e_J \coloneqq \sum_{i\in J} e_i\,.
\end{align}
In particular, we have $p(e_I)=\gamma$. Following the notations set-up in \S\ref{section_trees_flows}, we denote $\cM \coloneqq \Hom(\cN,\Z)$, $\cM_\R \coloneqq \cM \otimes \R$ and 
$\cN^+ \coloneqq \big\{ \sum_{i\in I} a_i e_i \,|\, a_i \geq 0,\, \sum_{i\in I} a_i >0 \big\}$.
The map $p \colon \cN \rightarrow N$
defines by duality a linear map 
\begin{align}
\label{Eq: q}
    q \colon M_{\R} &\longrightarrow \cM_{\R} \\
    \theta & \longmapsto \theta \circ p \,. \nonumber
\end{align}
We define a skew-symmetric bilinear form $\eta \in \bigwedge^2 \cM$
by 

\begin{align}
\label{Eq: eta}
\eta(e_i,e_j) \coloneqq \langle \gamma_i,\gamma_j\rangle
\end{align}
for every $i, j \in I$. In other words, $\eta$ is the pullback of
$\langle -,-\rangle$ by $p$.

\begin{definition} \label{def_h}
We define a $\cN^+$-graded Lie algebra 
$\fh= \bigoplus_{n \in \cN^+} \fh_n$ as follows.
First, we introduce the finite set 
\begin{equation} \label{eq:N_e}
    \cN_e^+ \coloneqq \{ \sum_{i\in I}a_i e_i \in \cN^+|\, a_i \in \{0,1\}\, \forall i \in I\}=\{ e_J\,|\, J \subset I,\, J \neq \emptyset\} \subset \cN^+ \,.
\end{equation}
Then, as vector spaces, we set $\fh_n \coloneqq \fg_{p(n)}$ if
$n \in \cN_e^+$, and $\fh_n \coloneqq 0$ else. 
For $x \in \fh_{n_1}$ and $y \in \fh_{n_2}$, we define the bracket 
$[x,y]$ as being the bracket $[x,y]$ in $\fh_{n_1+n_2}=\fg_{p(n_1)+p(n_2)}$
if $n_1, n_2, n_1+n_2 \in \cN_e^+$, and as being $0$ else. 
\end{definition}
One checks easily that this defines a Lie bracket on 
$\fh$ and that the resulting Lie algebra is 
finitely $\cN^+$-graded: by construction, the support 
$\Supp(\fh)=\{n \in \cN^+|\, \fh_n \neq 0\}$ of $\fh$ is contained in 
$\cN_e^+$.
It follows from \eqref{eq:skew_lie} that $[\fh_{n_1},\fh_{n_2}]=0$ if $\eta(n_1,n_2)=0$.
Thus, we can consider $(\cN^+,\fh)$-scattering diagrams 
as in Defn.\ \ref{def_scattering_diagram} and their initial data as in Defn.\ \ref{Def: initial data}, where $N^+$, $\fg$ and $\langle-,-\rangle \in \bigwedge^2 M$ are replaced by 
$\cN^+$, $\fh$ and $\eta \in \bigwedge^2 \cM$.

Let $\fe \in \Wall_{\Supp(\fh)}$ be a wall in 
$\cM_\R$ with normal vector $n_\fe =e_I$ and which contains the image $q(\fd)$
of the wall $\fd \in \Wall_{\Supp(\fg)}$
by 
the map $q \colon M_\R \rightarrow \cM_\R$ as in \eqref{Eq: q}.
Applying Proposition \ref{prop_scattering_from_initial_explicit}
to the multiset $\Gamma_e \coloneqq \{e_i\}_{i \in I} \in \mult(\cN^+)$
of elements of $\cN^+$ and to the wall $\fe \in \Wall_{\Supp(\fh)}$, we obtain a map 
\begin{equation} \label{eq: F_gamma_e}
    F_{\Gamma_e}^{\fh,\fe} \colon \prod_{i \in I} \fh_{e_i} \longrightarrow \fh_{e_I} \,,
\end{equation}
where we used that, as $\{e_i\}_{i \in I}$ is a basis of 
$\cN$, we have $\overline{\Gamma}_e = \Gamma_e=\{e_i\}_{i \in I}$.

\subsection{From $(N^+,\fg)$ to $(\cN^+, \fh)$-scattering diagrams}
\label{section_reduction_to_Gamma}

The main result of this section, Theorem \ref{thm_D_D_gamma}, provides a comparison of the map $F_\Gamma^{\fg,\fd}$
in \eqref{eq: F_gamma_} and the map 
$F_{\Gamma_e}^{\fh,\fe}$ in \eqref{eq: F_gamma_e}. 
To prove it, we first need to compare the Lie algebras $\fg$ and $\fh$. We do this by going through an intermediate $N^+$-graded Lie algebra 
\begin{align}
    \label{Eq: tilde g}
\tilde{\fg} 
=\bigoplus_{n \in N^+} \tilde{\fg}_n    
\end{align}
defined using the map $p \colon \cN \rightarrow N$ 
in \eqref{Eq: p} and the finite subset $\cN_e^+ \subset \cN^+$
in \eqref{eq:N_e}.

\subsubsection{The Lie algebra $\tilde{\fg}$}

\begin{definition} \label{def_g_tilde}
Define the Lie algebra  $\tilde{\fg}$ as follows: As vector spaces, we set $\tilde{\fg}_n \coloneqq \fg_{n}$ if
$n \in p(\cN_e^+)$, and $\tilde{\fg}_n \coloneqq 0$ else. 
For $x \in \tilde{\fg}_{n_1}$ and $y \in \tilde{\fg}_{n_2}$, 
we define the bracket 
$[x,y]$ as being the bracket $[x,y]$ in $\tilde{\fg}_{n_1+n_2}=\fg_{n_1+n_2}$
if $n_1, n_2, n_1+n_2 \in p(\cN_e^+)$, and as being $0$ else. 
\end{definition}
One checks easily that this defines a Lie bracket on 
$\tilde{\fg}$ and that the resulting Lie algebra is 
finitely $\cN^+$-graded. It follows from \eqref{eq:skew_lie} that $[\tilde{\fg}_{n_1},\tilde{\fg}_{n_2}]=0$ if 
$\langle n_1,n_2 \rangle=0$.
As $\gamma=p(e) \in \Supp(\tilde{\fg})$, there exists 
a unique wall $\tilde{\fd} \in \Wall_{\Supp(\tilde{\fg})}$
such that $\fd \subset \tilde{\fd}$.
Applying Proposition \ref{prop_scattering_from_initial_explicit} for 
$(N^+,\tilde{\fg})$-scattering diagram to the multiset 
$\Gamma \in \mult(N^+)$ and the wall $\tilde{\fd}$, we obtain a map
\begin{equation} \label{eq: F_gamma_tilde}
    F_\Gamma^{\tilde{\fg},\tilde{\fd}} \colon \prod_{n \in \overline{\Gamma}} \tilde{\fg}_n \longrightarrow \tilde{\fg}_\gamma \,.
\end{equation}

\begin{proposition} \label{prop_scattering_tilde}
The maps $F_\Gamma^{\fg,\fd}$ in \eqref{eq: F_gamma_} and
$F_\Gamma^{\tilde{\fg},\tilde{\fd}}$ in \eqref{eq: F_gamma_tilde} are equal: $F_\Gamma^{\fg,\fd}=F_\Gamma^{\tilde{\fg},
\tilde{\fd}}$.
\end{proposition}

\begin{proof}
By definition of $\tilde{\fg}$, we have $\tilde{\fg}_n = \fg_n$
for every $n \in \Gamma \cup \{ \gamma \}$, and so the maps
$F_\Gamma^{\fg,\fd}$ and $F_\Gamma^{\tilde{\fg},\tilde{\fd}}$ have the same domain and codomain. The result then follows from the fact that the algorithmic construction of $F_\Gamma^{\fg,\fd}$ reviewed in the proof of Proposition \ref{prop_scattering_from_initial_explicit}
involves only brackets $[x,y]$ with $x \in \fg_{n_1}$, $y \in \fg_{n_2}$, $[x,y] \in \fg_{n_1+n_2}$ and $n_1, n_2, n_1+n_2 
\in p(\cN_e^+)$.
\end{proof}
In what remains, we compare the Lie algebras $\tilde{\fg}$ and 
$\fh$.
\begin{proposition}
 \label{prop_basic_tool}
Let $q \colon M_{\RR} \to \mathcal{M}_{\RR}$ be the linear map defined in \eqref{Eq: q}. Then,
\begin{itemize}
   \item[(i)]  For every $n \in \cN$, the preimage 
   $q^{-1}(n^{\perp})$ of the hyperplane $n^{\perp} \subset M_\R$ by the map 
   \[ q \colon M_\R \longrightarrow \cM_\R\] 
   is the hyperplane $(p(n))^{\perp} \subset \cM_\R$.
   \item[(ii)] For every cone $\sigma \in \fS_{\Supp(\tilde{\fg})}$, the image $q(\sigma)$ of $\sigma$ by $q \colon M_\R \rightarrow \cM_\R$
   is a cone $q(\sigma) \in \fS_{\Supp(\fh)}$. 
\end{itemize}
\end{proposition}

\begin{proof} The first part (i) of the Lemma follows immediately since we have $\theta \in q^{-1}(n^{\perp})$ if and only if $(q(\theta))(n)=0$ if and only if $\theta (p(n))=0$.

To show (ii), first note that by Defn.\ \ref{def_cone}, the assumption $\sigma \in \fS_{\Supp(\tilde{\fg})}$
implies that there exists a partition of the set $\Supp(\tilde{\fg}) \subset N^+$ into subsets  
$\Supp(\tilde{\fg})=P_+ \sqcup P_0 \sqcup P_-$
such that 
\begin{equation} \sigma
    \coloneqq \{ \theta \in M_\R \,|\, \theta(n)=0
    \,\, \text{for} \,\, n\in P_0\,, \pm \theta(n) \geq 0
    \,\, \text{for} \,\, n \in P_{\pm} \}\,.\end{equation}
Define $Q_\pm \coloneqq \{ n \in \Supp(\fh)|\,p(n) \in P_{\pm}\}$
and $Q_0 \coloneqq \{n \in \Supp(\fh)|\,p(n) \in P_{0}\}$.
As $\Supp(\tilde{\fg})=p(\Supp(\fh))$, we have 
$\Supp(\fh)=Q_+ \sqcup Q_0 \sqcup Q_-$.
Using that $\theta(p(n))=(q(\theta))(n)$ for every $n \in \cN$, we obtain $q(\sigma)= \{ \theta \in \cM_\R \,|\, \theta(n)=0
    \,\, \text{for} \,\, n\in Q_0\,, \pm \theta(n) \geq 0
    \,\, \text{for} \,\, n \in Q_{\pm} \}$.
    Hence, $q(\sigma) \in \fS_{\Supp(\fh)}$ by Defn.\ \ref{def_cone}.
\end{proof}

\begin{proposition} \label{prop_comparison_attractor_points}
For every $n \in \cN$, the attractor points 
$\langle p(n),-\rangle$ for $p(n)$ and 
$\iota_n \eta=\eta(n,-)$ for $n$ as in Proposition 
\ref{prop_initial_scattering} are related by:
\begin{equation}  q(\langle p(n),- \rangle)
= \iota_n \eta\,,
\end{equation}
where $\eta \in \bigwedge^2 \mathcal{M}$ is defined by \eqref{Eq: eta}.
\end{proposition}

\begin{proof}
For every $m \in \cN$, we have 
\begin{equation} \label{eq_11}
    (q(\langle p(n),-\rangle))(m)
    =\langle p(n), p(m) \rangle
    =\eta(n, m)
    = (\iota_n \eta)(m) \,,
\end{equation}
where the first equality uses \eqref{Eq: q} and the second equality uses
\eqref{Eq: eta}.
\end{proof}

\begin{figure}
\center{\scalebox{.6}{\input{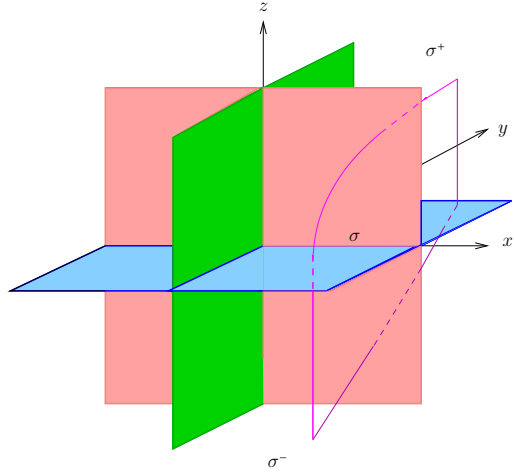}}}
\caption{Paths around a codimension two cone $\sigma$.}
\label{Fig: extension}
\end{figure}

\subsubsection{The $(N^+, \tilde{\fg})$-scattering diagram  and consistency}

In this section, we construct a consistent $(N^+,\tilde{\fg})$-scattering diagram
$\phi_\rho$ starting from a consistent $(N^+,\fh)$-scattering diagram $\rho$.

Let $\rho \colon \Wall_{\Supp(\fh)} \rightarrow \fh$ be a consistent 
$(N^+,\fh)$-scattering diagram. Following \cite[\S$2$]{mou2019scattering}, 
we start by defining an extension 
$\overline{\rho}  \colon \fS_{\Supp(\fh)} \rightarrow \fh$ of $\rho$
where the set of walls $\Wall_{\Supp(\fh)}$ is replaced by the set 
$\fS_{\Supp(\fh)}$ of all cones. 
For a cone $\sigma \in \fS_{\Supp(\fh)}$, there exists by Defn.\ \ref{def_cone}
 a decomposition 
$\Supp(\fh)=P_+ \sqcup P_0 \sqcup P_-$
such that 
\begin{equation}
\sigma
    \coloneqq \{ \theta \in \cM_\R \,|\, \theta(n)=0
    \,\, \text{for} \,\, n\in P_0\,, \pm \theta(n) \geq 0
    \,\, \text{for} \,\, n \in P_{\pm} \}\,.
\end{equation}    
We denote 
\[ \sigma^+ \coloneqq \{ \theta \in \cM_{\R}\,|\, 
\theta(m) >0,\,\forall m \in P_+ \cup P_0 \,, \text{and}\, 
\theta(m)<0,\, \forall m \in P_-\} \]
and 
\[
\sigma^- \coloneqq \{ \theta \in \cM_{\R}\,|\, 
\theta(m) >0,\,\forall m \in P_+\,, \text{and}\, \theta(m)<0,\, \forall m \in P_0 \cup P_- \}\,.\]
Let $\fp \colon [0,1] \rightarrow \cM_{\R}$
be a $\fh$-generic path with $\fp(0) \in \sigma^+$ and 
$\fp(1) \in \sigma^-$ (see Figure \ref{Fig: extension}). 
By
\eqref{eq:path_automorphism}, we have the corresponding
path-ordered product
$\Psi_{\fp,\rho} \in H \coloneqq \exp(\fh)$, and we define
\begin{equation}
\overline{\rho}(\sigma) \coloneqq \log \Psi_{\fp,\rho} \in \fh\,.
\end{equation}
By consistency of $\rho$, this definition of
$\overline{\rho}(\sigma)$ is independent of the choice of the 
path $\fp$.

By Defn.\ \ref{def_h} and Defn.\ \ref{def_g_tilde}, we have
$\tilde{\fg}=\bigoplus_{n \in p(\cN_e^+)} \fg_n$ and
$\fh=\bigoplus_{n \in \cN_e^+} \fg_{p(n)}$.
We denote by $\nu \colon \fh \rightarrow \tilde{\fg}$ the natural projection map sending $\fh_n=\fg_{p(n)}$ onto 
$\fg_{p(n)}=\tilde{\fg}_{p(n)}$.
We now define a
$(N^+, \tilde{\fg})$-scattering diagram 
$\phi_\rho \colon \Wall_{\Supp(\tilde{\fg})} \rightarrow \tilde{\fg}$.
For every wall $\sigma \in  \Wall_{\Supp(\tilde{\fg})}$, the image 
$q(\sigma)$ of $\sigma$ by $q$ is a cone in $\fS_{\Supp(\fh)}$
by Proposition \ref{prop_basic_tool} (ii). Therefore, one can apply $\overline{\rho}$ to $q(\sigma)$ to obtain 
$\overline{\rho}(q(\sigma)) \in \fh$, and finally 
$\nu \colon \fh \rightarrow \tilde{\fg}$:
\begin{align} \label{eq:def_bar_phi}
    \phi_\rho(\sigma) \coloneqq  \nu (\overline{\rho}
    (q(\sigma))) \in \tilde{\fg} \,.
\end{align}

\begin{lemma}
For every consistent $(\cN^+,\fh)$-scattering diagram 
$\rho \colon \Wall_{\Supp(\fh)} \rightarrow \fh$, the $(N^+, \tilde{\fg})$-scattering diagram $\phi_\rho \colon \Wall_{\Supp(\tilde{\fg})} \rightarrow \tilde{\fg}$ defined by \eqref{eq:def_bar_phi} is consistent.
\end{lemma}

\begin{proof}
Let $\fp \colon [0,1] \rightarrow M_\R$ be a $\tilde{\fg}$-generic loop. 
Let 
$\fq$ be a small generic perturbation of 
$t \mapsto q(\fp(t))$
such that, for every 
$\sigma \in \Wall_{\Supp(\tilde{\fg})}$ and 
$t' \in [0,1]$ with $\fp(t') \in \sigma$, the perturbed path $t \mapsto \fq(t)$ goes from $(q(\sigma))^-$ to 
$(q(\sigma))^+$, or from $(q(\sigma))^+$ to 
$(q(\sigma))^-$, in a small neighborhood of $t'$.
By the definition of 
$\phi_\rho$ in \eqref{eq:def_bar_phi}, the group element
$\Psi_{\fp,\phi_\rho}$ is the image in $\tilde{G}=\exp(\tilde{\fg})$
of the group element $\Psi_{\fq,\rho}$
by $\exp(\nu) \colon H \rightarrow \tilde{G}$.
By consistency of $\rho$ we have $\Psi_{\fq,\rho}=\mathrm{id}$, 
and hence $\Psi_{\fp,\phi_\rho}=\mathrm{id}$.
\end{proof}

\begin{lemma} \label{lem_initial_bar_no_bar}
For every consistent $(\cN^+,\fh)$-scattering diagram 
$\rho \colon \Wall_{\Supp(\fh)} \rightarrow \fh$, the initial data of 
$\rho$ and of the $(N^+, \tilde{\fg})$-scattering diagram $\phi_\rho \colon \Wall_{\Supp(\tilde{\fg})} \rightarrow \tilde{\fg}$ defined by \eqref{eq:def_bar_phi} are related as follows:
for every $n \in \Supp(\tilde{\fg})=p(\Supp(\fh))$, we have
\begin{equation} \label{eq:lem_initial}
  I_{\phi_\rho, n}= \sum_{\substack{m 
    \in \Supp(\fh)\\ p(m)=n}}
   \nu(I_{\rho,m})\,,
\end{equation}
where $I_{\phi_\rho, n}$ and $I_{\rho,m}$ are the initial data
of $\phi_\rho$ and $\rho$ as in Defn.\ \ref{Def: initial data}.
\end{lemma}

\begin{proof}
Let $\sigma \in \Wall_{\Supp(\tilde{\fg})}$ be a wall 
containing the attractor point $\langle n,-\rangle$ for $n$
and such that $n \in \Z_{\geq 1} n_{\sigma}$. 
By Proposition \ref{prop_initial_scattering}
applied to $\phi$, we have 
\begin{equation}\label{eq:1}I_{\phi,n}= (\phi(\sigma))_n\,.
\end{equation}
Let $\Delta \subset \Supp(\fh)$ be the subset of 
primitive $m \in \Supp(\fh)$ such that $p(m) 
\in \Z_{\geq 1}n_{\sigma}$. 
By Proposition \ref{prop_basic_tool}, for every primitive 
$m \in \Supp(\fh)$, the hyperplane
$m^{\perp}$ contains the cone $q(\sigma)$ 
if and only if $m \in \Delta$. 

Let $\fp \colon [0,1] \rightarrow \cM_{\R}$
be a $\fh$-generic path with $\fp(0) \in (q(\sigma))^+$ and 
$\fp(1) \in (q(\sigma))^-$. 
For every $m\in \Delta$, we have 
$\theta(m)>0$ for every $\theta \in (q(\sigma))^+$
and $\theta(m)<0$ for every $\theta \in (q(\sigma))^-$.
Therefore, up to straightening $\fp$, one can assume that for every $m\in \Delta$, the path $\fp$ intersects the hyperplane $m^\perp$ exactly once. 
We can also assume that for every 
$m \in \Delta$, the intersection of $\fp$ with 
$m^{\perp}$ lies in a wall $\fd_{m} \subset m^{\perp}$
containing the cone $q(\sigma)$.
For every $m, m' \in \Delta$, we have 
$\eta(m,m')=
\langle p(m),p(m')\rangle=0$, and so 
$[\rho(\fd_{m}),\rho(\fd_{m'})]=0$.
Thus it follows from the definition 
\eqref{eq:def_bar_phi}
of $\phi_\rho$ that 
\begin{equation} \label{eq:2}
    \phi_\rho(\sigma)_n = \sum_{\substack{m \in \Delta, \,k \in \Z_{\geq 1}\\ p(km)=n}} \nu (\rho(\fd_{m})_{km}) \,.
\end{equation}
By Proposition \ref{prop_comparison_attractor_points}, for every 
$m \in \Delta$ and  $k \in \Z_{\geq 1}$ such that  $p(km)=n$,
we have $\iota_{km}\eta=q(\langle n,-\rangle) 
\in q(\sigma) \subset \fd_{m}$. 
We deduce from Proposition \ref{prop_initial_scattering} applied to 
$\rho$ that 
\begin{equation}\label{eq:3}
\rho(\fd_{m})_{km}
= I_{\rho, km}\,.
\end{equation}
Equation \eqref{eq:lem_initial} follows from  \eqref{eq:1}-\eqref{eq:2}-\eqref{eq:3}.
\end{proof}

\begin{definition} \label{def_specialization}
Given a map $\varphi \colon \prod_{i \in I}\fh_{e_i} \rightarrow \fh_{e_I}$, the \emph{specialization} of $\varphi$ is the map 
$\hat{\varphi} \colon \prod_{n \in \overline{\Gamma}} \fg_n \rightarrow \fg_\gamma$ defined as follows. For $(x_n)_{n \in\overline{\Gamma}}
\in \prod_{n \in\overline{\Gamma}} \fg_n$, define  
$(y_i)_{i\in I} \in \prod_{i\in I} \fh_{e_i}$ by $y_i \coloneqq x_{p(e_i)}$, where $p \colon \cN \rightarrow N$ is as in \eqref{Eq: p}, and set
\begin{equation} \hat{\varphi} \left((x_n)_{n \in\overline{\Gamma}} \right) \coloneqq \varphi 
\left((y_i)_{i\in I} \right)\,. 
\end{equation}
\end{definition}

\begin{theorem} \label{thm_D_D_gamma}
Let $\fd \in \Wall_{\Supp(\fg)}$ be a wall in 
$M_\R$ and 
$\Gamma=\{\gamma_i\}_{i\in I}\in \mult(N^+)$ a multiset of elements in $N^+$ 
such that 
$\fd \subset \gamma^{\perp}$, where $\gamma=\sum_{i \in I} \gamma_i$.
Let $\Gamma_e =\{e_i\}_{i \in I} \in \mult(\cN^+)$, and $\fe \in \Wall_{\Supp(\fh)}$ a wall in $\cM_\R$ such that $\fe \subset e_I^{\perp}$ and containing the image $q(\fd)$ of $\fd$ by the map 
$q \colon M_\R \rightarrow \cM_\R$ as in \eqref{Eq: q}.
Then, the maps $F_{\Gamma}^{\fg,\fd}$
in \eqref{eq: F_gamma_} and $F_{\Gamma_e}^{\fh, \fe}$ in \eqref{eq: F_gamma_e} satisfy
\begin{equation}
F_{\Gamma}^{\fg,\fd}
=\frac{1}{\prod_{n \in N^+}m_\Gamma(n)!}
\hat{F}_{\Gamma_e}^{\fh, \fe}\,,
\end{equation}
where $\hat{F}_{\Gamma_e}^{\fh, \fe}$ is the specialization of $F_{\Gamma_e}^{\fh, \fe}$ as in Defn.\ \ref{def_specialization}.
\end{theorem}

\begin{proof}
By Proposition \ref{prop_scattering_tilde}, it is enough to show that 
\begin{equation} \label{eq:comb}
F_{\Gamma}^{\tilde{\fg},\tilde{\fd}}
=\frac{1}{\prod_{n \in N^+}m_\Gamma(n)!}
\hat{F}_{\Gamma_e}^{\fh, \fe}\,.
\end{equation}
Let $\Delta \subset \Supp(\fh)$ be the subset of primitive 
$m \in \Supp(\fh)$ such that $p(m) \in \Z_{\geq 1} n_{\tilde{\fd}}$. 
As $p(e)=\gamma$, we have $e \in \Delta$.
By Proposition \ref{prop_basic_tool}, for primitive 
$m \in \Supp(\fh)$, the hyperplane
$m^{\perp}$ contains the cone $q(\tilde{\fd})$ 
if and only if $m \in \Delta$. 
Arguing as in the proof of Lemma \ref{lem_initial_bar_no_bar}, one can find a $\fh$-generic path 
$\fp \colon [0,1] \rightarrow \cM_{\R}$
with $\fp(0) \in (q(\tilde{\fd}))^+$, 
$\fp(1) \in (q(\tilde{\fd}))^-$, 
and such that for every $m \in \Delta$, 
the path $\fp$ intersects the hyperplane 
$m^{\perp}$ at a single point, lying in a wall 
$\fd_{m} \subset m^{\perp}$ which contains the cone $q(\tilde{\fd})$. 
We can also assume that $\fd_e=\fe$.

Let 
$\rho \colon \Wall_{\Supp(\fh)} \rightarrow \fh$
be a consistent $(\cN^+,\fh)$-scattering diagram and 
$\phi_\rho \colon \Wall_{\Supp(\tilde{\fg})} \rightarrow \tilde{\fg}$
the corresponding consistent $(N^+, \tilde{\fg})$-scattering diagram defined by \eqref{eq:def_bar_phi}.
As in the proof of Lemma \ref{lem_initial_bar_no_bar},
for every $m, m' \in \Delta$, we have 
$[\rho(\fd_{m}),\rho(\fd_{m'})]=0$
and so it follows from the definition \eqref{eq:def_bar_phi} of $\phi_\rho$ that 
\begin{equation} \label{eq:4}
    \phi_\rho(\tilde{\fd})_\gamma = \sum_{\substack{m \in \Delta, \,k \in \Z_{\geq 1}\\ p(km)=\gamma}} \nu (\rho(\fd_{m})_{km}) \,.
\end{equation}
We show below that the equality \eqref{eq:comb}
follows from identifying on both sides of \eqref{eq:4}
the terms homogeneous of degree $m_\Gamma(n)$ in the initial data $I_{\phi_\rho,n}$.

By Proposition \ref{prop_scattering_from_initial_explicit}
applied to $\phi_\rho$, we have 
\begin{equation} \label{eq:bar_phi}
    \phi_\rho(\tilde{\fd})_\gamma
    =
    \sum_{\substack{\Gamma'=\{\gamma'\}\in \mult(N^+) \\
    \gamma' \in \Supp(\tilde{\fg}),\,\, \sum_{\gamma'\in \Gamma'} \gamma'=\gamma}}
    F_{\Gamma'}^{\tilde{\fg},\tilde{\fd}}\left(
    (I_{\phi_\rho, \gamma'})_{\gamma' \in \Gamma'}\right)   \,.
\end{equation}
The only term homogeneous of degree $m_\Gamma(n)$ in the initial data $I_{\phi_\rho,n}$ in \eqref{eq:bar_phi} is obtained for $\Gamma'=\Gamma$ and is equal to $F_{\Gamma}^{\tilde{\fg},\tilde{\fd}}\left(
    (I_{\phi_\rho, n})_{n \in\Gamma}\right)$.

On the other hand, by Proposition \ref{prop_scattering_from_initial_explicit}
applied to $\rho$, the right-hand side
of \eqref{eq:4} is equal to
\begin{equation} \label{eq:6}
    \sum_{\substack{m \in \Delta, \,k \in \Z_{\geq 1}\\ p(km)=\gamma}} 
    \sum_{\substack{\Gamma'=\{n'\} \in \mult(\cN^+)\\
    n' \in \Supp(\fh)\,,\sum_{n'\in \Gamma'} n'=km}}
     \nu \left( F_{\Gamma'}^{\fh,\fd_m} \left((I_{\rho, n'})_{n' \in \Gamma'} \right) \right) \,.
\end{equation}
The only term homogeneous of degree $1$ in the initial data 
$I_{\rho, e_i}$ in \eqref{eq:6} 
is obtained for $\Gamma'=\Gamma_e$ and is equal to 
$\nu(F_{\Gamma_e}^{\fh,\fe}((I_{\rho, e_i})_{1\leq i\leq r}))$.

Finally, by Lemma 
\ref{lem_initial_bar_no_bar},
we have for every $n \in\Gamma$,
\begin{equation} \label{eq_init}
    I_{\phi_\rho,n}= \sum_{e_i, p(e_i)=n} \nu (I_{\rho, e_i}) \,.
\end{equation}
Note that the sum in \eqref{eq_init} contains 
$m_\Gamma(n)$ terms.
Therefore, \eqref{eq:comb} 
follows from the following algebraic claim applied to 
$(x_i)_i=(I_{\phi_\rho,n})_n$, $(y_{ij})_{ij}=(\nu(I_{\rho,e_i}))_i$, 
$f=F_{\Gamma}^{\tilde{\fg},\tilde{\fd}}$ and $g=\nu(F_{\Gamma_e}^{\fh,\fe})$: \\
\textbf{Claim:} Let $f((x_i)_{1 \leq i\leq s})$ be a polynomial function of $s$ variables which is homogeneous of degree $a_i$ in the variable $x_i$. Write each variable $x_i$ as a sum of $a_i$ variables $y_{ij}$: $x_i=\sum_{j=1}^{a_i} y_{ij}$, and let $g((y_{ij})_{1\leq i\leq r, 1\leq j \leq a_i})$ be  the component of
$f((\sum_{j=1}^{a_i}y_{ij})_{1\leq i\leq s})$ which is homogeneous of degree $1$ in each variable $y_{ij}$. Finally, let $\hat{g}((x_i)_{1\leq i\leq s})$ be the function obtained from $g((y_{ij})_{1\leq i\leq s, 1\leq j \leq a_i})$ by the specialization of variables $y_{ij} \mapsto x_i$, for every $1 \leq i \leq s$ and $1\leq j \leq a_i$. Then, we have \begin{equation}\hat{g}((x_i)_{1\leq i\leq s})=\left(\prod_{i=1}^s a_i!\right)f((x_i)_{1\leq i\leq s})\,.\end{equation}
\textbf{Proof of the claim:} It is enough to prove the result for $f=\prod_{i=1}^s x_i^{a_i}$. For $f=\prod_i x_i^{a_i}$, $g$ is the term proportional to $\prod_{i,j}y_{ij}$ in $\prod_i(\sum_j y_{ij})^{a_i}$. So, \ $g=(\prod_i a_i!)\prod_{i,j}y_{ij}$ and so $\hat{g}=(\prod_i a_i!) \prod_i x_i^{a_i}=(\prod_i a_i!) f$. Hence, the result follows.
\end{proof}

\subsection{$(\cN^+, \fh)$-scattering diagrams and flow tree maps} 
\label{section_proof_thm}

This section includes the technical heart of the paper, Theorem \ref{thm_flow_tropical}. The key result of the paper, the flow tree formula in Theorem \ref{thm_flow_tree_formula_scattering}, will follow from Theorem \ref{thm_flow_tropical} and Theorem \ref{thm_D_D_gamma}.

\subsubsection{Small enough generic perturbations of the skew-symmetric bilinear form}
\label{Small generic perturbations of the skew-symmetric bilinear form}

In this section, we define small enough generic perturbations of the 
skew-symmetric bilinear form 
$\eta \in \bigwedge^2 \cM$ defined by \eqref{Eq: eta}.

\begin{definition} \label{def_U_eta}
We denote by $U^\eta$ the set of $\omega \in \bigwedge^2 \cM_\R$
such that for every $n_1, n_2 \in \cN_e^+$ with 
$\eta(n_1,n_2)$ nonzero, $\omega(n_1,n_2)$
is nonzero and has the same sign as
$\eta(n_1,n_2)$.
We have $\eta \in U^\eta$ and $U^\eta$ is an open neighborhood of $\eta$ in 
$\bigwedge^2 \cM_\R$.
\end{definition}

For a fixed 
$(I,\eta)$-generic point $\alpha \in e_I^{\perp} \subset \cM_\R$
as in Defn.\ \ref{def_J_generic}, 
we call a perturbation $\omega$ of $\eta$ \emph{generic} if it belongs to the open dense subset $U_{I,\alpha}\subset \bigwedge^2 \cM_\R$, as in Defn.\ \ref{Def: alpha generic delta}, and we say that the perturbation is \emph{small enough} if
$\omega$ belongs to the open neighborhood 
$U^\eta \subset \bigwedge^2 \cM_\R$, as in Defn.\ \ref{def_U_eta}. Hence, $\omega$ is a \emph{small enough generic perturbation} of $\eta \in \bigwedge^2 \cM$ if
\begin{equation}
    \label{Eq: small and generic}
\omega \in U_{I,\alpha} \cap U^{\eta}\,.
\end{equation}

\subsubsection{Embedding treees in $\cM_\R$ via the discrete attractor flow}
We fix a $(I,\eta)$-generic point $\alpha \in e_I^{\perp}\subset \cM_\R$ as in
Defn.\ 
\ref{def_J_generic} and $\omega \in U_{I,\alpha}$ as in Defn. \ref{Def: alpha generic delta}.
In this section we use the discrete attractor flow defined in \S\ref{section_discrete_flow} to define an embedding of binary trees in $\cM_\R$ as follows.
For every tree $T \in \cT_I$, where $\cT_I$ is defined as in Lemma \ref{lem_cardinal_T}, we denote by $T^{\circ}$
the graph obtained from $T$ by removing all the leaves 
$v \in V_T^L$, and extending the resulting open intervals to unbounded edges.
For every tree $T \in \cT_I$, we fix a continuous map 
\begin{equation}
    \label{Eq: tropical curve j}
j_T^{\alpha,\omega} \colon T^{\circ} \longrightarrow \cM_{\R}    
\end{equation}
such that:
\begin{enumerate}
    \item 
for every vertex $v \in \mathcal{R}_T \cup V_T^\circ$, we have
\begin{equation} \label{eq:j_T_vertex}
j_T^{\alpha,\omega}(v)=\theta_{T,v}^{\alpha, \omega}\,.
\end{equation}
    \item for every bounded edge $E$ of $T^\circ$, connecting vertices 
    $v$ and $v'$, the image of the map $j_T^{\alpha,\omega}$ restricted to $E$ is the line segment in $\cM_\R$ with endpoints $\theta_{T,v}^{\alpha, \omega}$ and $\theta_{T,v'}^{\alpha, \omega}$.
    \item for every 
unbounded edge $E$ of $T^\circ$ obtained by removing the leaf decorated by $e_i$,  the image of the map $j_T^{\alpha,\omega}$ restricted to $E$
is the half-line $\theta_{T,v}^{\alpha, \omega} + \R_{\geq 0} \,
    \iota_{e_i} \omega$ in $\cM_\R$, where $v$ is the vertex in $V_T^\circ$ incident to $E$.
\end{enumerate}

\begin{remark} \label{rem_tropical_curves}
For every tree $T \in \cT_I$, the embedded graph $j_T^{\alpha,\omega}(T^\circ) \subset \cM_\R$ in \eqref{Eq: tropical curve j} defined using the discrete flow has a natural structure of \emph{tropical disks} in $\cM_\R$ \cite{MR2259922,MR2667135, cps} if $\omega \in \bigwedge^2 \cM \otimes_\Z \Q \subset \bigwedge^2 \cM_\R$: edges have then rational weighted directions of the form $\iota_{e_v} \omega$ and the tropical balancing condition at vertices distinct from the root follows from the relation $e_v=e_{v'}+e_{v''}$ in Defn. \ref{Def: charge}.  
\end{remark}

\begin{proposition} \label{prop_not_contracted}For every tree $T \in \cT_I^\eta$ and interior vertex $v \in V_T^\circ$, we have $j_T^{\alpha,\omega}(v) \notin j_T^{\alpha,\omega}(p(v))$, that is, the edge connecting $v$ and $p(v)$ is not contracted to a point by $j_T^{\alpha,\omega}$.\end{proposition}

\begin{proof}
From the assumption $\omega \in U_{I,\alpha}$
and Defn.\ \ref{Def: alpha generic delta} of $U_{I,\alpha}$, we have $\theta_{T,p(v)}^{\alpha,\omega}(e_{v'}) \neq 0$, and so $\theta_{T,p(v)}^{\alpha, \omega} \neq \theta_{T,v}^{\alpha,\omega}$ by \eqref{eq_discrete_flow_2}.\end{proof}

\begin{definition}
We denote by $F^{\alpha,\omega}$ the union of all the images of the trees $T^{\circ}$ by the maps $j_T^{\alpha,\omega}$ for  
$T \in \cT_I^\eta$:  
\begin{equation} 
\label{eq:the_forest}    
F^{\alpha,\omega} \coloneqq \bigcup_{T \in \cT_I^\eta} j_T^{\alpha,\omega}(T^{\circ})     \subset \cM_{\R}\,.\end{equation}
We view $F^{\alpha,\omega}$ as a graph embedded in $\cM_{\R}$. 
Note that we have $\alpha \in F^{\alpha,\omega}$ because
$\alpha$ is the common image by the maps 
$j_T^{\alpha,\omega}$ of the roots of the trees 
$T \in \cT_I^\eta$. 
\end{definition}

\subsubsection{Scattering diagrams via flow tree maps}
Now we are ready to state our main theorem of this section, that allows us to describe scattering diagrams in terms of flow tree maps. This is the technical heart of this paper.

\begin{theorem} \label{thm_flow_tropical}
Fix a $(I,\eta)$-generic point 
$\alpha \in e_I^{\perp} \subset \cM_\R$ as in Defn.\ \ref{def_J_generic} and a small enough generic perturbation $\omega \in U_{I,\alpha} \cap U^\eta$ of $\eta$ as in \S \ref{Small generic perturbations of the skew-symmetric bilinear form}. Let $J \subset I$ be a nonempty index set, and $x \in e_J^{\perp}$ a $(J,\eta)$-generic point such that $x \in F^{\alpha, \omega}$ and  the line segment 
$(x+\R \iota_{e_J}\omega) \cap F^{\alpha,\omega}$ is not a point. Let 
$\sigma \in \Wall_{\Supp(\fh)}$ be a wall containing $x$ and with normal vector
$n_\sigma=e_J$. Then for every consistent $(\cN^+,\fh)$-scattering diagram 
$\phi$ constructed from initial data $I_{\phi,n}$ that satisfies $I_{\phi,n}=0$ if $n \notin \{e_i\}_{i\in I}$, we have 
\begin{equation} \label{eq_thm}
    \phi(\sigma)_{e_J}
    = A_J^{x,\omega}\left( (I_{\phi,e_i})_{i \in J} \right)
\end{equation}
where $\phi(\sigma)_{e_J} \in \fh_{e_J}$ is the component of $\phi(\sigma) \in \fh$ in $\fh_{e_J}$, and $A_J^{x,\omega}$ is the flow tree map
with initial point $x$ as in Defn.\ \ref{def_final_A}.
\end{theorem}
\begin{proof}
The proof is done by induction on the cardinality of the subset
$J \subset I$. For the initial step of the induction, let $J$ be a singleton, that is, \ $J=\{i\}$ for some
$i\in J$. Then by Lemma \ref{lem_cardinal_T}, 
$\cT_J$ consists of a single tree $T$, with one root and one leg connected by a single edge. Therefore by item (1)
of Defn.\ \ref{def_flow_tree_map}, the map 
$A_{J,T}^{x,\omega} \colon \fg_{e_i} \rightarrow \fg_{e_i}$
is the identity map. Hence, $A_J^{x,\omega}(I_{\phi,e_i})=I_{\phi,e_i}$. On the other hand, let $\sigma$ be a wall with $n_{\sigma}=e_i$. 
As $e_i$ does not admit any non-trivial decomposition as a sum of elements of $\Supp(\fh) \subset \cN_e^+$, it follows from the algorithmic construction of scattering diagrams
from initial data reviewed in the proof of
Proposition \ref{prop_scattering_from_initial_explicit} that 
$\phi(\sigma)_{e_i}=I_{\phi,e_i}$ for every consistent $(\cN^+,\fh)$-scattering diagram $\phi$. Therefore, we conclude
$\phi(\sigma)_{e_i} = A_J^{x,\omega}(I_{\phi,e_i})$,
and hence the initial step of the induction.

For the induction step, let $J \subset I$ of cardinality $|J|>1$. 
We assume that Theorem \ref{thm_flow_tropical} holds for every $J' \subset I$ 
with $|J'|<|J|$. Let $\sigma \in \Wall_{\Supp(\fh)}$ be a wall such that $n_\sigma=e_J$ and 
let $x \in F^{\alpha,\omega} \cap \sigma$ be a $(J,\eta)$-generic point such that 
$(x+\R \iota_{e_J} \omega) \cap F^{\alpha,\omega}$ is a non-trivial line segment.

In the remaining part of the section, we show that the statement of the theorem holds for $J$, $x$, $\sigma$ in the following four steps:
\begin{itemize}
    \item[Step I:] We define a set of \emph{relevant joints} $\cJ$, and show in Lemma \ref{lem_propagation} that if two walls contained in $e_J^{\perp}$ intersect along any joint that is not relevant, then the elements of the Lie algebra $\fh$ associated to these walls are the same. This enables us to partition the hyperplane $e_J^{\perp}$ into regions where any wall in a given region has the same associated element of the Lie algebra, which we denote by $\phi_{i-1,i} \in \fh_{e_J}$ in \eqref{Eq: phiis}, for $i\in \{1,\ldots,k\}$, and $\phi_{k,\infty} \in \fh_{e_J}$ in \eqref{Phikinfty}.
    \item[Step II:] Using the genericity of $\omega$, we prove Lemma \ref{lem_no_triple_intersection} and we obtain \eqref{Eq: difference}, expressing the difference $\phi_{i-1,i} - \phi_{i,i+1}$ in terms of some Lie brackets. On the other hand, using that $\omega$ is close enough to $\eta$, we prove that $\phi_{k,\infty}=0$. 
      \item[Step III:] Using the consistency condition around the relevant joints and the induction hypothesis, we determine explicitly the Lie brackets appearing in \eqref{Eq: difference}.
      \item[Step IV:] Using the explicit expression obtained in Step III  for the difference in \eqref{Eq: difference}, we obtain the expression \eqref{Phikinfty} for $\phi(\sigma)_{e_J}$. This, together with $\phi_{k,\infty} = 0$ shown in Step II, concludes the proof.  
\end{itemize}
We expand each of these steps in the remaining part of this section.

\textbf{Step I}: 
We define the set $\cJ$ of \emph{relevant joints}: a joint 
$\fj \in \fS_{\Supp(\fh)}$, that is, a codimension $2$ cone of 
the cone complex $\fS_{\Supp(\fh)}$ is \emph{relevant} if there exists a subindex set $J' \subset J$ with 
$\fj \subset e_{J'}^\perp \cap e_J^{\perp}$ and $\eta(e_{J'},e_J) \neq 0$. Note that the point $x$ is not contained in a relevant joint because of the assumption that
$x$ is $(J,\eta)$-generic. 
Let $0=t_0 <t_1 <\dots <t_k$ be an increasing sequence of positive real numbers, such that the intersection points of the half-line 
$x+\R_{\geq 0} \iota_{e_J} \omega$ with relevant joints 
$\fj \in \cJ$ correspond to points
\begin{equation}
    \label{Points xi}
x_i=x+t_i \iota_{e_J} \omega \subset e_J^{\perp} \subset \cM_\R, 
\end{equation}
for $i \in \{ 1, \ldots,k\}$, as illustrated in Figures \ref{Fig: walls2d}
and \ref{Fig: walls2d3d}.

\begin{figure}
\center{\scalebox{.3}{\input{walls2d.pspdftex}}}
\caption{Joints on the wall $e_J^{\perp}$, the perturbation $\iota_{e_J}\omega$ of $\iota_{e_J}\eta$, and the half line $x+\RR_{\geq 0}\iota_{e_J}\omega$.}
\label{Fig: walls2d}
\end{figure}

\begin{figure}
\center{\scalebox{.25}{\input{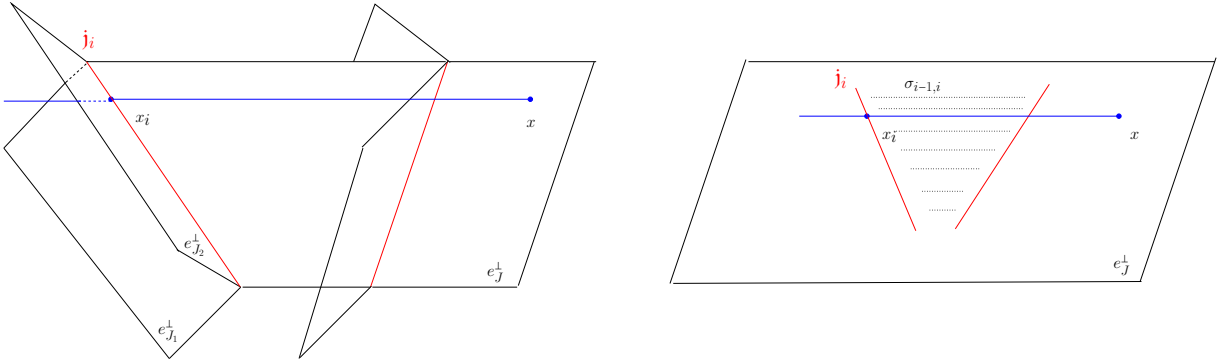}}}
\caption{Walls intersecting along joints on the left and the wall $\sigma_{i-1,i} \subset e_J^{\perp}$ on the right.}
\label{Fig: walls2d3d}
\end{figure}

\begin{lemma} \label{lem_propagation}
Let $\phi$ be a consistent $(\cN^+,\fh)$-scattering diagram, such that $I_{\phi,n}=0$ if 
$n \notin \{e_i\}_{i\in J}$. Let $\sigma_1, \sigma_2 \in \Wall_{\Supp(\fh)}$ such that 
$n_{\sigma_1} =n_{\sigma_2}=e_J$. 
Assume that the intersection $\sigma_1 \cap \sigma_2$ is a joint not belonging to $\cJ$. Then
we have 
$\phi(\sigma_1)_{e_J}=\phi(\sigma_2)_{e_J}$.
\end{lemma}

\begin{proof}
By consistency of $\phi$ applied around the joint 
$\sigma_1 \cap \sigma_2$, the difference $\phi(\sigma')_{e_J}-\phi(\sigma)_{e_J}$
is an element of $\fh_{e_J}$ equal to a sum of iterated Lie brackets in the elements $\phi(\fd_k)$, where $\fd_k \in \Wall_{\Supp(\fh)}$
are the walls containing $\sigma_1 \cap \sigma_2$ apart from 
$\sigma_1$ and $\sigma_2$. As by assumption 
$\sigma_1 \cap \sigma_2 \notin \cJ$, 
for every such wall $\fd_k$, we have either 
$n_{\fd_k}=e_{J'}$ for $J' \subset I$ not contained in $J$, or $n_{\fd_k}=e_{J'}$
with $J' \subset J$ and $\eta(e_J,e_{J'})=0$.
If 
$J' \subset I$ is not contained in $J$, then $[\fh_{e_{J'}},\fh] \cap 
\fh_{e_J}=\{0\}$ and so in this case the wall
$\fd_k$ does not contribute non-trivially to the sum of iterated Lie brackets.
If $J' \subset J$ and $\eta(e_{J'},e_J)=0$, then $\eta(e_{J'},n)=0$ and $[e_{J'},\fh_n]=0$ for every $n \in N$ such that
$e_J=e_{J'}+n$, and so also in this case the wall $\fd_k$ does not contribute to the sum of iterated Lie brackets. We conclude that
$\phi(\sigma_1)_{e_J}-\phi(\sigma_2)_{e_J}=0$.
\end{proof}

By Lemma \ref{lem_propagation}, for any $i \in \{ 1,\ldots,k\}$, if 
$\sigma_1, \sigma_2$ are two walls 
with $n_{\sigma_1}=n_{\sigma_2}=e_J$
such that
$\sigma_1 \cap (x+\R_{\geq 0} \iota_{e_J}\omega)$
and $\sigma_2 \cap (x+\R_{\geq 0} \iota_{e_J}\omega)$ are non-trivial line segments contained in  $x+[t_{i-1},t_{i}] \iota_{e_J}\omega$, then $\phi(\sigma_1)_{e_J}=\phi(\sigma_2)_{e_J}$. We denote by \begin{equation}
\label{Eq: phiis}
    \phi_{i-1,i} \in \fh_{e_J}
\end{equation}
this common value. Note that $\phi(\sigma)_{e_J}=\phi_{0,1}$.
Similarly, for every walls $\sigma_1, \sigma_2$ with $n_{\sigma_1}=n_{\sigma_2}=e_J$
such that 
$\sigma_1 \cap (x+\R_{\geq 0} \iota_{e_J}\omega)$
and $\sigma_2 \cap (x+\R_{\geq 0} \iota_{e_J}\omega)$ are non-trivial line segments contained in  $x+[t_k,\infty) \iota_{e_J} \omega$, we have $\phi(\sigma_1)_{e_J}=\phi(\sigma_2)_{e_J}$, and we denote by
\begin{align}
\label{Phikinfty}
    \phi_{k,\infty} \in \fh_{e_J}
\end{align}
this common value.

\textbf{Step II:}
In this step, we show that the differences between 
$\phi_{i-1,i}$ and $\phi_{i,i+1}$ have the form given by 
\eqref{Eq: difference}, and  we prove that $\phi_{k,\infty}=0$.

\begin{lemma} \label{lem_no_triple_intersection}
Let $\omega \in U_{I,\alpha}$ as in Defn.\ \ref{Def: alpha generic delta}.
Let $J=J_1 \sqcup \dots \sqcup J_s$ be a partition of $J$ in $s$ subsets such that $x_i \in e_{J_1}^{\perp} \cap \dots \cap e_{J_s}^{\perp}$. 
Then, we have
$s \leq 2$.
\end{lemma}

\begin{proof}
If $s\geq 3$, then writing
$J_1'=J_1$, $J_2'=J_2$ and $J_3'=\bigcup_{k=3}^s J_k$, we have 
$J=J_1' \sqcup J_2' \sqcup J_3'$ and $x_i \in e_{J_1'}^{\perp} 
\cap e_{J_2'}^{\perp}\cap e_{J_3'}^{\perp}$. Thus, it is enough to prove that the case $s=3$ cannot happen.

So we assume by contradiction that there exists a partition 
$J=J_1 \sqcup J_2 \sqcup J_3$
such that 
$x_i \in e_{J_1}^{\perp} 
\cap e_{J_2}^{\perp}\cap e_{J_3}^{\perp}$. 
As we are assuming that 
$(x+\R \iota_{e_J}\omega) \cap F^{\alpha,\omega}$
is a non-trivial line segment, there exists a tree $T \in \cT_I^\eta$ and an edge $E$ of $T$ such that, denoting by
 $v$ the vertex of $T$ incident to $E$ on the path from $E$ to the
leaves,  $x$ is in the interior of $j_T^{\alpha,\omega}(E)$ and 
the charge $e_v$
as in Defn.\ \ref{Def: charge} is given by $e_v=e_J$.

\begin{figure}
\center{\scalebox{.4}{\input{T11.pspdftex}}}
\caption{A tree $\tilde{T}$ as in the proof of Lemma \ref{lem_no_triple_intersection}.}
\label{Fig:T1}
\end{figure}

We choose a tree $T_{12} \in \cT_{J_1 \sqcup J_2}$ 
such that, denoting by $v_{12}$ the child of the root of
$T_{12}$, we have
$e_{v_{12}'}=e_{J_1}$ and $e_{v_{12}''}=e_{J_2}$.
We also choose a tree $T_3 \in \cT_{J_3}$.
We construct a new tree $\tilde{T} \in \cT_I^\eta$ from $T$, 
$T_{12}$ and $T_3$
as follows (see Figure \ref{Fig:T1}). First, let 
$\overline{T}$ be the tree obtained by removing from $T$
all the edges and vertices descendant from $v$, so that $v$ becomes a leaf of $\overline{T}$. Then, we obtain
$\tilde{T}$ by gluing the three trees $\overline{T}$, $T_{12}$, and $T_3$: we identify the leaf $v$ of $\overline{T}$ with the roots of
$T_{12}$ and $T_3$. 
We still denote by $v$ the vertex of $\tilde{T}$ where 
$\overline{T}$, $T_{12}$ and $T_3$ are glued together, and by
$E$ the edge of 
$\tilde{T}$ incident to $v$ on the path from $v$ to the root.
We have $e_v= e_J$, and we label 
$v'$ and $v''$ the children of $v$ so that
$e_{v'}=e_{J_1}+e_{J_2}$, $e_{v''}=e_{J_3}$, and 
$(v')'$ and $(v')''$ the children of $v'$ so that
$e_{(v')'}=e_{J_1}$ and $e_{(v')''}=e_{J_2}$.

By \eqref{eq:j_T_vertex}, we have $j_{\tilde{T}}(v)=\theta_{\tilde{T},v}^{\alpha,\omega}$ and it follows from Lemma \ref{lem_attractor} that $j_{\tilde{T}}(v) \in 
(e_{J_1}+e_{J_2})^{\perp}\cap e_{J_3}^{\perp}$. 
As we also have
$j_{\tilde{T}}(E) \subset x+\R \iota_{e_J} \omega$, we deduce that 
$j_{\tilde{T}}(v)$ is the intersection point of the line $x+\R \iota_{e_J}$ with 
$(e_{J_1}+e_{J_2})^{\perp} \cap e_{J_3}^{\perp}$
and so
$j_{\tilde{T}}(v)=x_i$. As we are assuming $x_i \in e_{J_1}^{\perp} 
\cap e_{J_2}^{\perp} \cap e_{J_3}^{\perp}$, we have in particular 
$\theta^{\alpha,\omega}_{\tilde{T},v}(e_{J_1})=0$, so \ 
$\theta^{\alpha,\omega}_{\tilde{T},v}(e_{(v')'})=0$,
in contradiction with our assumption that 
$\omega \in U_{I,\alpha}$ and Defn.\ \ref{Def: alpha generic delta} of
$U_{I,\alpha}$.
\end{proof}

For every $i\in \{1,\dots,k\}$, we pick a relevant joint $\fj_i \in \cJ$ containing the point $x_i$.
By consistency of $\phi$ around the joint 
$\fj_i$, the difference 
$\phi_{i-1,i}-\phi_{i,i+1}$ can be
written in terms of the walls containing 
$\fj_i$ as a sum of iterated Lie brackets.
By Lemma \ref{lem_no_triple_intersection},
$\phi_{i-1,i}-\phi_{i,i+1}$ only receives contributions from two-terms decompositions $e_J = e_{J_1}+e_{J_2}$. Denote by $P_{\fj_i}$ the set of
$\{J_1,J_2\}$ with $J_1, J_2\subset J$, 
$J=J_1\sqcup J_2$, $\fj_i \subset e_{J_1}^{\perp} \cap e_{J_2}^{\perp}$,
and $\eta(e_{J_1},e_{J_2}) \neq 0$. 
Then, we have 
\begin{equation}
    \label{Eq: difference}
    \phi_{i-1,i}-\phi_{i,i+1}
=\sum_{\{J_1,J_2\} \in P_{\fj_i}} g^{\fj_i}_{J_1,J_2}
\end{equation}
where $g_{J_1,J_2}^{\fj_i}$ is a scalar multiple of a Lie bracket produced 
by the walls contained in the hyperplanes $e_{J_1}^{\perp}$
and $e_{J_2}^{\perp}$ and intersecting 
along the joint $\fj_i$.
It follows from Lemma \ref{lem_no_triple_intersection} that
one can compute each term $g_{J,J'}^{\fj_i}$
as if the only walls intersecting along
the joint $\fj_i$ were contained in the hyperplanes 
$e_{J_1}^{\perp}$, 
$e_{J_2}^\perp$ and $e_J^\perp$.
The precise form of 
$g_{J_1,J_2}^{\fj_i}$ is given in Lemma \ref{lem_g} below.

\begin{proposition} \label{prop_reduction_initial}
For $\omega \in U^\eta$, we have $\phi_{k,\infty}=0$.
\end{proposition}

\begin{proof}
As the set of walls $\Wall_{\Supp(\fh)}$ is finite, there exists a wall $\sigma_\infty \in \Wall_{\Supp(\fh)}$
such that $n_{\sigma_\infty}=e_J$ and 
$x+t \iota_{e_J} \omega \subset \sigma_\infty$ for 
$t$ large enough, as illustrated in Figures \ref{Fig: walls2d}.
As $\sigma_\infty$ is a cone in 
$\cM_\R$, this last condition is only possible if $\iota_{e_J} \omega \in \sigma_\infty$. 
As $\Supp(\fh) \subset \cN_e^+$, it follows from the assumption $\omega \in U^\eta$ and from the Defn.\ \ref{def_U_eta} of $U^\eta$ that $\iota_{e_J} \eta \in \sigma_\infty$:
indeed the condition that $\omega(e_J,n)$ has the same sign as $\eta(e_J,n)$ for all 
$n \in \cN_e^+$ exactly means that there are no hyperplane $n^{\perp}$ with 
$n \in \cN_e^+$ and separating the points 
$\iota_{e_J}\omega$ and $\iota_{e_J}\omega$.
Therefore, we have by Proposition 
\ref{prop_initial_scattering} that 
$\phi(\sigma_\infty)_{e_J}=I_{\phi,e_J}$.
But we are assuming that 
$I_{\phi,n}=0$ if $n \notin \{e_i\}_{i\in I}$
and $|J|>1$, so $I_{\phi,e_J}=0$.
We conclude that $\phi_{k,\infty}
=\phi(\sigma_\infty)_{e_J}=0$.
\end{proof}

\textbf{Step III:}
In this step, we apply the consistency condition for $\phi$ around the joint
$\fj_i$ through the point $x_i=x+t_i \iota_{e_J}\omega$ to compute the quantities
$g_{J_1,J_2}^{\fj_i}$ appearing in \eqref{Eq: difference}.

 We denote by 
$\sigma_{i-1,i}$ (resp.\ 
$\sigma_{i,i+1}$) the wall containing 
$\fj_i$ such that $n_{\sigma_{i-1,i}}=e_J$
and $\sigma_{i-1,i} \subset \fj_i -\R_{\geq 0} \iota_{e_J}\omega$
(resp.\ $\sigma_{i,i+1} \subset \fj_i+\R_{\geq 0} \iota_{e_J} \omega$),
as illustrated in Figures \ref{Fig: walls2d} and \ref{Fig: walls2d3d}.
We have $\phi(\sigma_{i-1,i})=\phi_{i-1,i}$ and 
$\phi(\sigma_{i,i+1})=\phi_{i,i+1}$.

Let $\{J_1,J_2\} \in P_{\fj_i}$. We denote by 
$\fd_1^{in}$, $\fd_2^{in}$, $\fd_1^{out}$ and $\fd_2^{out}$ 
the walls containing $\fj_i$ such that $n_{\fd_1^{in}}
=n_{\fd_1^{out}}=e_{J_1}$,
$n_{\fd_2^{in}}=n_{\fd_2^{out}}=e_{J_2}$,
\begin{equation} \fd_1^{in} \subset \fj_i + \R_{\geq 0} \iota_{e_{J_1}}\omega\,,\,\, 
\fd_2^{in} \subset \fj_i + \R_{\geq 0} \iota_{e_{J_2}}\omega
\end{equation}
\begin{equation}
\fd_1^{out} \subset 
\fj_i - \R_{\geq 0} \iota_{e_{J_1}}\omega\,, \,\,\fd_2^{out} \subset 
\fj_i - \R_{\geq 0} \iota_{e_{J_2}}\omega\,.
\end{equation}
By Lemma \ref{lem_no_triple_intersection}, there are no non-trivial decomposition $e_{J_1}=\sum_{j=1}^s n_j$ with
$n_j \in \cN_e^+$ and $\fj_i \subset \cap_{j=1}^s n_j^{\perp}$, and so it
follows from the consistency of 
$\phi$ around $\fj_i$ that $\phi(\fd_1^{out})_{e_{J_1}}=\phi(\fd_1^{in})_{e_{J_1}}$.
Similarly, we have $\phi(\fd_2^{out})_{e_{J_2}}=\phi(\fd_2^{in})_{e_{J_2}}$.

\begin{figure}
\center{\scalebox{.6}{\input{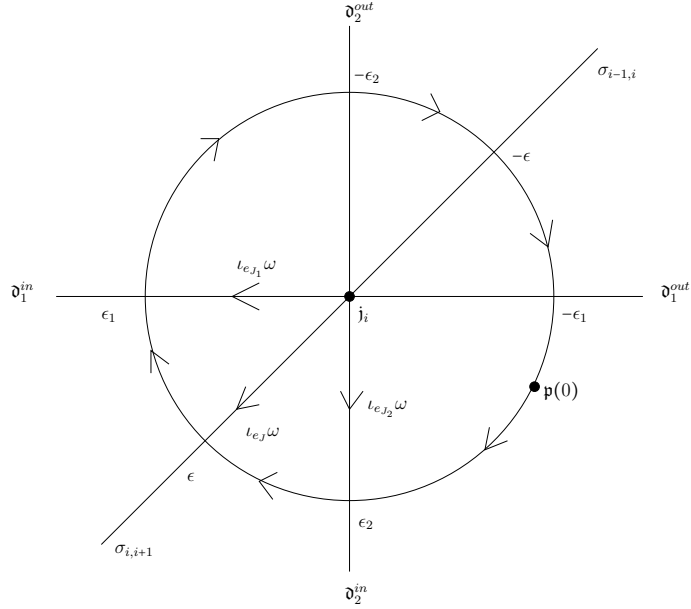}}}
\caption{Consistency around the joint $\fj_i$.}
\label{Fig:Circle}
\end{figure}

\begin{lemma}\label{lem_g}
Let $\fp \colon [0,1] \rightarrow \cM_\R$ be a $\fh$-generic oriented loop around $\fj_i$ intersecting successively 
$\fd_2^{in}$, $\sigma_{i,i+1}$, $\fd_1^{in}$, $\fd_2^{out}$, 
$\sigma_{i-1,i}$, $\fd_1^{out}$ (see Figure \ref{Fig:Circle}). 
Then, we have 
\begin{equation}
    g_{J_1,J_2}^{\fj_i} = - \sgn(\omega(e_{J_1},e_{J_2}))
    [\phi(\fd_1^{in})_{e_{J_1}},\phi(\fd_2^{in})_{e_{J_2}}]\,. 
\end{equation}
\end{lemma}

\begin{proof}
Denote by $\epsilon_1$ (resp.\ $\epsilon_2$ and
$\epsilon$) the sign of the derivative of 
$t \mapsto -\fp(t)(e_{J_1})$
(resp.\ $-\fp(t)(e_{J_2})$ and $-\fp(t)(e_J)$)
at the intersection point of 
$\fp$ with $\fd_1^{in}$ (resp.\ $\fd_2^{in}$ and 
$\sigma_{i,i+1}$). 
According to \eqref{eq:path_automorphism}, we have 
\begin{equation}
    \Psi_{\fp,\phi}=
    e^{-\epsilon_1 \phi(\fd_1^{in})_{e_{J_1}}}
    e^{-\epsilon \phi_{i-1,i}}
    e^{-\epsilon_2 \phi(\fd_2^{in})_{e_{J_2}}}
    e^{\epsilon_1 \phi(\fd_1^{in})_{e_{J_1}}}
    e^{\epsilon \phi_{i,i+1}}
    e^{\epsilon_2 \phi(\fd_2^{in})_{e_{J_2}}}\,.
\end{equation}
Therefore, the consistency of $\phi$ around $\fj_i$
implies 
\begin{equation}
    \epsilon(\phi_{i,i+1}-\phi_{i-1,i}) + \epsilon_1 \epsilon_2 
    [\phi(\fd_1^{in})_{e_{J_1}}, \phi(\fd_2^{in})_{e_{J_2}}]=0
\end{equation}
and so 
\begin{equation}
    g_{J_1,J_2}^{\fj_i} = \epsilon \epsilon_1 \epsilon_2 [\phi(\fd_1^{in})_{e_{J_1}}, \phi(\fd_2^{in})_{e_{J_2}}] \,.
\end{equation}
We show $- \sgn(\omega(e_{J_1},e_{J_2})) = \epsilon \epsilon_1\epsilon_2$ in the remaining part of the proof. 
We work in the plane transverse to $\fj_i$ spanned by 
$\iota_{e_{J_1}}\omega$, $\iota_{e_{J_2}}\omega$, and we view 
$(e_{J_1},e_{J_2})$ as coordinates on this plane. 
Up to smoothly deforming 
$\fp$, one can assume that $\fp$ intersects 
$\fp_2^{in}$
(resp.\ $\sigma_{i,i+1}$ and 
$\fd_1^{in}$) at the point $\fj_i+\iota_{e_{J_2}}\omega$
(resp.\ $\fj_i+\iota_{e_J}\omega$ and 
$\fj_i+\iota_{e_{J_1}}\omega$), which has coordinates 
$(-\omega(e_{J_1},e_{J_2}),0)$
(resp.\ $(-\omega(e_{J_1},e_{J_2}), \omega(e_{J_1},e_{J_2}))$
and $(0, \omega(e_{J_1},e_{J_2}))$).

By definition, $\epsilon_2$ is minus the sign of variation of 
the coordinate $e_{J_2}$ when $\fp$ crosses $\fd_2^{in}$. 
When $\fp$ goes from $\fj_i+\iota_{e_{J_2}}\omega$ to 
$\fj_i+\iota_{e_J}\omega$, the variation of the coordinate 
$e_{J_2}$ is $\omega(e_{J_1},e_{J_2})$, and so 
$\epsilon_2=-\sgn(\omega(e_{J_1},e_{J_2}))$.
Similarly, one checks that 
$\epsilon = -\sgn(\omega(e_{J_1},e_{J_2}))$
and $\epsilon_1= -\sgn(\omega(e_{J_1},e_{J_2}))$.
\end{proof}

\begin{lemma}
\label{Lem: rewriting commutators}
We can apply the induction hypothesis to $J_1, \fd_1^{in}$, $x_i$ and 
$J_2$, $\fd_2^{in}$, $x_i$. Hence, 
\begin{align}
\label{Eq: Ind 1}
    \phi(\fd_1^{in}) & =A_{J_1}^{x_i,\omega}((I_{\phi,e_j})_{j\in J_1}) \,,\\
    \label{Eq: Ind 2}
 \phi(\fd_2^{in}) & =A_{J_2}^{x_i,\omega}((I_{\phi,e_j})_{j\in J_2})  \,.
\end{align}
\end{lemma}
\begin{proof}
To show we can apply the induction hypothesis to $J_1, \fd_1^{in}$, $x_i$ and 
$J_2$, $\fd_2^{in}$, $x_i$, we need to show that:
\begin{itemize}
    \item[(i)] the point $x_i$ 
    is $(J_1,\eta)$-generic and 
$(J_2,\eta)$-generic, 
    \item[(ii)] the intersections
$(x_i + \R \iota_{e_{J_1}}\omega) \cap F^{\alpha,\omega}$ and
$(x_i + \R \iota_{e_{J_1}}\omega) \cap F^{\alpha,\omega}$ 
are non-trivial line segments.
\end{itemize}

To prove (i), first note that $x_i \in \fj_i \subset e_{J_1}^{\perp} \cap e_{J_2}^{\perp}$.
If there were $J_1' \subsetneq J_1$ such that $x_i \in e_{J_1'}^{\perp}$, then, writing 
$J_1=J_1' \sqcup J_1''$, one would have 
$e_J=e_{J_1'}+e_{J_1''}+e_{J_2}$ and $x_i \in e_{J_1'}^\perp \cap e_{J_1''}^\perp \cap e_{J_2}^\perp$, in contradiction with
Lemma \ref{lem_no_triple_intersection}. Therefore, 
$x_i$ is $(J_1,\eta)$-generic. Exchanging the roles of 
$J_1$ and $J_2$, this also proves that 
$x_i$ is $(J_2,\eta)$-generic.

To prove (ii), we follow the same logic as in the proof of Lemma 
\ref{lem_no_triple_intersection}.
As we are assuming that 
$x+\R \iota_{e_J} \omega \subset F$ is a non-trivial line segment, there exists a tree $T \in \cT_I^\eta$ and an edge $E$ of $T$ such that, denoting by
 $v$ the vertex of $T$ incident to $E$ on the path from $E$ to the
leaves, $x$ is in the interior of $j_T^{\alpha,\omega}(E)$
 and $e_v=e_J$.

\begin{figure}
\center{\scalebox{.4}{\input{T2.pspdftex}}}
\caption{A tree $\tilde{T}$ as in the proof of Lemma \ref{Lem: rewriting commutators}.}
\label{Fig:T2}
\end{figure}

We choose trees $T_1 \in \cT_{J_1}$ and $T_2 \in \cT_{J_2}$.
We construct a new tree $\tilde{T} \in \cT_I^\eta$ from 
$T$, $T_1$ and $T_2$
as follows (see Figure \ref{Fig:T2}). First, let 
$\overline{T}$ be the tree obtained by removing from $T$
all the edges and vertices descendant from $v$, so that $v$ becomes a leaf of $\overline{T}$. Then, we obtain
$\tilde{T}$ by gluing the three trees $\overline{T}$, $T_1$, and $T_2$: we identify the leaf $v$ of $\overline{T}$ with the roots of
$T_1$ and $T_2$. 
We still denote by $v$ the vertex of $\tilde{T}$ where 
$\overline{T}$, $T_1$ and $T_2$ are glued together and by $E$ the edge of 
$\tilde{T}$ incident to $v$ on the path from $v$ to the root.
We have $e_v=e_{J}$ and we label $v'$ and $v''$ the children of $v$ so that
$e_{v'}=e_{J_1}$ and $e_{v''}=e_{J_2}$. Let $E'$ (resp.\ $E''$) be the edge
of $\tilde{T}$ connecting $v$ to $v'$ (resp.\ $v''$).
We have $j_{\tilde{T}}(e_v)=\theta_{\tilde{T},v}^{\alpha,\omega}$, and so 
by Lemma \ref{lem_attractor}, $j_{\tilde{T}}(e_v)
\in e_{J_1}^{\perp}\cap e_{J_2}^{\perp}$. As we also have
$j_{\tilde{T}}(E) \subset x+\R \iota_{e_J} \omega$, we deduce that 
$j_{\tilde{T}}(v)=x_i$. We conclude that $j_{\tilde{T}}(E') \subset 
(x_i+\R \iota_{e_{J_1}}\omega)\cap F^{\alpha,\omega}$ and  $j_{\tilde{T}}(E'') \subset 
(x_i+\R \iota_{e_{J_2}}\omega)\cap F^{\alpha,\omega}$. 
By Proposition \ref{prop_not_contracted}, 
$j_{\tilde{T}}(E')$ and $j_{\tilde{T}}(E'')$ are non-trivial line segments and and hence the proof of (ii) follows.
\end{proof}

Thus, we can rewrite Lemma \ref{lem_g} as 
\begin{equation}
    g_{J_1,J_2}^{\fj_i}=-\sgn(\omega(e_{J_1},e_{J_2}))
    [A_{J_1}^{x_i,\omega}, A_{J_2}^{x_i,\omega}]((I_{\phi,e_j})_{j\in J})\,.
\end{equation}
By Defn.\ \ref{def_final_A} of the flow tree maps as sum over trees, 
this can be rewritten as
\begin{equation} \label{eq_1}
    g_{J_1,J_2}^{\fj_i}=- \sum_{T_1 \in \cT_{J_1}^{\eta}} \sum_{T_2 \in \cT_{J_2}^\eta}
    \sgn(\omega(e_{J_1},e_{J_2}))
    [A_{J_1,T_1}^{x_i,\omega}, A_{J_2,T_2}^{x_i,\omega}]((I_{\phi,e_j})_{j\in J})\,.
\end{equation}

\textbf{Step IV:} As a final step, we show that 
\begin{equation}
    \label{Eq: final step}
    \phi(\sigma)_{e_J}=\phi_{k,\infty}+A_J^{x,\omega}((I_{\phi,e_j})_{j \in J})\,.
\end{equation}

To prove \eqref{Eq: final step}, first observe that summing the equations \eqref{Eq: difference} side by side for all values $i\in \{ 1, \ldots,k \}$ we obtain $\phi(\sigma)_{e_J}=\phi_{k,\infty} 
+ \sum_{i=1}^k \sum_{\{J_1,J_2\} \in P_{\fj_i}}g_{J_1,J_2}^{\fj_i}$.
Then, using \eqref{eq_1}, we get
\begin{equation} \label{eq_2}
    \phi(\sigma)_{e_J}=\phi_{k,\infty} 
- \sum_{i=1}^k \sum_{\{J_1,J_2\} \in P_{\fj_i}}
 \sum_{T_1 \in \cT_{J_1}^{\eta}} \sum_{T_2 \in \cT_{J_2}^\eta}
    \sgn(\omega(e_{J_1},e_{J_2})) 
    [A_{J_1,T_1}^{x_i,\omega}, A_{J_2,T_2}^{x_i,\omega}]((I_{\phi,e_j})_{j\in J})\,.
\end{equation}
On the other hand, we have
$A_J^{x,\omega}=\sum_{T \in \cT_J^\eta} A_{J,T}^{x,\omega}$
by Defn.\ \ref{def_final_A}, and so, using
Defn.\ \ref{def_flow_tree_map}: 
\begin{equation} \label{eq_2_2}
\sum_{T\in \cT_J^\eta} A_{J,T}^{x, \omega}
    =-\sum_{T \in \cT_J^{\eta}}\frac{\sgn 
   (x(e_{v_T'}))
   + \sgn ( \omega(e_{v_T'}, e_{v_T''}))}{2} \, [A_{J,T,v_T'}^{x, \omega},
   A_{J,T,v_T''}^{x, \omega}] \,,
\end{equation}
where $v_T$ is the child of the root of the tree $T$, and 
$v_T', v_T''$ are the children of $v_T$.

Comparing \eqref{eq_2} and \eqref{eq_2_2}, it remains to show that 
\begin{align} \label{eq_3}
&\sum_{i=1}^k \sum_{\{J_1,J_2\} \in P_{\fj_i}}
 \sum_{T_1 \in \cT_{J_1}^{\eta}} \sum_{T_2 \in \cT_{J_2}^\eta}
    \sgn(\omega(e_{J_1},e_{J_2})) 
    [A_{J_1,T_1}^{x_i,\omega}, A_{J_2,T_2}^{x_i,\omega}] \\
&=  \sum_{T \in \cT_J^{\eta}}\frac{\sgn 
   (x(e_{v_T'}))
   + \sgn ( \omega(e_{v_T'}, e_{v_T''}))}{2} \, [A_{J,T,v_T'}^{x, \omega},
   A_{J,T,v_T''}^{x, \omega}]\,. \label{eq_4}  
\end{align}

Given $T \in \cT_J^\eta$ and writing $J_1=J_{T,v_T'}$ and 
$J_2=J_{T,v_T''}$, we obtain a tree $T_1 \in \cT_{J_1}$
(resp.\ $T_2 \in \cT_{J_2}$)
by considering the subtree of $T$ made of $v_T$
and its descendant through the child $v_{T}'$ (resp.\ 
$v_T''$) (see Figure \ref{Fig:T3}). If the contribution of 
$T$ in \eqref{eq_4} is nonzero, we have in fact 
$T_1 \in \cT_{J_1}^\eta$ and $T_2 \in \cT_{J_2}^\eta$. 
We claim that $x(e_{J_1})$ and $\omega(e_{J_1},e_{J_2})$
are of the same sign if and only if
the intersection point of the line 
$x+\R \iota_{e_J} \omega$ with $e_{J_1}^{\perp} \cap e_{J_2}^{\perp}$
is contained in the half-line $x+\R_{\geq 0} \iota_{e_J} \omega$.
Indeed, the intersection point of the line $x+\R \iota_{e_J} \omega$ with $e_{J_1}^{\perp} \cap e_{J_2}^{\perp}$
is the point 
\begin{equation}
x-\frac{x(e_{J_1})}{\omega(e_J,e_{J_1})} \iota_{e_J} \omega\,.
\end{equation}
Thus, if $\sgn (x(e_{J_1}))+\sgn(\omega(e_{J_1},e_{J_2})) \neq 0$, 
the intersection point of the line 
$x+\R \iota_{e_J} \omega$ with $e_{J_1}^{\perp} \cap e_{J_2}^{\perp}$
is equal 
$x_i$ for some $1 \leq i\leq k$ such that $\{J_1,J_2\} \in \fj_i$, and we have \begin{equation} x_i= x-\frac{x(e_{J_1})}{\omega(e_J,e_{J_1})}
\iota_{e_J} \omega = \theta_{T,v}^{x,\omega}\,. 
\end{equation}
Then, it follows from
Defn.\ \ref{def_flow_tree_map} 
and \ref{def_final_A} that
$A_{J_1,T_1}^{x_i,\omega}=A_{J,T,v_T'}^{x,\omega}$ and 
$A_{J_2,T_2}^{x_i,\omega}=A_{J,T,v_T''}^{x,\omega}$.
Conversely, for every 
$1 \leq i\leq k$ and $\{J_1,J_2\} \in P_{\fj_i}$, every 
$T_1 \in \cT_{J_1}^\eta$ and $T_2 \in \cT_{J_2}^\eta$
are obtained in this way. Hence, \eqref{Eq: final step} follows.

\begin{figure}
\center{\scalebox{.8}{\input{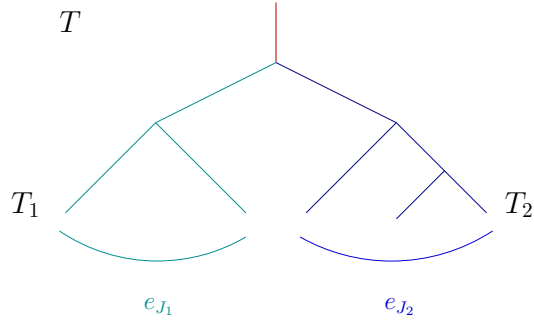}}}
\caption{Trees $T_1$, $T_2$ and $T$.}
\label{Fig:T3}
\end{figure}

From \eqref{Eq: final step} together with Proposition \ref{prop_reduction_initial}, we obtain $\phi(\sigma)_{e_J}
=A_J^{x,\omega}((I_{\phi,e_j})_{j\in J})$ and 
so Theorem \ref{thm_flow_tropical} holds for 
$J$, $\sigma$ and $x$. Hence, this concludes our
proof of Theorem \ref{thm_flow_tropical}.
\end{proof}

\subsection{The flow tree formula for scattering diagrams}
\label{section_proof_thm_2}

\begin{definition} \label{def:gamma_generic}
A point $\tau \in \gamma^{\perp} \subset M_\R$ is $\gamma$-generic if for every 
$\gamma' \in N$, $\theta(\gamma')=0$ implies that $\gamma'$ is collinear with 
$\gamma$.
\end{definition}

\begin{lemma}\label{lem_generic}
Let $\tau \in \gamma^{\perp} \subset M_\R$ be a $\gamma$-generic point as in 
Defn.\ \ref{def:gamma_generic}.
Then, the image $\alpha \coloneqq q(\tau) \in e_I^{\perp} \subset \cM_\R$ 
of $\tau$ by the map $q \colon M_\R \rightarrow \cM_\R$ given by \eqref{Eq: q}
is $(I,\eta)$-generic as in Defn.\ \ref{def_J_generic}.
\end{lemma}

\begin{proof}
Assume by contradiction that $\alpha$ is not $(I,\eta)$-generic, which means
by Defn.\ \ref{def_J_generic}
that there exists a tree
$T \in \cT_I^{\eta}$ such that $\alpha(e_{v'}) = 0$, 
where $v$ is the child of the root of $T$. Thus, we have
$\tau(p(e_{v'}))=0$, that is, \ $\tau \in p(e_{v'})^{\perp}$, and so the condition that 
$\tau$ is $\gamma$-generic implies by Defn.\ \ref{def:gamma_generic}
that 
$p(e_{v'})$ is collinear with $\gamma=p(e_I)$. Recalling that 
$e_v=e_I$, this implies that
$\eta(e_{v'},e_v)=\eta(e_{v'},e_I)=
\langle p(e_{v'}), p(e_I) \rangle =0$, in contradiction with the assumption that $T \in \cT_I^\eta$ and the Defn.\ \ref{def_tree_eta} of $\cT_I^\eta$.
\end{proof}

Let $\tau \in \gamma^\perp$ be a $\gamma$-generic point as 
in Defn.\ \ref{def:gamma_generic}.
By Lemma \ref{lem_generic}, the point $\alpha \coloneqq q(\tau) \in e_I^{\perp}$
is $(I,\eta)$-generic. 
Therefore, by Proposition \ref{Prop: def_V} 
the set $U_{I,\alpha}$ of $(I,\alpha)$-generic skew-symmetric bilinear form is open and dense in $\bigwedge^2 \cM_\R$, and 
for every $\omega \in U_{I,\alpha}$ the flow tree map
    $A^{\alpha, \omega}_I \colon \prod_{i\in I} \fh_{e_i}
    \rightarrow \fh_e$
is defined by Defn.\ \ref{def_final_A}.

Finally, we arrive at our main theorem of this section, the flow tree formula for scattering diagrams I:
\begin{theorem} \label{thm_flow_tree_formula_scattering}
Let $\fd \in \Wall_{\Supp(\fg)}$ be a wall in $M_\R$ and 
$\Gamma=\{\gamma_i\}_{i\in I}\in \mult(N^+)$ a multiset of elements of $N^+$
such that $\fd \subset \gamma^{\perp}$, where $\gamma=\sum_{i \in I} \gamma_i$.
Let $\tau \in \fd$ be a $\gamma$-generic point and 
$\alpha \coloneqq q(\tau)\in \cM_\R$ the image of $\tau$ by the map 
$q: M_{\RR} \to \cM_{\RR}$ as in \eqref{Eq: q}.
For every small enough generic perturbation $\omega \in U_{I,\alpha} \cap U^{\eta}$ of $\eta$ as in \ref{Small generic perturbations of the skew-symmetric bilinear form}, 
the map
$F_\Gamma^{\fg,\fd}$ in \eqref{Eq: FgDeltaGamma} is given by the ``flow tree formula for scattering diagrams I'':
\begin{equation}
    F_\Gamma^{\fg,\fd} = \frac{1}{\prod_{n \in N^+}m_\Gamma(n)!}\,\,\hat{A}^{\alpha,\omega}_I\,.
\end{equation}
where $\hat{A}^{\alpha,\omega}_I$ is as in Defn.\ \ref{def_specialization} the specialization of the flow tree map
$A^{\alpha,\omega}_I$ defined in Defn.\ \ref{def_final_A}.
\end{theorem}

\begin{proof}
Let $\fe \in \Wall_{\Supp(\fh)}$ be a wall in $\cM_\R$ containing $q(\fd)$ such that $\fe \subset e_I^{\perp}$. In particular, we have $\alpha \in \fe$.
By Theorem \ref{thm_D_D_gamma}, we have 
\begin{equation} \label{eq_0_1}
F_{\Gamma}^{\fg,\fd}
=\frac{1}{\prod_{n \in N^+}m_\Gamma(n)!}
\hat{F}_{\Gamma_e}^{\fh, \fe}\,.
\end{equation}
On the other hand, as $\alpha$ is $(I,\eta)$-generic by Lemma \ref{lem_generic}, we can apply Theorem \ref{thm_flow_tropical} for $J=I$, $\sigma=\fe$, $x=\alpha$, and we obtain 
\begin{equation} \label{eq_0_2}
    F_{\Gamma_e}^{\fh, \fe}=A_I^{\alpha,\omega}\,.
\end{equation}
The result follows from \eqref{eq_0_1} and \eqref{eq_0_2}.
\end{proof}

We provide also a variant of the flow tree formula for scattering diagrams, the flow tree formula for scattering diagrams II, which involves perturbing the points in $\mathcal{M}_{\RR}$ rather than the skew-symmetric bilinear form, as in Theorem \ref{thm_flow_tree_formula_scattering}.

Note that from Proposition \ref{prop_generic_point}
that the set $V_{I,\eta}$ of $\beta \in e_I^{\perp} \subset \cM_\R$
such that $\beta$ is $(I,\eta)$-generic and $\eta$ is $\beta$-generic is open and dense in $e_I^{\perp}$.
For every $\beta \in V_{I,\eta}$, we define the flow tree maps 
$A_I^{\beta,\eta} \colon \prod_{i\in I} \fh_{e_i} \rightarrow \fh_e$
as in Defn.\ \ref{def_final_A}
and its specialization $\hat{A}_I^{\beta,\eta} \colon \prod_{n \in\overline{\Gamma}} \fg_n \rightarrow \fg_\gamma$ as in Defn.\ \ref{def_specialization}. 
For every $\beta \in V_{I,\eta}$, we define 
$F^{\beta,\eta}$ as $F^{\alpha,\omega}$
in \ref{eq:the_forest} and replacing $\alpha$
with $\beta$, and $\omega$ with $\eta$.
We also define $V^\alpha \subset e_I^{\perp}$ as the set of 
$\beta \in e_I^{\perp}$ such that there exists a wall 
$\fe \in \Wall_{\Supp(\fh)}$ with $\fe \subset e_I^{\perp}$
which contains both $\alpha$ and $\beta$. We have $\alpha \in V^\alpha$
and $V^{\alpha}$ is an open neighborhood of $\alpha$ in 
$e_I^{\perp}$. We say that $\beta$ is a \emph{small enough generic perturbation} of 
$\alpha$ in $e_I^{\perp}$ if 
\begin{equation}
\beta \in V_{I,\alpha} \cap V^{\alpha}\,.
\end{equation}

\begin{theorem} \label{thm_flow_tropical_2}
Fix a $(I,\eta)$-generic point 
$\alpha \in e_I^{\perp} \subset \cM_\R$ as in Defn.\ \ref{def_J_generic} and a small enough generic perturbation $\beta \in V_{I,\alpha} \cap V^\alpha$ of $\alpha$ in $e_I^{\perp}$. Let $J \subset I$ be a nonempty index set, and $x \in e_J^{\perp}$ a $(J,\eta)$-generic point such that $x \in F^{\beta, \eta}$ and  the line segment 
$(x+\R \iota_{e_J}\omega) \cap F^{\beta,\eta}$ is not a point. Let 
$\sigma \in \Wall_{\Supp(\fh)}$ be a wall containing $x$ and with normal vector
$n_\sigma=e_J$. 
Then for every consistent $(\cN^+,\fh)$-scattering diagram 
$\phi$ such that $I_{\phi,n}=0$ if $n \notin \{e_i\}_{i\in I}$, we have 
\begin{equation} \label{eq_thm_2}
    \phi(\sigma)_{e_J}
    = A_J^{x,\eta}\left( (I_{\phi,e_i})_{i \in J} \right)
\end{equation}
\end{theorem}
\begin{proof}
The proof is analogous to the proof of 
Theorem \ref{thm_flow_tropical}, with $\alpha$, $\omega$ replaced respectively by 
$\beta$, $\eta$, and with an extra simplification in
Proposition \ref{prop_reduction_initial}: 
for $t$ positive large enough, $x+t \iota_{e_J}\eta$ is contained in a wall 
$\sigma_\infty$, which thus necessarily contains 
$\iota_{e_J}\eta$ and so $\phi_{k,\infty}=0$ follows from Proposition 
\ref{prop_initial_scattering}.
\end{proof}

\begin{theorem} \label{thm_flow_tree_formula_scattering_2}
Let $\fd \in \Wall_{\Supp(\fg)}$ be a wall in $M_\R$ and 
$\Gamma=\{\gamma_i\}_{i\in I}\in \mult(N^+)$ a multiset of elements of $N^+$
such that $\fd \subset \gamma^{\perp}$, where $\gamma=\sum_{i \in I} \gamma_i$.
Let $\tau \in \fd$ be a $\gamma$-generic point and 
$\alpha \coloneqq q(\tau)\in \cM_\R$ the image of $\tau$ by the map 
$q: M_{\RR} \to \cM_{\RR}$ as in \eqref{Eq: q}.
For every small enough generic perturbation 
$\beta \in V_{I,\alpha} \cap V^{\alpha}$ of $\alpha$ in 
$e_I^{\perp}$, the universal map 
$F_\Gamma^{\fg,\fd}$ in \eqref{Eq: FgDeltaGamma} is given by by the ``flow tree formula for scattering diagrams II'':
\begin{equation}
    F_\Gamma^{\fg,\fd} = \frac{1}{\prod_{n\in N^+} m_\Gamma(n)!} 
    \hat{A}_I^{\beta,\eta} \,.
\end{equation}
\end{theorem}

\begin{proof}
Let $\fe \in \Wall_{\Supp(\fh)}$ be a wall in $\cM_\R$ such that $\fe \subset e_I^{\perp}$ and containing both $q(\fd)$
and $\beta$.
By Theorem \ref{thm_D_D_gamma}, we have 
\begin{equation} \label{eq_0_0_1}
F_{\Gamma}^{\fg,\fd}
=\frac{1}{\prod_{n \in N^+}m_\Gamma(n)!}
\hat{F}_{\Gamma_e}^{\fh, \fe}\,.
\end{equation}
On the other hand, as $\alpha$ is $(I,\eta)$-generic by Lemma \ref{lem_generic}, we can apply Theorem \ref{thm_flow_tropical_2} for $J=I$, $\sigma=\fe$, $x=\beta$, and we obtain 
\begin{equation} \label{eq_0_0_2}
    F_{\Gamma_e}^{\fh, \fe}=A_I^{\beta,\eta}\,.
\end{equation}
The result follows from \eqref{eq_0_0_1} and \eqref{eq_0_0_2}.
\end{proof}

\begin{remark} \label{rem_gps}
We compare briefly the passage 
from scattering diagrams in $N$ to scattering diagrams in $\cN$ and the perturbation of scattering diagrams introduced in \cite{MR2667135}.
Using our notations, the perturbation of  \cite{MR2667135} 
consists in replacing the hyperplanes $\gamma_i^{\perp} 
=\{ \theta \in M_\R \,|\, \theta(\gamma_i)=0\}$ by the affine hyperplanes 
$\{ \theta \in M_\R \,|\, \theta(\gamma_i)=\epsilon_i \}$ where 
$\epsilon_i \in \R$ are generic perturbation parameters.
On the other hand, denoting by $K$ the kernel of $p \colon \cN \rightarrow N$, we obtain by duality a surjective map $\pi \colon \cM_\R \rightarrow K^\vee_\R$,
where $K^\vee_{\R} \coloneqq \Hom(K,\R)$. 
We claim that our scattering diagram in $\cM_\R$ is a universal family of perturbed scattering diagrams in the sense of \cite{MR2667135}.
Indeed, fixing $\epsilon \in K^\vee_\R$
is equivalent to fixing the perturbation parameters 
$\epsilon_i$ of  \cite{MR2667135}, and the intersections of our scattering diagram in $\cM_\R$ with the fibers $\pi^{-1}(\epsilon)$ are essentially the perturbed
scattering diagrams of \cite{MR2667135}.

The embedded trees $j_T^{\beta,\eta}(T^{\circ})$ used in the proof of 
Theorem \ref{thm_flow_tree_formula_scattering_2} 
are all contained in the fiber $\pi^{-1}(\pi(\beta))$ of 
$\pi$. Indeed, all edges have directions of the form $\iota_{e_v}\eta$,
and so for every $k \in K$, we have 
$\iota_{e_v}\eta(k)=\eta(e_v,k)=0$ because $\eta$ is the pullback of 
$\langle-,-\rangle$ by $p$.
Therefore, these embedded trees viewed inside $\pi^{-1}(\pi(\beta))$
essentially coincide with the tropical curves contained in the perturbed scattering diagrams considered 
in \cite{MR2667135} (see also \cite{mandel2015scattering, MR4131036}).

By contrast, the embedded trees $j_T^{\alpha,\omega}(T^\circ)$
used in the proof of Theorem \ref{thm_flow_tree_formula_scattering}
are not contained in a given fiber of $\pi$ in general:
one cannot use the perturbed scattering diagrams in the sense of
\cite{MR2667135} and it is essential to work with 
scattering diagrams in $\cM_\R$.
\end{remark}

\section{The flow tree formula for DT invariants}
\label{section_DT_invariants}

In \S \ref{section_quivers}-\S \ref{section_dt_quivers}, we review the definition of DT invariants 
of quivers with potentials. In \S \ref{section_dt_flow_tree}, we state the flow tree formula,
which computes DT invariants in terms of a smaller set of attractor DT invariants.
We prove the flow tree formula for DT invariants in \S \ref{section_dt_proof} by 
applying the flow tree formula for scattering diagrams to the stability scattering diagram 
introduced by Bridgeland in \cite{MR3710055}.

\subsection{Quivers with potentials}
\label{section_quivers}
A \emph{quiver} $Q$ is a finite oriented graph. A \emph{potential} $W \in \C Q$ for $Q$ is a finite linear combination of oriented cycles of $Q$ in the path algebra $\C Q$ of $Q$. We assume that $Q$ does not contain oriented $2$-cycles and we denote by $Q_0$ the set of vertices of $Q$, and set $N \coloneqq \Z^{Q_0}$, with dual $M_\R \coloneqq \Hom(N,\R)$, and
\begin{equation} N^+ \coloneqq \NN^{Q_0} \backslash \{0\} \subset N\,.
\end{equation}
\begin{definition}
A \emph{representation} $E$ of $Q$ is a finite-dimensional left-module over the path algebra 
$\C Q$, that is,\ the data of a finite-dimensional $\C$-vector space 
$E_i$ for each vertex $i\in Q_0$ and of a linear map $f_{\alpha} \colon E_i \rightarrow E_j$
for every arrow $\alpha \colon i \rightarrow j$ in $Q$. Every nonzero representation of $Q$ has a \emph{dimension vector} $(\dim E_i)_{i\in Q_0} \in N^+$. 
\end{definition}
\begin{definition}
Given $\gamma \in N^+$ and a stability parameter 
$\theta \in 
\gamma^{\perp}=\{\theta'\in M_\R|\, \theta'(\gamma)=0\}$, 
a representation $E$ of $Q$ of dimension vector $\gamma$ is $\theta$-\emph{semistable} 
(resp.\ $\theta$-\emph{stable}) if for every non-zero strict subrepresentation $F \subsetneq E$, we have $\theta(F) \leq 0$ (resp.\ $\theta(F)<0$).
\end{definition}
It is shown in \cite{MR1315461} that there exists a smooth quasiprojective variety $M_\gamma^{\theta-st}$
parametrizing isomorphism classes of $\theta$-stable representations of $Q$ of dimension vector $\gamma$, and a generally singular quasiprojective variety $M_\gamma^\theta$ parametrizing S-equivalence classes of 
$\theta$-semistable representations of $Q$ of dimension vector $\gamma$. A potential $W \in \C Q$ defines regular functions $\Tr(W)_\gamma^\theta$ on the moduli spaces $M_\gamma^{\theta}$ as follows: Given a representation $E=(E_i, f_\alpha)_{i,\alpha} \in M_\gamma^\theta$
and an a oriented cycle  $c=\alpha_r \dots \alpha_0$  in $Q$ starting and ending at the vertex 
$i_0 \in Q_0$, the composition 
\begin{equation} f_c \coloneqq f_{\alpha_r} \circ \dots \circ f_{\alpha_0}
\end{equation}
of the linear maps 
$f_{\alpha_i}$ along the arrows of the cycle is an endomorphism of $E_{i_0}$, and we define the evaluation of the function $\Tr(c)_\gamma^\theta$ on $E$ as being the trace of this endomorphism. More generally, $W$ is a linear combination
$\sum_k a_k c_k$ of oriented cycles $c_k$ and we define $\Tr(W)_\gamma^\theta$
by linearity, that is, $\Tr(W)_\gamma^\theta \coloneqq \sum_k a_k \Tr(c_k)_\gamma^\theta$.

\subsection{DT invariants of quivers with potentials and flow trees}
\label{section_dt_quivers}
Let $(Q,W)$ be a quiver with potential, $\gamma \in N^+$
a dimension vector and
$\theta \in \gamma^{\perp} \subset M_\R$ a stability parameter. 
We assume that 
$\theta$ is $\gamma$-generic in the sense that $\theta(\gamma')=0$ implies $\gamma'$ collinear with $\gamma$. Then, the (refined) Donaldson-Thomas invariant 
of $(Q,W)$ for the dimension vector $\gamma$ and the stability parameter 
$\theta$ is a Laurent polynomial 
\begin{equation} \Omega_\gamma^\theta(y,t) \in \Z[y^{\pm},t^{\pm}]
\end{equation}
in two variables 
$y$ and $t$, and with integer coefficients.
In the ideal case where $M_\gamma^\theta$ is smooth and the critical locus of 
$\Tr(W)_\gamma^\theta$ is non-degenerate, $\Omega_\gamma^\theta(y,t)$ 
coincides with the 
(signed symmetrized) Hodge polynomial of the critical locus of 
$\Tr(W)_\gamma^\theta$. In general, the singularities of 
$M_\gamma^\theta$ and the degeneracy of the critical locus
require respectively the use of the theory of perverse sheaves \cite{MR751966}
and of the theory of vanishing cycles \cite{deligne1973groupes}. We will in a moment review the definition of 
$\Omega_\gamma^\theta(y,t)$ following the approach of 
\cite{MR4000572, MR4132957} and referring to \cite{MR4132957} for technical details.

We define the DT sheaf $\mathcal{D}\mathcal{T}_\gamma^\theta$ on 
$M_\gamma^{\theta}$ by 
\begin{equation} \mathcal{D}\mathcal{T}_\gamma^\theta = \begin{cases} 
      \phi_{\Tr(W)_\gamma^\theta}(IC_{M_\gamma^{\theta}}) & \mathrm{if} ~ M_\gamma^{\theta-st} \neq \emptyset \\
      0 & \mathrm{elsewise},
   \end{cases}
\end{equation}
where $IC_{M_d^{\theta}}$ denotes the intersection cohomology sheaf on $M_\gamma^\theta$
and $\phi_{\Tr(W)_\gamma^\theta}$ 
is the vanishing cycle functor 
defined by the function 
\begin{equation}\Tr(W)_\gamma^\theta \colon 
M_\gamma^\theta \longrightarrow \C\,.\end{equation} 
The cohomological DT invariant $DT_\gamma^\theta$ is then defined as the cohomology of the DT sheaf: 
\begin{equation} DT_\gamma^\theta \coloneqq H^{*}(M_\gamma^\theta, \mathcal{D}\mathcal{T}_\gamma^\theta)\,.\end{equation} 
By Saito's theory of mixed Hodge modules \cite{MR1047415}, the graded vector space $DT_\gamma^\theta$ is naturally endowed with
a (monodromic) mixed Hodge structure, and so in particular with an increasing weight 
filtration $\mathbf{W}$ and a decreasing Hodge filtration $\mathbf{F}$.
The Hodge-Deligne numbers of $DT_\gamma^\theta$ are 
\begin{equation}
h^{p,q} \coloneqq \sum_{i \in \Z} (-1)^i \dim Gr_{\mathbf{F}}^p Gr^{\mathbf{W}}_{p+q} H^i(M_\gamma^\theta, \mathcal{D}\mathcal{T}_\gamma^\theta)\,,
\end{equation}
where $Gr_{\mathbf{F}}^{*}$ and $Gr^{\mathbf{W}}_{*}$
are the graded pieces of the filtrations 
$\mathbf{F}$ and $\mathbf{W}$.
The (refined) DT invariant 
$\Omega_\gamma^\theta(y,t)$ is by definition a Laurent polynomial with 
coefficients the Hodge-Deligne numbers of $DT_\gamma^\theta$:
\begin{equation} \label{Eq: just omega}
\Omega_\gamma^\theta(y,t) \coloneqq \sum_{p,q} h^{p,q} y^{p+q}t^{p-q} \in \Z[y^{\pm}, t^{\pm}]\,.
\end{equation}

The flow tree formula is more naturally formulated in terms of 
the rational DT invariants $\overline{\Omega}_\gamma^\theta(y,t) \in \Q(y,t)$
defined by 
\begin{equation} \label{eq_multicover}
    \overline{\Omega}_{\gamma}^\theta(y,t) 
    \coloneqq \sum_{\substack{\gamma' \in N^+\\ 
    \gamma=k \gamma',\, k\in \Z_{\geq 1}}} \frac{1}{k} \frac{y-y^{-1}}{y^k - y^{-k}} \Omega_{\gamma'}^{\theta}(y^k,t^k) \,.
\end{equation}
It is proved in \cite{davison2015donaldson, MR4132957} that the dependence on $\theta$ of the invariants $\Omega_\gamma^\theta(y,t)$ is given by the wall-crossing formula of Joyce-Song and Kontsevich-Soibelman, and that the invariants 
$\overline{\Omega}_{\gamma}^\theta(y,t)$ coincide with those  
previously defined in \cite{JoyceSong, kontsevich2008stability} using the motivic Hall algebra.

\subsection{Attractor invariants and the flow tree formula}
\label{section_dt_flow_tree}
In this section we state our main result, the flow tree formula in Theorem \ref{main_thm}, which expresses the DT invariants in terms of a smaller subset of invariants, referred to as attractor invariants and defined as follows.

Let 
$\langle -,- \rangle \colon N \times N \rightarrow \Z$ be the skew-symmetric bilinear form defined by 
\begin{equation} \langle \gamma,\gamma' \rangle \coloneqq \sum_{i,j \in Q_0} 
(a_{ij}-a_{ji}) \gamma_i \gamma_j'\,,\end{equation} 
where 
$a_{ij}$ is the number of arrows in $Q$ from the vertex $i$ to the vertex $j$.

\begin{definition} \label{def_attractor_invariant}
For every $\gamma \in N^+$, the \emph{rational attractor invariant}
$\overline{\Omega}_\gamma^{*}(y,t)$ is defined by 
\begin{equation} \overline{\Omega}_\gamma^{*}(y,t) \coloneqq \overline{\Omega}_\gamma^{\theta_\gamma}(y,t)\,,\end{equation}
where $\overline{\Omega}_\gamma^{\theta_\gamma}(y,t)$ is as in \eqref{eq_multicover}, and $\theta_\gamma$ is a small $\gamma$-generic perturbation of the 
\emph{attractor point} $\langle \gamma, - \rangle \in M_\R$.
\end{definition}
\begin{remark}

The definition \ref{def_attractor_invariant} of rational attractor invariants is independent of the choice of the small 
$\gamma$-generic perturbation (see \cite[Theorem 3.1]{mozgovoy2020attractor}):
indeed, if there is a wall of marginal stability associated to a decomposition 
$\gamma=\gamma'+\gamma'$ passing through the attractor point 
$\langle \gamma,-\rangle$, then $\langle \gamma,\gamma' \rangle=0$
and so $\overline{\Omega}_\gamma^\theta(y,t)$ does not jump through this wall according to the wall-crossing formula. Replacing $\overline{\Omega}_\gamma^{\theta_\gamma}(y,t)$ in Defn.\ \ref{def_attractor_invariant} by $\Omega_\gamma^{\theta_\gamma}(y,t)$ in \eqref{Eq: just omega}, we obtain the definition of an \emph{attractor invariants}, which are related to rational attractor invariants via the formula \eqref{eq_multicover}. In what follows, we  often make use of the rational attractor invariants, which suit better to wall-crossing computations.
\end{remark}
By iteration of the wall-crossing formula, the DT invariants 
$\overline{\Omega}_\gamma^\theta(y,t)$ for any $\gamma$-generic stability parameter $\theta \in \gamma^{\perp}$ can be expressed in terms of the attractor invariants 
$\overline{\Omega}_\gamma^{*}$ by a formula of the form
\begin{equation} \label{eq_reconstruction}
    \overline{\Omega}_\gamma^\theta(y,t) = \sum_{r\geq 1} \sum_{\substack{\{\gamma_i\}_{1\leq i\leq r}\\ \sum_{i=1}^r\gamma_i = \gamma}} \frac{1}{|\Aut(\{\gamma_i\}_i)|} 
    F_r^{\theta}(\gamma_1,\dots,\gamma_r) \prod_{i=1}^r \overline{\Omega}_{\gamma_i}^{*}(y,t)\,,
\end{equation}
where the second sum is over the multisets $\{\gamma_i\}_{1\leq i\leq r}$ with $\gamma_i \in N$ and $\sum_{i=1}^r \gamma_i=\gamma$. Here, the denominator $|\Aut(\{\gamma_i\}_i)|$ is the order of the symmetry group of $\{\gamma_i\}$: if $m_{\gamma'}$ is the number of times that $\gamma' \in N$ appears in $\{\gamma_i\}_i$, then 
$|\Aut(\{\gamma_i\}_i)|=\prod_{\gamma'\in N}m_{\gamma'}!$.
The coefficients
$F_r^{\theta}(\gamma_1,\dots,\gamma_r)$ are universal in the sense that they depend 
on $(Q,W)$ only through the skew-symmetric form 
$\langle -,-\rangle$ on $N$.
The flow tree formula gives an explicit formula for coefficients 
$F_r^{\theta}(\gamma_1,\dots,\gamma_r)$ as a sum over binary trees.
We state the flow tree formula in Theorem \ref{main_thm} after introducing some notation.

Let $\gamma_1,\dots,\gamma_r \in N$ 
such that $\sum_{i=1}^r \gamma_i=\gamma$. 
As in \eqref{Eq: p}-\eqref{Eq: q}, we set 
$I \coloneqq \{1, \dots, r\}$ and we introduce a rank $r$ free abelian group $\cN=\bigoplus_{i\in I}\Z e_i$, along with the map $p: \cN \rightarrow N$ as in \eqref{Eq: p}
and the map $q \colon M_\R \rightarrow \cM_\R=\Hom(\cN,\R)$ defined as in \eqref{Eq: q}.
We also define a skew-symmetric bilinear form $\eta \in \bigwedge^2 \cM$ on $\cN$ by $\eta(e_i,e_j) \coloneqq \langle \gamma_i,\gamma_j\rangle$, and consider the image $\alpha$ of the stability parameter 
$\theta$ by $q$:
\begin{equation}
    \alpha \coloneqq q(\theta) \in \cM_\R\,.
\end{equation}
By Lemma \ref{lem_generic} the assumption that 
$\theta$ is $\gamma$-generic implies that $\alpha$ is 
$(I,\eta)$-generic and so we can consider a small enough  generic perturbation 
$\omega \in U_{I,\alpha} \cap U^\eta$ of $\eta$ as in 
Defn.\ \ref{Def: alpha generic delta} and Defn.\ \ref{def_U_eta}.

In the following theorem we state our main result, \emph{the flow tree formula}, which provides an explicit description for the universal coefficient $F_r^{\theta}(\gamma_1,\dots,\gamma_r)$ that appears in the formula \eqref{eq_reconstruction} expressing the DT invariants $\overline{\Omega}_\gamma^{\theta}(y,t)$ in terms of the attractor invariants $\overline{\Omega}_{\gamma_i}^{*}(y,t)$.

\begin{theorem} \label{main_thm}
For every small enough generic perturbation $\omega \in U_{I,\alpha} \cap U^\eta \subset \bigwedge^2 \cM_\R$ of $\eta \in \bigwedge^2 \cM_\R$, the universal coefficients $F_r^{\theta}(\gamma_1,\dots,\gamma_r)$ in \eqref{eq_reconstruction}
are given by the 
\emph{flow tree formula}:
\begin{equation}
    F_r^\theta(\gamma_1,\dots,\gamma_r) 
    = \sum_{T \in \cT_r^{\eta}} \prod_{v \in V_T^\circ} \epsilon_{T,v}^{\alpha,\omega} \,\, \kappa(\eta(e_{v'},e_{v''}))\,,
\end{equation}
where the sum is over binary trees as in \S\ref{section_trees}, 
$\epsilon^{\alpha,\omega}_{T,v} \in \{0,1,-1\}$ is as in 
\eqref{eq:epsilon}
and \begin{equation}
    \label{Eq: kappa_2}
    \kappa(x)\coloneqq {(-1)^x} \cdot \frac{y^x-y^{-x}}{y-y^{-1}}
\end{equation}
for every $x\in \Z$.
\end{theorem}

The flow tree formula stated in Theorem \ref{main_thm} was conjectured 
by Alexandrov and Pioline in \cite{AlexandrovPioline}. The assumption $\omega \in U_{I,\alpha} \cap U^\eta$
in Theorem \ref{main_thm} makes precise and explicit the conditions ``small enough" and ``generic" which were left slightly vague in the statement of Theorem \ref{main_thm_intro}
given in the introduction and in
the original formulation of the conjecture in \cite{AlexandrovPioline}: $\omega \in U^\eta$ is the condition ``small enough", and 
$\omega \in U_{I,\alpha}$ is the condition ``generic".

We also prove a variant of the flow tree formula
recently
conjectured by Mozgovoy \cite{mozgovoy2021operadic}
in which one perturbs points  in $\mathcal{M}_{\RR}$ rather than the skew-symmetric form. 
Recall that we denote $e_I \coloneqq \sum_{i\in I}e_i$.
By Proposition \ref{prop_generic_point},
the set $V_{I,\eta}$ of $\beta \in e_I^{\perp} \subset \cM_\R$
such that $\beta$ is $(I,\eta)$-generic and $\eta$ is $\beta$-generic is open and dense in $e_I^{\perp}$.
Finally, we denote by $V^\alpha$ the open neighborhood of 
$\alpha$ in $e_I^{\perp}$ defined by:
$\beta \in V^{\alpha}$ if and only if for every 
$n \in \cN_e^+$ such that $\alpha(n)$ is nonzero, 
$\beta(n)$ is nonzero and of the same sign as 
$\alpha(n)$.

\begin{theorem} \label{main_thm_moz_gps}
For every small enough generic perturbation $\beta \in V_{I,\eta} \cap V^\alpha$ of $\alpha$ in  $e_I^{\perp}$, 
the universal coefficient $F_r^{\theta}(\gamma_1,\dots,\gamma_r)$ which appears in
the formula \eqref{eq_reconstruction} expressing the DT invariants $\overline{\Omega}_\gamma^{\theta}(y,t)$ in terms of the attractor invariants 
$\overline{\Omega}_{\gamma_i}^{*}(y,t)$ is given by:
\begin{equation}
    F_r^\theta(\gamma_1,\dots,\gamma_r) 
    = \sum_{T \in \cT_r^{\eta}} \prod_{v \in V_T^\circ} \epsilon_{T,v}^{\beta,\eta} \,\, \kappa(\eta(e_{v'},e_{v''}))\,,
\end{equation}
where the sum is over binary trees as in \S\ref{section_trees}, $\epsilon_{T,v}^{\alpha,\omega}$ is as in \eqref{eq:epsilon} and 
$\kappa$ is as in \eqref{Eq: kappa_2}.
\end{theorem}

In Theorem \ref{main_thm_moz_gps}, the assumption 
$\beta \in V_{I,\eta} \cap V^\alpha$ makes precise and explicit the 
expression ``small enough generic perturbation" used in the statement of 
Theorem \ref{main_thm_intro_moz_gps} given in the introduction:
$\beta \in V^\alpha$ is the condition ``small enough", and 
$\beta \in V_{I,\eta}$ is the condition ``generic".

\subsection{Proofs of Theorems \ref{main_thm} and 
\ref{main_thm_moz_gps} }
\label{section_dt_proof}
We derive the proof of the flow tree formula in Theorem
\ref{main_thm} (and of its variant in Theorem \ref{main_thm_moz_gps}), from the flow tree formula for scattering diagrams in Theorem
\ref{thm_flow_tree_formula_scattering} (and from its variant in Theorem \ref{thm_flow_tree_formula_scattering_2} respectively).
We do this by applying the latter formulas to the stability scattering diagram, which is a $(N^+,\fg)$-scattering diagram as in Defn.\ \ref{def_scattering_diagram}, introduced by 
Bridgeland. We roughly review its description here, and for details refer to \cite{MR3710055}.

Let $(Q,W)$ be a quiver with potential, and 
$\gamma \in N^+$ be a dimension vector. 
Define a $N^+$-graded Lie algebra over $\Q(y,t)$ by
\begin{equation}
\tilde{\fg} \coloneqq \bigoplus_{n \in N^+} \Q(y,t)z^n,
\end{equation}
where the Lie bracket 
$[-,-]$ is given by
\begin{equation} \label{eq:lie_bracket}
[z^{n_1}, z^{n_2}] \coloneqq \kappa(\langle n_1, n_2 \rangle) z^{n_1+n_2}\,.
\end{equation}
where $\kappa$ is as in \eqref{Eq: kappa_2}. 
Let $\delta \colon N \rightarrow \Z$ be an additive map such that 
$\delta(N^+) \subset \Z_{\geq 1}$. Then 
\begin{equation}
\tilde{\fg}^{>n(\gamma)}
\coloneqq \bigoplus_{\substack{n \in N^+ \\ \delta(n)>\delta(\gamma)}} \Q(y,t)z^n
\end{equation}
is a Lie ideal of 
$\tilde{\fg}$ and we consider the quotient Lie algebra 
\begin{equation}\fg \coloneqq \tilde{\fg}/\tilde{\fg}^{>\delta(\gamma)}\,, 
\end{equation}
which is finitely
$N^+$-graded. The support of $\fg$ is 
$\Supp(\fg)=\{n \in N^+\,|\, \delta(n) \leq \delta(\gamma)\}$.

For every wall $\fd \in \Wall_{\Supp(\fg)}$, pick a point $x_\fd \in \fd$
such that $x_\fd \notin \fd'$ for all $\fd' \in \Wall_{\Supp(\fg)}$
distinct from $\fd$. 
The \emph{stability scattering diagram}
\begin{equation}
\label{Eq: stability scattering}
\phi \colon \Wall_{\Supp(\fg)}  \longrightarrow \fg \end{equation}
is defined by 
\begin{equation}    \phi(\fd) \coloneqq  \sum_{\substack{k \geq 1\\\delta(kn_\fd) \leq \delta(\gamma) }} \overline{\Omega}_{kn_\fd}^{x_\fd}(y,t) z^{k n_\fd}\,,
\end{equation}
for every wall $\fd\in \Wall_{\Supp(\fg)}$,
where $\overline{\Omega}_{kn_\fd}^{x_\fd}(y,t)$ are rational DT invariants defined as in \eqref{eq_multicover}. The definition of $\phi$ is in fact independent of the choices of the points $x_\fd$: by the wall-crossing formula, the DT invariants 
$\Omega_n^\theta(y,t)$
with $\delta(n) \leq \delta(\gamma)$ 
do not jump as long as $\theta$ stays in the interior of a wall $\fd \in \Wall_{\Supp(\fg)}$.
The following key theorem is due to Bridgeland \cite[Thm 1.1]{MR3710055}: 
\begin{theorem}[(Bridgeland, 2016)]
\label{Thm: Bridgeland}
The stability scattering diagram is consistent. 
\end{theorem}

More precisely, the main results of \cite{MR3710055} are stated in terms of the rational DT invariants defined by Joyce-Song \cite{JoyceSong} using the motivic Hall algebra. The comparison with the rational DT invariants defined as in \eqref{eq_multicover} is established in
\cite[\S 6.7]{davison2015donaldson}. Moreover, \cite[Thm 1.1]{MR3710055} is only stated for the ``unsigned unrefined" invariants (virtual Euler characteristics), but the proof by applying an integration map to the Hall algebra scattering diagram immediately generalizes to the case of the ``signed refined" invariants (virtual signed Hodge polynomials)
(see also \cite[\S 7.1]{davison2019strong}).

By Theorem \ref{Thm: Bridgeland}, we can apply Theorems 
\ref{thm_flow_tree_formula_scattering} and \ref{thm_flow_tree_formula_scattering_2} to the stability scattering diagram $\phi$.
By Proposition \ref{prop_initial_scattering}, the initial data of 
$\phi$ are given by the attractor DT invariants: 
\begin{equation}I_{\phi,n}=\overline{\Omega}_n^{*}(y,t) z^n\,,
\end{equation}
for every 
$n\in N^+$ with $\delta(n) \leq \gamma$, and so
Theorems 
\ref{main_thm} and \ref{main_thm_moz_gps} follow.

\bibliographystyle{plain}
\bibliography{biblio1}

\begin{thebibliography}{10}

\bibitem{MR4170291}
Sergei Alexandrov, Jan Manschot, and Boris Pioline.
\newblock S-duality and refined {BPS} indices.
\newblock {\em Comm. Math. Phys.}, 380(2):755--810, 2020.

\bibitem{AlexandrovPioline}
Sergei Alexandrov and Boris Pioline.
\newblock Attractor flow trees, {BPS} indices and quivers.
\newblock {\em Adv. Theor. Math. Phys.}, 23(3):627--699, 2019.

\bibitem{MR4072224}
Sergei Alexandrov and Boris Pioline.
\newblock Black holes and higher depth mock modular forms.
\newblock {\em Comm. Math. Phys.}, 374(2):549--625, 2020.

\bibitem{MR3106506}
Murad Alim, Sergio Cecotti, Clay C\'{o}rdova, Sam Espahbodi, Ashwin Rastogi,
  and Cumrun Vafa.
\newblock B{PS} quivers and spectra of complete {$\mathcal{N}=2$} quantum field
  theories.
\newblock {\em Comm. Math. Phys.}, 323(3):1185--1227, 2013.

\bibitem{MR3268234}
Murad Alim, Sergio Cecotti, Clay C\'{o}rdova, Sam Espahbodi, Ashwin Rastogi,
  and Cumrun Vafa.
\newblock {$\mathcal{N}=2$} quantum field theories and their {BPS} quivers.
\newblock {\em Adv. Theor. Math. Phys.}, 18(1):27--127, 2014.

\bibitem{arguz2020higher}
H{\"u}lya Arg{\"u}z and Mark Gross.
\newblock The higher dimensional tropical vertex.
\newblock {\em arXiv preprint arXiv:2007.08347, to appear in Geometry and
  Topology}, 2020.

\bibitem{MR2567952}
Paul Aspinwall, Tom Bridgeland, Alastair Craw, Michael Douglas, Mark Gross,
  Anton Kapustin, Gregory Moore, Graeme Segal, Bal\'{a}zs Szendr\H{o}i, and
  Pelham Wilson.
\newblock {\em Dirichlet branes and mirror symmetry}, volume~4 of {\em Clay
  Mathematics Monographs}.
\newblock American Mathematical Society, Providence, RI; Clay Mathematics
  Institute, Cambridge, MA, 2009.

\bibitem{beaujard2020vafa}
Guillaume Beaujard, Jan Manschot, and Boris Pioline.
\newblock Vafa-{W}itten invariants from exceptional collections.
\newblock {\em Comm. Math. Phys.}, 385(1):101--226, 2021.

\bibitem{MR3036440}
Iosif Bena, Micha Berkooz, Jan de~Boer, Sheer El-Showk, and Dieter Van~den
  Bleeken.
\newblock Scaling {BPS} solutions and pure-{H}iggs states.
\newblock {\em J. High Energy Phys.}, (11):171, front matter + 36, 2012.

\bibitem{MR751966}
Alexander Be\u{\i}linson, Joseph Bernstein, and Pierre Deligne.
\newblock Faisceaux pervers.
\newblock In {\em Analysis and topology on singular spaces, {I} ({L}uminy,
  1981)}, volume 100 of {\em Ast\'{e}risque}, pages 5--171. Soc. Math. France,
  Paris, 1982.

\bibitem{MR4157555}
Pierrick Bousseau.
\newblock The quantum tropical vertex.
\newblock {\em Geom. Topol.}, 24(3):1297--1379, 2020.

\bibitem{MR2373143}
Tom Bridgeland.
\newblock Stability conditions on triangulated categories.
\newblock {\em Ann. of Math. (2)}, 166(2):317--345, 2007.

\bibitem{MR3710055}
Tom Bridgeland.
\newblock Scattering diagrams, {H}all algebras and stability conditions.
\newblock {\em Algebr. Geom.}, 4(5):523--561, 2017.

\bibitem{cps}
Michael Carl, Max Pumperla, and Bernd Siebert.
\newblock A tropical view of {L}andau-{G}inzburg models.
\newblock {\em arXiv preprint arXiv:2205.07753}, 2022.

\bibitem{cecotti2010r}
Sergio Cecotti, Andrew Neitzke, and Cumrun Vafa.
\newblock {R}-twisting and 4d/2d correspondences.
\newblock {\em arXiv preprint arXiv:1006.3435}, 2010.

\bibitem{MR3087917}
Sergio Cecotti and Cumrun Vafa.
\newblock Classification of complete {$N=2$} supersymmetric theories in 4
  dimensions.
\newblock In {\em Surveys in differential geometry. {G}eometry and topology},
  volume~18 of {\em Surv. Differ. Geom.}, pages 19--101. Int. Press,
  Somerville, MA, 2013.

\bibitem{MR4131036}
Man-Wai Cheung and Travis Mandel.
\newblock Donaldson-{T}homas invariants from tropical disks.
\newblock {\em Selecta Math. (N.S.)}, 26(4):Paper No. 57, 46, 2020.

\bibitem{davison2019strong}
Ben Davison and Travis Mandel.
\newblock Strong positivity for quantum theta bases of quantum cluster
  algebras.
\newblock {\em Invent. Math.}, 226(3):725--843, 2021.

\bibitem{davison2015donaldson}
Ben Davison and Sven Meinhardt.
\newblock {D}onaldson-{T}homas theory for categories of homological dimension
  one with potential.
\newblock {\em arXiv preprint arXiv:1512.08898}, 2015.

\bibitem{MR4132957}
Ben Davison and Sven Meinhardt.
\newblock Cohomological {D}onaldson-{T}homas theory of a quiver with potential
  and quantum enveloping algebras.
\newblock {\em Invent. Math.}, 221(3):777--871, 2020.

\bibitem{MR2511440}
Jan de~Boer, Sheer El-Showk, Ilies Messamah, and Dieter Van~den Bleeken.
\newblock Quantizing {$N=2$} multicenter solutions.
\newblock {\em J. High Energy Phys.}, (5):002, 63, 2009.

\bibitem{deligne1973groupes}
Pierre Deligne and Nicholas Katz.
\newblock Groupes de monodromie en g{\'e}om{\'e}trie alg{\'e}brique ({SGA VII},
  2) lecture notes in math. 340, 1973.

\bibitem{MR1792870}
Frederik Denef.
\newblock Supergravity flows and {D}-brane stability.
\newblock {\em J. High Energy Phys.}, (8):Paper 50, 40, 2000.

\bibitem{MR1952307}
Frederik Denef.
\newblock Quantum quivers and {H}all/hole halos.
\newblock {\em J. High Energy Phys.}, (10):023, 42, 2002.

\bibitem{MR1845419}
Frederik Denef, Brian Greene, and Mark Raugas.
\newblock Split attractor flows and the spectrum of {BPS} {D}-branes on the
  quintic.
\newblock {\em J. High Energy Phys.}, (5):Paper 12, 47, 2001.

\bibitem{MR2913216}
Frederik Denef and Gregory Moore.
\newblock Split states, entropy enigmas, holes and halos.
\newblock {\em J. High Energy Phys.}, (11):129, i, 152, 2011.

\bibitem{MR2480710}
Harm Derksen, Jerzy Weyman, and Andrei Zelevinsky.
\newblock Quivers with potentials and their representations. {I}. {M}utations.
\newblock {\em Selecta Math. (N.S.)}, 14(1):59--119, 2008.

\bibitem{donaldson1998gauge}
Simon Donaldson and Richard Thomas.
\newblock Gauge theory in higher dimensions.
\newblock In {\em The geometric universe ({O}xford, 1996)}, pages 31--47.
  Oxford Univ. Press, Oxford, 1998.

\bibitem{MR1360416}
Sergio Ferrara, Renata Kallosh, and Andrew Strominger.
\newblock {$N=2$} extremal black holes.
\newblock {\em Phys. Rev. D (3)}, 52(10):R5412--R5416, 1995.

\bibitem{MR3383167}
Sara Filippini and Jacopo Stoppa.
\newblock Block-{G}\"ottsche invariants from wall-crossing.
\newblock {\em Compos. Math.}, 151(8):1543--1567, 2015.

\bibitem{MR2219440}
Bartomeu Fiol.
\newblock The {BPS} spectrum of {$N=2$} {$SU(N)$} {SYM}.
\newblock {\em J. High Energy Phys.}, (2):065, 23, 2006.

\bibitem{ginzburg2006calabi-yau}
Victor Ginzburg.
\newblock {C}alabi-{Y}au algebras.
\newblock {\em arXiv preprint arXiv:0612139}, 2006.

\bibitem{MR2600995}
Mark Gross.
\newblock Mirror symmetry for {$\mathbb{P}^2$} and tropical geometry.
\newblock {\em Adv. Math.}, 224(1):169--245, 2010.

\bibitem{MR3758151}
Mark Gross, Paul Hacking, Sean Keel, and Maxim Kontsevich.
\newblock Canonical bases for cluster algebras.
\newblock {\em J. Amer. Math. Soc.}, 31(2):497--608, 2018.

\bibitem{gross2016theta}
Mark Gross, Paul Hacking, and Bernd Siebert.
\newblock Theta functions on varieties with effective anti-canonical class.
\newblock {\em Mem. Amer. Math. Soc.}, 278(1367):xii+103, 2022.

\bibitem{MR2662867}
Mark Gross and Rahul Pandharipande.
\newblock Quivers, curves, and the tropical vertex.
\newblock {\em Port. Math.}, 67(2):211--259, 2010.

\bibitem{MR2667135}
Mark Gross, Rahul Pandharipande, and Bernd Siebert.
\newblock The tropical vertex.
\newblock {\em Duke Math. J.}, 153(2):297--362, 2010.

\bibitem{MR2846484}
Mark Gross and Bernd Siebert.
\newblock From real affine geometry to complex geometry.
\newblock {\em Ann. of Math. (2)}, 174(3):1301--1428, 2011.

\bibitem{MR1941627}
Dominic Joyce.
\newblock On counting special {L}agrangian homology 3-spheres.
\newblock In {\em Topology and geometry: commemorating {SISTAG}}, volume 314 of
  {\em Contemp. Math.}, pages 125--151. Amer. Math. Soc., Providence, RI, 2002.

\bibitem{JoyceSong}
Dominic Joyce and Yinan Song.
\newblock A theory of generalized {D}onaldson-{T}homas invariants.
\newblock {\em Mem. Amer. Math. Soc.}, 217(1020):iv+199, 2012.

\bibitem{keel2019frobenius}
Sean Keel and Tony~Yue Yu.
\newblock The {F}robenius structure theorem for affine log {C}alabi-{Y}au
  varieties containing a torus.
\newblock {\em arXiv preprint arXiv:1908.09861}, 2019.

\bibitem{MR2484733}
Bernhard Keller.
\newblock Calabi-{Y}au triangulated categories.
\newblock In {\em Trends in representation theory of algebras and related
  topics}, EMS Ser. Congr. Rep., pages 467--489. Eur. Math. Soc., Z\"{u}rich,
  2008.

\bibitem{MR1315461}
Alastair King.
\newblock Moduli of representations of finite-dimensional algebras.
\newblock {\em Quart. J. Math. Oxford Ser. (2)}, 45(180):515--530, 1994.

\bibitem{KS}
Maxim Kontsevich and Yan Soibelman.
\newblock Affine structures and non-{A}rchimedean analytic spaces.
\newblock In {\em The unity of mathematics}, volume 244 of {\em Progr. Math.},
  pages 321--385. Birkh\"auser Boston, Boston, MA, 2006.

\bibitem{kontsevich2008stability}
Maxim Kontsevich and Yan Soibelman.
\newblock Stability structures, motivic {D}onaldson-{T}homas invariants and
  cluster transformations.
\newblock {\em arXiv preprint arXiv:0811.2435}, 2008.

\bibitem{MR3330788}
Maxim Kontsevich and Yan Soibelman.
\newblock Wall-crossing structures in {D}onaldson-{T}homas invariants,
  integrable systems and mirror symmetry.
\newblock In {\em Homological mirror symmetry and tropical geometry}, volume~15
  of {\em Lect. Notes Unione Mat. Ital.}, pages 197--308. Springer, Cham, 2014.

\bibitem{MR3033854}
Seung-Joo Lee, Zhao-Long Wang, and Piljin Yi.
\newblock B{PS} states, refined indices, and quiver invariants.
\newblock {\em J. High Energy Phys.}, (10):094, front matter + 42, 2012.

\bibitem{MR2967676}
Seung-Joo Lee, Zhao-Long Wang, and Piljin Yi.
\newblock Quiver invariants from intrinsic {H}iggs states.
\newblock {\em J. High Energy Phys.}, (7):169, front matter+18, 2012.

\bibitem{mandel2015scattering}
Travis Mandel.
\newblock Scattering diagrams, theta functions, and refined tropical curve
  counts.
\newblock {\em J. Lond. Math. Soc. (2)}, 104(5):2299--2334, 2021.

\bibitem{MR2888006}
Jan Manschot.
\newblock Wall-crossing of {D}4-branes using flow trees.
\newblock {\em Adv. Theor. Math. Phys.}, 15(1):1--42, 2011.

\bibitem{MR2875965}
Jan Manschot, Boris Pioline, and Ashoke Sen.
\newblock Wall crossing from {B}oltzmann black hole halos.
\newblock {\em J. High Energy Phys.}, (7):059, 73, 2011.

\bibitem{MR3036499}
Jan Manschot, Boris Pioline, and Ashoke Sen.
\newblock From black holes to quivers.
\newblock {\em J. High Energy Phys.}, (11):023, front matter + 53, 2012.

\bibitem{MR3080495}
Jan Manschot, Boris Pioline, and Ashoke Sen.
\newblock On the {C}oulomb and {H}iggs branch formulae for multi-centered black
  holes and quiver invariants.
\newblock {\em J. High Energy Phys.}, (5):166, front matter+42, 2013.

\bibitem{manschot2014generalized}
Jan Manschot, Boris Pioline, and Ashoke Sen.
\newblock Generalized quiver mutations and single-centered indices.
\newblock {\em Journal of High Energy Physics}, 2014(1):50, 2014.

\bibitem{MR4000572}
Sven Meinhardt and Markus Reineke.
\newblock Donaldson-{T}homas invariants versus intersection cohomology of
  quiver moduli.
\newblock {\em J. Reine Angew. Math.}, 754:143--178, 2019.

\bibitem{mou2019scattering}
Lang Mou.
\newblock Scattering diagrams of quivers with potentials and mutations.
\newblock {\em arXiv preprint arXiv:1910.13714}, 2019.

\bibitem{mozgovoy2021operadic}
Sergey Mozgovoy.
\newblock Operadic approach to wall-crossing.
\newblock {\em J. Algebra}, 596:53--88, 2022.

\bibitem{mozgovoy2020attractor}
Sergey Mozgovoy and Boris Pioline.
\newblock Attractor invariants, brane tilings and crystals.
\newblock {\em arXiv preprint arXiv:2012.14358}, 2020.

\bibitem{MR2259922}
Takeo Nishinou and Bernd Siebert.
\newblock Toric degenerations of toric varieties and tropical curves.
\newblock {\em Duke Math. J.}, 135(1):1--51, 2006.

\bibitem{piolinem}
Boris Pioline.
\newblock Mathematica {P}ackage {C}oulomb{H}iggs.
\newblock {\em
  \href{https://www.lpthe.jussieu.fr/~pioline/computing.html}{https://www.lpthe.jussieu.fr/~pioline/computing.html}},
  2020.

\bibitem{MR2650811}
Markus Reineke.
\newblock Poisson automorphisms and quiver moduli.
\newblock {\em J. Inst. Math. Jussieu}, 9(3):653--667, 2010.

\bibitem{MR2801406}
Markus Reineke.
\newblock Cohomology of quiver moduli, functional equations, and integrality of
  {D}onaldson-{T}homas type invariants.
\newblock {\em Compos. Math.}, 147(3):943--964, 2011.

\bibitem{MR1047415}
Morihiko Saito.
\newblock Mixed {H}odge modules.
\newblock {\em Publ. Res. Inst. Math. Sci.}, 26(2):221--333, 1990.

\bibitem{MR1402863}
Andrew Strominger.
\newblock Macroscopic entropy of {$N=2$} extremal black holes.
\newblock {\em Phys. Lett. B}, 383(1):39--43, 1996.

\bibitem{MR1818182}
Richard Thomas.
\newblock A holomorphic {C}asson invariant for {C}alabi-{Y}au 3-folds, and
  bundles on {$K3$} fibrations.
\newblock {\em J. Differential Geom.}, 54(2):367--438, 2000.

\bibitem{MR1957663}
Richard Thomas and Shing-Tung Yau.
\newblock Special {L}agrangians, stable bundles and mean curvature flow.
\newblock {\em Comm. Anal. Geom.}, 10(5):1075--1113, 2002.

\end{thebibliography}

\end{document}